\documentclass[11pt,a4paper]{article}
\usepackage[T1]{fontenc}
\usepackage[utf8]{inputenc}
\usepackage{lmodern}
\usepackage[margin=1in]{geometry}
\usepackage{amsmath,amssymb,amsthm,mathtools}
\usepackage{longtable, color}
\usepackage{pdflscape}
\usepackage{tikz}
\usepackage{comment}
\usepackage{adjustbox}
\usetikzlibrary{arrows.meta,backgrounds}
\usepackage{microtype,enumitem,booktabs,needspace}
\usepackage[hidelinks]{hyperref}
\allowdisplaybreaks[2]
\setlist{leftmargin=*,itemsep=3pt,topsep=5pt}
\newcommand{\R}{\mathbb R}
\newcommand{\E}{\mathbb E}
\newcommand{\Prob}{\mathbb P}
\newcommand{\ind}{\mathbf1}
\newcommand{\op}{\mathrm{op}}
\newcommand{\dist}{\operatorname{dist}}
\newcommand{\diag}{\operatorname{diag}}
\newcommand{\LO}{\mathrm{LO}}
\newcommand{\ALO}{\mathrm{ALO}}
\newcommand{\cl}{\operatorname{cl}}
\newcommand{\vol}{\operatorname{vol}}

\newcommand{\supp}{\operatorname{supp}}

\newtheorem{theorem}{Theorem}[section]
\newtheorem{lemma}[theorem]{Lemma}
\newtheorem{example}{Example}
\newtheorem{proposition}[theorem]{Proposition}
\newtheorem{corollary}[theorem]{Corollary}
\theoremstyle{definition}
\newtheorem{assumption}[theorem]{Assumption}

\theoremstyle{remark}
\newtheorem{remark}{Remark}[section]
\numberwithin{equation}{section}
\newcommand{\wz}[1]{%
  {\color{blue!70!black}[WZ: #1]}%
}

\hypersetup{
  pdftitle={Approximate Leave-One-Out on the Active Set: Error Bounds When the Support Changes},
  pdfauthor={},
  pdfsubject={Finite-third-moment comparison of exact restricted and deleted fits; V3-based manuscript}
}
\title{On the Accuracy of Approximate Leave-One-Out with Blockwise Nonsmooth Regularization in High-dimensions}
\author{
Honghao Lin$^{1}\thanks{Authors are listed in alphabetical order.}$,
Songbin Liu$^{2}$,
Arian Maleki$^{1,2, 3}$,
Vahab Mirrokni$^{1}$,
Chen Shang$^{3}$,
\\[0.4em]
David P. Woodruff$^{1,4}$,
Hanbo Yang$^{3}$,
Wenjin Zhang$^{3}$
}
\date{October 2026}

\begin{document}
\maketitle
\footnotetext[1]{Google.}
\footnotetext[2]{Department of Electrical Engineering, Columbia University.}
\footnotetext[3]{Department of Statistics, Columbia University.}
\footnotetext[4]{Computer Science Department, Carnegie Mellon University.}

\begin{abstract}
Approximate leave-one-out (ALO) refers to methods that provide computationally efficient and statistically accurate approximations to leave-one-out risk estimation. These methods have recently attracted increased attention because of their connections to other important problems, including machine unlearning and data valuation. For empirical risk minimization, ALO methods often rely on a few Newton steps and therefore require smooth loss functions and regularizers. In high-dimensional problems, however, regularizers are frequently non-smooth. This has motivated the development of heuristic ALO methods for non-smooth regularization, whose empirical accuracy is often strong but whose theoretical justification remains limited.

In this paper, we study the accuracy of a heuristic ALO method for a broad class of non-differentiable penalties that promote block sparsity in the proportional high-dimensional regime. Assuming uniformly bounded block sizes, we prove that the expected average absolute discrepancy between the exact leave-one-out loss and its ALO approximation is $O(\log^{3/2}(en)/\sqrt{n})$. Our proof technique differs substantially from existing approaches and relies on obtaining sharp bounds on the number of active groups whose status changes between the full-data and leave-one-out estimators. In addition to yielding stronger guarantees under weaker assumptions than prior work, our analysis provides new insight into the stability of sparse estimators under perturbations of the training data.

\end{abstract}

\noindent\textbf{Keywords:} approximate leave-one-out; active-set changes; block penalties; finite moments; proximal maps; anti-concentration.

\clearpage
\begingroup
\small
\tableofcontents
\endgroup
\clearpage

\section{Introduction}\label{sec:intro}



\subsection{Main challenge and summary of contributions}
Reliable estimation of out-of-sample prediction error is central to model assessment, model selection, and hyperparameter tuning. Leave-one-out cross-validation (LO) assesses predictive performance by fitting the model to all but one observation and evaluating it on the omitted observation, repeating this procedure across the sample.
LO has long been used for predictive assessment and model choice \cite{allen1974relationship,stone1974cross, geisser1975predictive}, and its statistical accuracy has been studied in both low-dimensional \cite{li1987asymptotic, andrews1991asymptotic, shao1993linear, aghbalou2024sharp, austern2025asymptotics} and high-dimensional settings \cite{rad2020error,zou2025theoretical, patil2021uniform}.

When the estimator is obtained through empirical risk minimization, direct computation of the leave-one-out (LO) estimate requires solving many closely related optimization problems, one for each observation removed from the dataset, and can therefore be computationally prohibitive. Approximate leave-one-out (ALO) methods aim to reduce this cost by using the full-data estimator to approximate the effect of deleting each observation, thereby avoiding the need to refit the model many times. These methods have attracted increasing attention in recent years, not only as computationally efficient tools for risk estimation and model assessment, but also because of their close connections to several emerging problems, including machine unlearning \cite{guo2019certified, zou2025certified, pandey2026gaussian} and data valuation problems \cite{koh2017understanding, zou2025newfluence, rubinstein2026rescaled, ilyas2025magic}.

For smooth objectives, ALO methods typically use one or a few Newton steps \cite{rad2020scalable, zou2025certified} to approximate the leave-one-out estimates. By exploiting the close relationship among the resulting linear systems, simple linear algebraic identities can then be used to avoid refitting the model for each observation, leading to an efficient approximation of the leave-one-out risk. The accuracy of these methods has been studied extensively over the past decade in both low- and high-dimensional settings \cite{beirami2017optimal,giordano2019swiss,rad2020scalable, wilson2020approximate, ghosh2020approximate, austin2025distributional}.

Despite this progress, the non-differentiable setting is arguably more relevant in many modern high-dimensional problems, where sparsity-inducing penalties such as the LASSO, group LASSO, and related structured regularizers are widely used \cite{tibshirani1996regression,yuan2006model,
zou2005regularization}. In these settings, the lack of smoothness prevents a direct application of the standard Newton-based ALO arguments. Although several heuristic approximations have been proposed and shown empirically to perform well \cite{wang2018approximate, beirami2017optimal}, their theoretical accuracy is much less understood. This leaves an important gap between the practical use of ALO with non-smooth regularizers and the available theory.

In this paper, we theoretically study the accuracy of a heuristic ALO method, proposed in \cite{wang2018approximate} for a broad class of non-differentiable penalties that promote block sparsity. The key heuristic underlying these methods is that the sparsity pattern of the full-data estimator remains unchanged after removing a single observation. Under this assumption, the leave-one-out optimization problem can be restricted to the active groups of the full-data solution, yielding a smooth reduced problem to which Newton-based approximations can be applied efficiently. Empirically, these methods have been shown to provide highly accurate approximations to LO \cite{wang2018approximate, auddy2024approximate}. At the same time, numerical studies indicate that the underlying sparsity-pattern assumption is generally not exactly valid: the active sets of the full-data and leave-one-out estimators can differ. Thus, the strong empirical performance of these methods is somewhat surprising and raises a fundamental theoretical question: 
\begin{itemize}

\item Why can an approximation based on an incorrect active-set assumption still accurately estimate the leave-one-out risk?

\item How much error do these inaccurate assumptions introduce into the resulting approximations?

\end{itemize}

In this paper, we study the theoretical accuracy of heuristic ALO methods in the proportional high-dimensional regime, where the number of samples $n$ is at the same order as the number of features $p$. At a high level, our main contributions can be summarized as follows. A more detailed discussion of the contributions is presented in Section \ref{sec:setting}.

\begin{itemize}

\item We show that
\[
\E\big[|{\rm LO}-{\rm ALO}|\big]\leq C\frac{\log^{3/2}(en)}{\sqrt{n}}.
\]
Recall that LO itself is an estimator of the out-of-sample prediction error. Since it is based on \(n\) observations, its estimation error is generally of order \(O(1/\sqrt{n})\) without additional structural assumptions. Thus, up to the logarithmic factor, the error introduced by replacing LO with ALO is of the same order as the intrinsic statistical error of LO in estimating the out-of-sample prediction risk. In this sense, ALO achieves nearly the same statistical accuracy as LO while avoiding its substantial computational cost.

\item We establish an $O_p\!\left(\sqrt{n}\log^{3/2}(n)\right)$ bound on the discrepancy between the active sets of the full-data and leave-one-out estimators. This substantially improves upon the previously known bound of $O_p\!\left(n^{11/12}\log^{1/4}(n)\right)$ obtained in \cite{auddy2024approximate}. Beyond their role in establishing the accuracy of ALO, such active-set stability bounds may also be useful for analyzing ensemble and resampling-based procedures, including bagging \cite{breiman1996bagging} and stability selection \cite{meinshausen2010stability}, where sensitivity of the selected variables or groups to perturbations of the training data is a central issue.

\end{itemize}

Although our primary focus is on the accuracy of ALO as an approximation to LO, as we discussed before, the techniques developed here for controlling the discrepancy between the full-data and leave-one-out solutions can be adapted to study the accuracy of fast approximation methods in machine unlearning and data valuation. Thus, beyond their immediate implications for ALO, our results provide a broader framework for analyzing data-deletion and sample (group)-influence problems.

\subsection{Related work}
\label{sec:related}

Developing ALO methods for non-smooth losses and regularizers has been an active area of research over the past decade. Some of the earliest approaches in high-dimensional settings, used the approximate message passing (AMP) framework \cite{donoho2009message} to derive efficient leave-one-out approximations, including for the LASSO problem \cite{obuchi2016cross, mousavi2018consistent}. However, the theoretical guarantees for AMP-based methods typically rely on restrictive assumptions, such as independence of features, which may not hold in many practical settings. Moreover, extending these techniques beyond specific losses and regularizers is often nontrivial.

These limitations motivated the development of alternative approaches for deriving ALO formulas in non-smooth settings \cite{rad2020scalable, beirami2017optimal, wang2018approximate}. In particular, \cite{wang2018approximate} proposed several heuristic approximations based on primal, dual, and proximal representations of the estimator. These methods substantially broaden the class of models for which ALO can be computed efficiently. Despite their strong empirical performance, however, the theoretical accuracy of these approximations remains poorly understood.

Among the existing works, \cite{auddy2024approximate,bellec2025simultaneous} are the two most closely related to ours, as both study the theoretical accuracy of ALO in proportional high-dimensional regimes and allow for non-differentiable regularizers. In particular, \cite{auddy2024approximate} derives high-probability bounds for computational ALO under $\ell_1$ regularization, with the approximation error controlled by the number of coordinates whose active-set status may change after removing an observation. For squared-loss elastic-net regression, their explicit control of this instability relies on Gaussian noise, isotropic Gaussian designs, and additional regularization conditions. Combining their main results, and taking the regularization parameters to be fixed positive constants with $p\asymp n$, yields an ALO--LO error bound of order $n^{-1/24}$, up to logarithmic factors. In our results substantially sharpen this rate, improving the polynomial exponent from $1/24$ to $1/2$, up to logarithmic factors, while also providing an unconditional bound on the expected absolute ALO--LO error.

More importantly, our results apply to
regression settings beyond the explicit
guarantees established in that work.
In particular, we permit correlation among the features, deterministic weighted penalties,
and non-Gaussian responses with only a finite
third moment of the null score. The Gaussian
noise assumption used to establish the
explicit active-set perturbation bound of
Auddy et al. is not required in our analysis.

Bellec~\cite{bellec2025simultaneous} develops
another approach to nonsmooth ALO through
its connection with mean-field inference.
This approach yields near-root-$n$
high-probability ALO--LO guarantees for
nonsmooth regularizers, including elastic
net and group lasso. However, there are a few major differences between this work and the result of ~\cite{bellec2025simultaneous}:

\begin{enumerate}
\item Our proof strategy differs substantially from that of \cite{bellec2025simultaneous}. We develop a quantitative analysis of active-set instability using blockwise proximal geometry together with conditional anti-concentration. This allows us to bound the number of coordinates or groups whose activity status can change after removing a single observation. We then quantify how these changes propagate into the ALO approximation error. Beyond the specific class of block-separable penalties considered here, this proof strategy may provide a useful framework for studying other non-smooth regularizers, including spectral penalties such as the nuclear norm. Furthermore, such results  may be useful for analyzing ensemble and resampling-based procedures.

\item The ALO--LO guarantees for robust regression and single-index models in Corollaries 4.1 and 5.1 of \cite{bellec2025simultaneous} rely on Lipschitz conditions on the loss function and its derivatives. Although that work also considers squared loss separately, the corresponding leverage-concentration result in Proposition 2.1 assumes Gaussian regression noise. In contrast, the framework developed in this paper accommodates a broader class of loss functions and substantially more general noise distributions and tail behavior, thereby extending the theoretical scope of ALO analysis beyond these settings.

\end{enumerate}

\subsection{Notation}\label{sec:intro-notation}

For a vector $x\in\R^p$, let $\|x\|$ denote its Euclidean norm. For a matrix $X\in\R^{n\times p}$, let $\|X\|_{\op}$ and $\|X\|_F$ denote its operator and Frobenius norms, respectively. We write the $j$th row of $X$ as $x_j^\top$.

For positive deterministic sequences, $a_n=O(b_n)$ means that $a_n/b_n$ is bounded, and $a_n\asymp b_n$ means that this ratio is bounded above and away from zero. We write $Z_n=O_p(a_n)$ when $Z_n/a_n$ is bounded in probability.

For a finite set $I$, $|I|$ denotes its cardinality. The order $A\preceq A'$ between symmetric matrices is the Loewner order, and $\mathbf{1}_E$ denotes the indicator of an event $E$.


\subsection{Organization}\label{sec:organization}

The rest of the paper is organized as follows.
Section~\ref{sec:setting} introduces the model, assumptions, and main results.
Section~\ref{sec:proof-of-thm2.5} outlines the main steps in the proof of Theorem~\ref{thm:main}, one of the principal results of the paper.
Section~\ref{sec:proof:thm2.8} establishes a stronger version of Theorem~\ref{thm:main} under stronger tail assumptions on the scores.
Section~\ref{sec:discussion} discusses the assumptions that are essential to the arguments developed in the paper and highlights some limitations of the results.
Detailed proofs of all results are deferred to the appendices.

\section{Model, assumptions, and main results}\label{sec:setting}

\subsection{Data generation model and regularizer}\label{sec:model}

 Consider a setting in which we observe independent and identically distributed data points
\[
(y_j,x_j)\sim q(y_j\mid x_j^\top \beta^*)\,p(x_j), \qquad j\in\{1,\ldots,n\},
\]
where \(y_j\in\mathbb{R}\) denotes the response, \(x_j\in\mathbb{R}^p\) denotes the feature vector, and \(\beta^* \in \mathbb{R}^p\) denotes the true model parameter. A widely used approach for estimating \(\beta^*\) from data is the regularized empirical risk minimization (R-ERM), which obtains an estimate of \(\beta^*\) by minimizing an objective of the following form:
\begin{equation}\label{eq:objectives}
 \hat{\beta} = \arg\min_\beta F(\beta)= \arg\min_{\beta} L(\beta)+\lambda_1 r(\beta)+\frac\mu2\|\beta\|^2.
\end{equation}
In this equation, $L(\beta)=\sum_{j=1}^n \ell(y_j;x_j^\top\beta),$
where \(\ell(y_j;x_j^\top\beta)\) denotes the loss function, while
\[
\lambda_1 r(\beta)+\frac{\mu}{2}\|\beta\|^2
\]
serves as the regularizer. The hyperparameters \(\lambda_1\) and \(\mu\) are typically tuned to achieve the best predictive performance on unseen data. A popular approach for estimating the predictive performance on unseen data is the so-called leave one out cross validation (LOOCV).  LOOCV is computed by repeatedly leaving out one observation, fitting the model using the remaining \(n-1\) samples, evaluating the fitted model on the held-out observation, and then averaging the resulting prediction losses over all \(n\) observations. More specifically, 
the leave-one-out risk estimate is defined as
\begin{equation}\label{eq:LO}
 \LO=\frac1n\sum_{i=1}^n\ell(y_i;x_i^\top\hat\beta^{(i)}),
\end{equation}
where
\begin{equation}\label{eq:LOO_prob}
\hat\beta^{(i)} =  \arg\min_\beta  F_i(\beta)= \arg\min_\beta \underbrace{\sum_{j\neq i}\ell(y_j;x_j^\top\beta)}_{L_i(\beta)}+\lambda_1 r(\beta)+\frac\mu2\|\beta\|^2
 \end{equation}
One of the main challenges associated with LOOCV is its computational cost, since the model must be refit \(n\) times, once for each held-out observation. This has motivated substantial effort toward developing computationally efficient yet statistically accurate approximations to LOOCV. As discussed in the previous section, our focus is on understanding the accuracy of such approximations when the regularizer is non-differentiable. Throughout the paper, we consider the following broad class of non-differentiable regularizers, which includes many commonly used penalties in high-dimensional statistics.

\begin{assumption}[Block penalty]\label{ass:blocks}
The coordinates of $\R^p$ are partitioned into blocks $g=1,\dots,G$ of sizes $1\le d_g\le d_0$, and we write $\beta=(\beta_g)_{g=1}^G$ with $\beta_g\in\R^{d_g}$. The partition does not depend on the data. The penalty is $r(\beta)=\sum_{g=1}^Gr_g(\beta_g)$, where each $r_g:\R^{d_g}\to[0,\infty)$ is convex, satisfies $r_g(0)=0$, is three times continuously differentiable on $\R^{d_g}\setminus\{0\}$, and satisfies
\[
 B(0,\rho_-)\subseteq\partial r_g(0)\subseteq B(0,\rho_+)
\]
for constants $0<\rho_-\le\rho_+$, where $\partial r_g(0)$ denotes the subdifferential of $r_g(\cdot)$ at zero, and $B(0,r)$ denotes the closed ball centered at zero with radius $r$ in the corresponding Euclidean space. Let
\begin{equation}\label{eq:prox-def}
 K_g=\lambda_1\,\partial r_g(0),\qquad
 \Phi_g=(\mu I+\lambda_1\,\partial r_g)^{-1}.
\end{equation}
The map $\Phi_g$ is twice continuously differentiable
on $\R^{d_g}\setminus K_g$, and
\begin{equation}
    \sup_{g\le G}\ \sup_{v\notin K_g}\|D^2\Phi_g(v)\|\le C_\Phi.
    \label{eq:bound_on_d2_phi_g}
\end{equation}
Here, $D^2\Phi_g(v):\R^{d_g}\times\R^{d_g}\to\R^{d_g}$
is viewed as a bilinear map, with its norm defined as
\begin{equation*}
\|D^2\Phi_g(v)\|
=
\sup_{\substack{h,h'\in\R^{d_g}\\ \|h\|=\|h'\|=1}}
\bigl\|D^2\Phi_g(v)[h,h']\bigr\|,
\end{equation*}
where all vector norms are Euclidean.
The functions $r_g$ may depend on $n$ and $p$, provided these conditions hold with the same constants.
\end{assumption}

Assumption \ref{ass:blocks} covers several regularizers that are widely used in high-dimensional statistics, including the weighted LASSO, group LASSO, and blockwise elastic-net penalties; this is formalized in Lemma~\ref{lem:prox}. As a first step, in the next section we review heuristic approaches for deriving ALO formulas in the presence of non-smooth regularizers. Our discussion is motivated in part by the framework developed in \cite{wang2018approximate}.

\subsection{Approximate Leave One Out (ALO) }\label{sec:restricted}

In this section, we show how heuristic arguments can be used to derive an approximation to LO. When both the loss function and the regularizer are differentiable, one can use $\hat{\beta}$, the solution to the full-data problem, as a warm start for Newton's method applied to \eqref{eq:LOO_prob} . Since $\hat{\beta}$ is expected to be close to $\hat{\beta}^{(i)}$, a single Newton step starting from $\hat{\beta}$ can provide a computationally efficient approximation to $\hat{\beta}^{(i)}$ \cite{rad2020scalable}. The accuracy of this one-step Newton approximation has been theoretically established in both low-dimensional and high-dimensional settings under suitable smoothness assumptions on the loss function and the regularizer \cite{rad2020scalable}. However, this approach is not directly applicable when either the loss function or the regularizer is non-smooth. Hence, inspired by \cite{wang2018approximate} we suggest an extra approximation for the regularizers we considered in Assumption \ref{ass:blocks}.  

For a vector $\beta$, let $S(\beta)=\{g:\beta_g\ne0\}$. We can consider these groups as ``active groups''. Define the full-data product stratum
\begin{equation}\label{eq:expanded-display-006}
\mathcal M^\circ(b)=\prod_g M_g(b_g),
\end{equation}
where for each block $g$ and every $u\in\mathbb R^{d_g}$, $M_g(u)$ is defined as
\begin{equation*}
    M_g(u)=
\begin{cases}
\{0\}, & u=0,\\
(0,\infty), & d_g=1,\ u>0,\\
(-\infty,0), & d_g=1,\ u<0,\\
\mathbb R^{d_g}\setminus\{0\}, & d_g\ge2,\ u\ne0.
\end{cases}
\end{equation*}
and $\prod_g$ denotes a Cartesian product of sets. As discussed in Assumption \ref{ass:blocks}, the main challenge in applying approximation methods that rely on the smoothness of the loss function and the regularizer arises from the non-differentiability of the regularizer at groups for which $\hat{\beta}_g=0$. A natural heuristic is therefore to assume that the active groups of $\hat{\beta}$ and $\hat{\beta}^{(i)}$ remain unchanged. Under this assumption, one can replace \eqref{eq:LOO_prob} with a simpler optimization problem restricted to an appropriate active region. To describe these restricted problems more clearly, define
\begin{equation}\label{eq:domains}
 A(b)=\{\beta:\beta_g=0\text{ for }g\notin S(b)\},
 \qquad C(b)=\cl\mathcal M^\circ(b),
\end{equation}
where $\cl$ denotes the closure operation. Both $A(b)$ and $C(b)$ are closed convex sets containing $0$ and $b$. For $D=A$ or $D=C$, define
\begin{equation}\label{eq:restricted}
 m_i^D=\arg\min_{\beta\in D(\hat{\beta})}F_i(\beta).
\end{equation}
Note that both of these restricted optimization problems are different from the original leave-one-out problem \eqref{eq:LOO_prob}, whose solution is required for computing the exact leave-one-out risk estimate. However, when $F_i$ is restricted to $C(\hat{\beta})$, the objective is differentiable on the relevant active region. This makes it possible to apply a one-step Newton approximation to obtain a computationally efficient approximation to the solution of the restricted problem.

\begin{remark}
To obtain the final ALO formula for \eqref{eq:objectives}, one would further approximate the solution of \eqref{eq:restricted} by taking a single Newton step initialized at $\hat{\beta}$  \cite{wang2018approximate}. Since the analysis of this one-step Newton approximation is essentially the same as that developed in \cite{rad2020scalable}, we do not repeat it here and instead focus on the additional error introduced by restricting the leave-one-out problem.
\end{remark}

The discrepancy between problems \eqref{eq:LOO_prob} and \eqref{eq:restricted} raises the following questions: 
\begin{itemize}

\item[] \textbf{Q.1}: How close is the solution of \eqref{eq:restricted} to $\hat{\beta}^{(i)}$? Intuitively, when \(n\) is large, removing a single observation should have only a limited effect on the fitted model, suggesting that the set of active groups may remain unchanged, i.e.,
\[
S(\hat{\beta})=S(\hat{\beta}^{(i)}).
\]
If this active-set stability were to hold, with scalar signs also preserved, one would expect the restricted solution \(m_i^D\) to coincide with the exact leave-one-out estimator $\hat{\beta}^{(i)}$.

\item[] \textbf{Q.2}: If we use these two new approximations for obtaining the following ALO:
\begin{equation}\label{eq:ALO}
  \widetilde\LO^D=\frac1n\sum_i\ell(y_i;x_i^\top m_i^D),
\end{equation}
then how large can $|\LO-\widetilde\LO^D|$ be? 

\end{itemize}

To address these questions we would like to consider the high-dimensional analysis setting that we lay out carefully in the next section.

\subsection{High-dimensional analysis framework and other assumptions}
Following a  long line of work in the field of high-dimensional statistics, in this paper we adopt the ``proportional asymptotic regime" for studying \textbf{Q.1} and \textbf{Q.2}. The main assumption in the proportional asymptotic regime is the following:

\begin{assumption}[proportional asymptotic setting]\label{ass:design}
There are fixed positive constants $c_\Sigma,C_\Sigma,\delta_-,\delta_+,B$ such that
\begin{equation}\label{eq:expanded-display-002}
x_j\sim N(0,\Sigma),\qquad
 \frac{c_\Sigma}{p}I\preceq\Sigma\preceq\frac{C_\Sigma}{p}I,
 \qquad \delta_-\le n/p\le\delta_+,
 \qquad \|\beta^*\|^2\le B^2p.
\end{equation}
Note that in this assumption we let both $p$ and $n$ be large, while the ratio of the two can not go to infinity or zero. 
\end{assumption}

Our next assumption is concerned with the hyperparameters of the model. 

\begin{assumption}\label{ass:hperparam}
The hyperparameters $\lambda_1,\mu$ are deterministic and lie in fixed compact subintervals of $(0,\infty)$.
\end{assumption}

As discussed in \cite{zou2025theoretical}, this is not a restrictive assumption, since the optimal choice of these hyperparameters under the proportional asymptotic setting, are neither zero nor infinity, except in a few pathological settings, e.g. linear regression with no noise, or infinite additive noise.

In this paper, our focus is exclusively on the challenges arising from non-smooth regularization. To isolate these effects and simplify the analysis, we assume throughout that the loss function is smooth.

\begin{assumption}[Smooth loss and finite third score moment]\label{ass:loss}
The loss function $\ell(\cdot;\cdot)$ is finite and jointly measurable. Primes denote derivatives with respect to its second argument, i.e.,
\begin{equation*}
\ell'(y;z):=\frac{\partial}{\partial z}\ell(y;z),\qquad\ell''(y;z):=\frac{\partial^2}{\partial z^2}\ell(y;z).
\end{equation*}
For every response $y$ in its support, $\ell(y;\cdot)\in C^2(\R)$ is convex, and satisfies
\begin{equation}\label{eq:expanded-display-003}
0\le\ell''(y;z)\le L_2,\qquad
 |\ell''(y;z)-\ell''(y;z')|\le L_3|z-z'|.
\end{equation}
For the null scores $\eta_j=\ell'(y_j;0)$, fixed constants $K_3,v_->0$ satisfy
\begin{equation}\label{eq:thirdmoment}
 \E|\eta_j|^3\le K_3,
 \qquad \E\eta_j^2\ge v_-.
\end{equation}
We write $\eta=(\eta_1,\ldots,\eta_n)^\top$ and set $v_+=K_3^{2/3}$, so $\E\eta_j^2\le v_+$. 
\end{assumption}

\begin{assumption}[Likelihood representation and one-sided curvature]\label{ass:likelihood}
The conditional law of response $y$ given $u$ has a common support $\mathcal Y\subseteq\R$, and a fixed sigma-finite dominating measure. Their density or mass function $q(y\mid u)$ is jointly measurable. For almost every $y$, $q(y\mid u)$ is positive for every $u$ and $C^2$ in $u$, with
\begin{equation}\label{eq:likelihood}
 \partial_u^2\log q(y\mid u)\ge-C_q.
\end{equation}
The dominating measure may be counting measure. Note that we have allowed the loss to differ from the negative log-likelihood.
\end{assumption}

Many loss functions $\ell$ and conditional distributions $q(y | u)$ satisfy the above assumptions. Below we mention a couple of such examples.

\begin{example}
    For $y\in\{0,1\}$, take
    \[
     q(y\mid u)=\frac{e^{yu}}{1+e^u},\qquad
     \ell(y;z)=\log(1+e^z)-yz.
    \]
    With $\pi(z)=e^z/(1+e^z)$, direct differentiation gives
    \[
     \ell'(y;z)=\pi(z)-y,\quad
     \ell''(y;z)=\pi(z)(1-\pi(z))\le\tfrac14,\quad
     \partial_u^2\log q(y\mid u)=-\pi(u)(1-\pi(u))\ge-\tfrac14.
    \]
    The derivative of $\ell''$ has absolute value at most $1/4$, so its Lipschitz condition holds. The null score has $|\eta_i|=1/2$ and $\E\eta_i^2=1/4$. 
\end{example}

\begin{example}
Let $y_i=x_i^\top\beta^*+e_i$, where the errors are independent and identically distributed, are independent of the covariates, and have density proportional to $(1+e^2/\nu)^{-(\nu+1)/2}$ for a fixed $3<\nu<4$. Note that the error has much heavier tail than the usual Gaussian errors. Suppose that we still use squared loss $\ell(y;z)=(y-z)^2/2$ for estimating $\beta^*$. The null score is $\eta_i=-y_i$. The error density has tails of order $|e|^{-\nu-1}$, so the third absolute moment is finite whereas the fourth moment is infinite. The Gaussian signal has uniformly bounded variance under Assumption~\ref{ass:design}; consequently $\E|\eta_i|^3$ is uniformly bounded, and the independent nondegenerate error gives the required positive second-moment lower bound. Direct differentiation of the location likelihood yields
    \begin{equation}\label{eq:example-heavy-curvature}
     \partial_u^2\log q(y\mid u)
     =\frac{(\nu+1)((y-u)^2-\nu)}{(\nu+(y-u)^2)^2}
     \ge-\frac{\nu+1}{\nu}.
    \end{equation}
    The curvature is positive when $|y-u|>\sqrt\nu$, so the likelihood is not log-concave in $u$. Nevertheless it satisfies the one-sided condition, while squared loss has $L_2=1$ and $L_3=0$. 
    \end{example}

\subsection{Theoretical contributions}

\subsubsection{Main results}
\label{sec:main results}
In this section, we theoretically address Q.1 and Q.2 raised in Section \ref{sec:restricted}. For notational simplicity we replace $\hat{\beta}$ and $\hat\beta^{(i)}$ with $b$ and $b_i$ respectively. Let $P$ and $P_i$ be the block-coordinate orthogonal projectors onto $S(b)$ and $S(b_i)$. Locally at these fits they are the tangent projectors of the corresponding product strata. Define
\begin{equation}\label{eq:scales}
 h_n=\sqrt{\frac{\log(en)}n},\qquad
 \mathfrak r_n=\frac{\log^{3/2}(en)}{\sqrt n}.
\end{equation}
The following theorem presents the first theoretical result of this paper.
\begin{theorem}\label{thm:main}
Under Assumptions~\ref{ass:blocks}--  \ref{ass:likelihood}, there is a constant $C$ such that, for all sufficiently large $n,p$, for $D=A$ and $D=C$,
\begin{align}
 \sup_{i\le n}\E\big|\ell(y_i;x_i^\top b_i)-\ell(y_i;x_i^\top m_i^D)\big|
 &\le C\mathfrak r_n,\label{eq:mainloss}\\
 \sup_{i\le n}\E|x_i^\top\Delta_i^D|
 +\sup_{i\le n}\E\|\Delta_i^D\|^2
 &\le C\mathfrak r_n,\label{eq:mainprediction}\\
 \sup_{i\le n}\E\left[\frac{\|P-P_i\|_F^2}{p}\right]
 &\le Ch_n,\label{eq:mainprojector}
\end{align}
where $\Delta_i^D = b_i - m_i^D$. Consequently,
\begin{equation}\label{eq:mainrisk}
 \E|\LO-\widetilde\LO^D|
 \le\E\left[\frac1n\sum_i
 \big|\ell(y_i;x_i^\top b_i)-\ell(y_i;x_i^\top m_i^D)\big|\right]
 \le C\mathfrak r_n.
\end{equation}
Thus 
\begin{equation*}
\frac1n\sum_{i=1}^n\left|\ell(y_i;x_i^\top b_i)-\ell(y_i;x_i^\top m_i^D)\right|=O_p(\mathfrak r_n).
\end{equation*}
The projector bound \eqref{eq:mainprojector} also holds for the complementary normal projectors, i.e.,
\begin{equation*}
\sup_i\mathbb E\left[\frac{\|(I_p-P)-(I_p-P_i)\|_F^2}{p}\right]\le Ch_n.
\end{equation*}
\end{theorem}


The proof of this theorem is presented in Section \ref{sec:proof-of-thm2.5}. 
In the above theorem the supremum appears after the expected value of the differences. We can strengthen these results if we impose extra tail conditions. 

\begin{assumption}[Tail assumption for scores]\label{ass:tail:score}
Suppose that for fixed $K,\alpha>0$, we have
\begin{equation}\label{eq:weibull}
 \Prob(|\eta_j|>u)\le 2\exp\{-(u/K)^\alpha\},\qquad u\ge0.
\end{equation}
\end{assumption}

Set $a=\max\{1/2,1/\alpha\}$ and $M_n=C_M\log^a(en)$, with a sufficiently large fixed $C_M$. If instead the loss has a globally bounded derivative, one may take $M_n$ constant. Define $t_n=D_0M_nh_n$, where $D_0$ is the fixed constant chosen in Section~\ref{sec:local}. Based on these extra assumptions, we can strengthen the result. 

\begin{theorem}\label{thm:uniform}
Under Assumptions~\ref{ass:blocks}--  \ref{ass:likelihood} and \ref{ass:tail:score}, for all sufficiently large $n,p$ and for $D=A,C$, we have
\begin{align}
 \E\max_{i\le n}|x_i^\top\Delta_i^D|
 &\le C h_nt_n(pt_n+M_n^2)+Cn^{-4},\label{eq:maxpred}\\
 \E\max_{i\le n}|\ell(y_i;x_i^\top b_i)-\ell(y_i;x_i^\top m_i^D)|
 &\le C t_n^2(pt_n+M_n^2)+Cn^{-4}.\label{eq:maxloss}
\end{align}
In particular, the right sides have orders
\begin{equation}\label{eq:expanded-display-015}
O\!\left(M_n^2\frac{\log^{3/2}(en)}{\sqrt n}\right),
 \qquad
 O\!\left(M_n^3\frac{\log^{3/2}(en)}{\sqrt n}\right),
\end{equation}
respectively. Moreover,
\begin{align}
 \max_i\|\Delta_i^D\|
 &=O_p\!\left(M_n^{3/2}\frac{\log^{3/4}(en)}{n^{1/4}}\right),\label{eq:maxvector}\\
 \max_i\frac{\|P-P_i\|_F^2}{p}&=O_p(M_nh_n).\label{eq:maxproj}
\end{align}
The same statements hold with constant $M_n$ when the score is globally bounded. 
\end{theorem}

The proof is given in Section~\ref{sec:proof:thm2.8}.

\section{Proof of Theorem~\ref{thm:main}}\label{sec:proof-of-thm2.5}

The proof combines deterministic comparison bounds with probabilistic control of blocks near activation boundaries.
Section~\ref{sec:local} bounds the prediction and loss discrepancies on suitable regularity events in terms of the held-out score and a boundary count at a score-dependent width.
Section~\ref{sec:integration} combines fixed-width boundary-count estimates with conditional bounds for the held-out score to obtain expectation bounds on suitable regularity events.
Finally, Section~\ref{sec:integration2} controls the expected errors on the remaining exceptional events and combines the resulting estimates with the preceding bounds to establish all the assertions of the theorem.

\subsection{Existence, measurability, and stability} We first establish that the estimators are well defined and measurable, and record basic stability bounds needed for the comparisons and conditional arguments in the proof.
\begin{lemma}[Existence, measurability, and elementary stability]
\label{lem:stability}

The minimizers $b$, $b_i$, $m_i^A$, and $m_i^C$, as well as the minimizers of $F$ and $F_i$ over $\{\beta\in\mathbb R^p:\beta_g=0\}$ for each deterministic block $g\in\{1,\ldots,G\}$, exist uniquely and are measurable.

For the following coefficient and score bounds,
let $T\subseteq\{1,\ldots,n\}$ be a deterministic
set of retained observations, and let
$\mathcal I\subseteq\{1,\ldots,G\}$ be a deterministic
set of free blocks.
Write
$d_{\mathcal I}=\sum_{g\in\mathcal I}d_g$,
$Z=X_{T,\mathcal I}$,
where $X_{T,\mathcal I}$ is the design restricted
to the retained observations and free blocks.
Let $z_j^\top$ denote its row corresponding
to observation $j\in T$.
Define the restricted penalty
$r_{\mathcal I}(u)
=\sum_{g\in\mathcal I}r_g(u_g)$
and the corresponding estimator
\[
a=\arg\min_{u\in\R^{d_{\mathcal I}}}
\left\{
\sum_{j\in T}\ell(y_j;z_j^\top u)
+\lambda_1r_{\mathcal I}(u)
+\frac{\mu}{2}\|u\|^2
\right\}.
\]
Let
$\eta=(\ell'(y_j;0))_{j\in T},$ $
\dot\ell(a)=(\ell'(y_j;z_j^\top a))_{j\in T}.$
Then
\begin{equation}\label{eq:normbound}
\mu\|a\|\le\|Z\|_{\op}\|\eta\|.
\end{equation}
Moreover,
\begin{equation}\label{eq:scoresandwich}
\frac{\mu}{\mu+L_2\|Z\|_{\op}^2}\|\eta\|
\le\|\dot\ell(a)\|
\le
\left(1+\frac{L_2\|Z\|_{\op}^2}{\mu}\right)\|\eta\|.
\end{equation}

The estimator $a$ is measurable with respect to
$\sigma((y_j,x_{j,\mathcal I}):j\in T)$.

Finally, for the full-data and deleted estimators,
with $c_i=\ell'(y_i;x_i^\top b)$,
\begin{equation}\label{eq:scorecontraction}
\|b-b_i\|\le\mu^{-1}|c_i|\|x_i\|,
\qquad
|c_i|\le|\ell'(y_i;x_i^\top b_i)|.
\end{equation}

\end{lemma}

\begin{remark}
The generic estimator $a$ covers both ambient and
zero-clamped optimization problems.
For the ambient deleted estimator, take
$T=\{1,\ldots,n\}\setminus\{i\}$ and let all
blocks be free. Then $a=b_i$ and $Z=X_i$.
In particular,~\eqref{eq:normbound} gives
$\|b_i\|
\le \mu^{-1}\|X_i\|_{\op}\|\eta_{-i}\|.$
This bound is used in Lemmas~\ref{lem:gaussian}
and~\ref{lem:baddeleted}, and consequently
in the probabilistic estimates of
Section~\ref{sec:integration}.
The deletion-stability and score-contraction
bounds in~\eqref{eq:scorecontraction} are also
used in Lemmas~\ref{lem:global}
and~\ref{lem:tilt}.

For the zero-clamped estimator, fix a block
$g$, impose $\beta_g=0$, and minimize the
deleted objective over all remaining blocks.
Taking
$T=\{1,\ldots,n\}\setminus\{i\}$,
$\mathcal I=\{1,\ldots,G\}\setminus\{g\},$
gives
$a=\breve b_{i,-g}^{(-g)},$
$Z=X_{i,-g}.$
Thus ``zero-clamped'' means that one
prescribed block is fixed at zero while
all other blocks are refitted.

\end{remark}

\subsection{Deterministic comparison bounds}\label{sec:local}

In this section, we establish deterministic bounds on the coefficient, prediction and loss differences between the exact LO estimator $b_i$ and the restricted estimators $m_i^A$ and $m_i^C$, together with a bound on the projector discrepancy associated with $b$ and $b_i$. 

A key step in analyzing the LO approximation error is to compare the minimizers defining the exact and restricted LO estimates. We explain the main ideas behind our approach and introduce the necessary notation and regularity events as the argument develops. The resulting bounds will be used to establish the expectation bounds in the following subsections.

\paragraph{Reduction to a dual-displacement estimate.}
Fix a deletion index $i$ and a restricted domain $D\in\{A,C\}$, as defined in \eqref{eq:domains}. 
The first step is to study the difference between the unrestricted minimizer $b_i$  and the restricted minimizer $m_i^D$.
Since the domain $D(b)$ depends on the full-data  estimator $b$, directly comparing the KKT conditions  of $b_i$ and $m_i^D$ would introduce an additional  normal-cone term associated with the constraints. 
We therefore seek to reformulate this comparison in terms of quantities that are easier to control.

Since Assumptions~2.1, 2.3, and~2.4 ensure that $F_i$ is $\mu$-strongly convex, and $b_i$ is its unrestricted minimizer, we have
\[
\frac{\mu}{2}\|m_i^D-b_i\|^2
\le F_i(m_i^D)-F_i(b_i).
\]
Moreover, since $m_i^D$ minimizes $F_i$ over $D(b)$, any feasible point $a_i\in D(b)$ satisfies $F_i(m_i^D)\le F_i(a_i)$.
Thus, if we can construct such a point using $b$ and $b_i$, the preceding inequality gives
\[
\frac{\mu}{2}\|m_i^D-b_i\|^2 \le F_i(a_i)-F_i(b_i).
\]
This reduces the comparison between $m_i^D$ and $b_i$ to controlling the objective increase $F_i(a_i)-F_i(b_i)$ for a suitably constructed point $a_i$.

To construct such a point, note that the restrictions in $D(b)$ are determined by the strata selected by $b$. We therefore identify the blocks whose strata differ between $b$ and $b_i$ and define the stratum-change set as
\begin{equation}\label{eq:stratum-change-set}
\mathcal S_i
=
\{g:M_g(b_g)\ne M_g(b_{i,g})\},
\end{equation}
where $M_g(\cdot)$ is defined in Section \ref{sec:restricted}. 
We then construct the hybrid point $a_i$ by retaining the deleted coefficients on blocks whose strata are unchanged and using the full-data coefficients on the remaining blocks:
\begin{equation}\label{eq:hybrid-introduction}
a_{i,g}
=
\begin{cases}
b_{i,g}, & g\notin\mathcal S_i,\\
b_g,     & g\in\mathcal S_i.
\end{cases}
\end{equation}
By construction, $a_i\in\mathcal M^\circ(b)\subseteq D(b)$, while $a_i$ differs from $b_i$ only on $\mathcal S_i$. To use this property to control the objective increase $F_i(a_i)-F_i(b_i)$, we introduce the dual inputs
\begin{equation}\label{eq:deleted-dual-inputs}
v_i(\beta)=-\nabla L_i(\beta),\qquad
v(\beta)=-\nabla L(\beta),\qquad
v_i^0=v_i(b_i).
\end{equation}
Using the construction of $a_i$ and strong convexity of $F_i(\cdot)$ we have the following lemma:

\begin{lemma} \label{lem:strong_convexity}
If $\|X_i\|_{\mathrm{op}}\le K_X$, then
\begin{equation}\label{eq:hybrid-motivation}
\frac{\mu}{2}\|m_i^D-b_i\|^2
\le
F_i(a_i)-F_i(b_i)
\le
C
\sum_{g\in\mathcal S_i}
\|v_g(b)-v_{i,g}^0\|^2,
\end{equation}
where $C$ depends only on $\mu$, $L_2$, and $K_X$.
\end{lemma}

The proof of this lemma can be found in Section \ref{proof:lemstrong:sconvex}.


\paragraph{Constructing trial point}

Obtaining a sharp bound on the sum in \eqref{eq:hybrid-motivation} directly remains challenging. We therefore introduce a ``trial'' dual vector $v_i^T$ and its corresponding primal point $b_i^T$ to split the dual movement $v(b)-v_i^0$ into two parts. For a tentative dual perturbation $u$ (the exact choice will be clarified later), set
\[
v_i^T=v_i^0+u,\qquad b_i^T=\Phi(v_i^T).
\]
where $\Phi_g$ is defined in \eqref{eq:prox-def}. 
By the definition of $v_i^T$, we have
\begin{equation}\label{eq:important:upper}
\begin{aligned}
\sum_{g\in\mathcal S_i}
\|v_g(b)-v_{i,g}^0\|^2
&\le
2\sum_{g\in\mathcal S_i}
\|v_g(b)-v_{i,g}^T\|^2+\|u_g\|^2
\le
2\|v(b)-v_i^T\|^2
+
2|\mathcal S_i|\max_g\|u_g\|^2.
\end{aligned}
\end{equation}

The preceding inequality suggests two requirements for the trial perturbation $u$:
\begin{itemize}
\item The trial dual vector $v_i^T$ should approximate $v(b)$ with a sharp bound on the error $\|v(b)-v_i^T\|^2$.

\item The largest blockwise movement $\max_g\|u_g\|^2$ should be small and admit an explicit bound.
\end{itemize}

We use the first requirement to guide the construction of $u$. 
Recall from~\eqref{eq:scorecontraction} that $c_i=\ell'(y_i;x_i^\top b)$. The full-data KKT condition gives $0\in\partial\widetilde F_i(b)$, where
\begin{equation}\label{eq:tilted-objective-motivation}
\widetilde F_i(\beta)
=
F_i(\beta)+c_ix_i^\top\beta.
\end{equation}
Here $c_i$ is held fixed as $\beta$ varies. Since $\widetilde F_i$ is $\mu$-strongly convex, $b$ is its unique minimizer. Thus, $b$ and $b_i$ minimize objectives differing only by a linear term.
Their optimality conditions also give
\begin{equation}\label{eq:proximal-optimality-pair}
b_i=\Phi(v_i^0),\qquad
b=\Phi(v(b)),\qquad
v(b)=-\nabla L_i(b)-c_ix_i,
\end{equation}
where
\[
\Phi(v)=(\Phi_g(v_g))_{g=1}^G
      =(\mu I+\lambda_1\partial r)^{-1}(v),
\]
and $\Phi_g$ is defined in \eqref{eq:prox-def}.

As we discussed before, the goal is to make $b_i^T$ and $v_i^T$ close to $b$ and $v(b)$. Since $b$ is the minimizer of $\widetilde F_i$, a small KKT residual
will imply that $b_i^T$ is close to $b$ by
strong convexity. For each block, let $J_g(v)=D\Phi_g(v)$ outside
$K_g$, and set $J_g(v)=0$ on $K_g$, including
its boundary. Define
\begin{equation}\label{eq:linearized-local-operators}
H_i=\nabla^2L_i(b_i),\qquad
J_i=\operatorname{diag}\{J_g(v_{i,g}^0)\}.
\end{equation}
 Note that the difference $b_i^T- b_i$ can be broken into $b_i^T- b_i = r+ d$
where $r$ and $d$ are defined as
\[
d=J_iu,\quad r=b_i^T-b_i-J_iu=\Phi(v_i^0+u)-\Phi(v_i^0)-\operatorname{diag}\{J_g(v_{i,g}^0)\}u.
\]
 Since $b_i^T = \Phi (v_i^T)$, we have
$v_i^T\in\mu b_i^T+\lambda_1\partial r(b_i^T)$.
Therefore,
\[
g_i^T
=
\nabla L_i(b_i^T)+c_ix_i+v_i^T
\in\partial\widetilde F_i(b_i^T).
\]
Note that since $0 \in \partial\widetilde F_i(b)$, if we can show that $g_i^T$ is small, then we can use strong convexity of $\widetilde F_i(b)$ to bound the difference $b_i^T-b$. Hence, next we simplify the expression for $g_i^T$. Using $v_i^0=-\nabla L_i(b_i)$ and
$b_i^T=b_i+J_iu+r$, we obtain
\begin{align*}
g_i^T&= c_ix_i+v_i^T-v_i^0+\nabla L_i(b_i+J_iu+r)+v_i^0
\\&=c_ix_i+u+H_iJ_iu \nonumber \\
&+(\nabla L_i(b_i+J_iu)-\nabla L_i(b_i)-H_iJ_iu)+(\nabla L_i(b_i+J_iu+r)-\nabla L_i(b_i+J_iu))
\\&=c_ix_i+u+H_iJ_iu+R_i^{\mathrm{loss}}
+R_i^{\mathrm{prox}},
\end{align*}
where
$R_i^{\mathrm{loss}}
=
\nabla L_i(b_i+J_iu)-\nabla L_i(b_i)-H_iJ_iu$
and
$R_i^{\mathrm{prox}}
=
\nabla L_i(b_i+J_iu+r)-\nabla L_i(b_i+J_iu).$
The first three terms are the linear part of the
KKT residual. We choose $u$ such that 
\begin{equation}\label{eq:linearized-kkt-motivation}
u+H_iJ_iu+c_ix_i=0, 
\end{equation}
Hence, if we define the following two matrices
\begin{equation}\label{eq:deletedobjects}
Q_i=(I+H_iJ_i)^{-1},\qquad
B_i=J_iQ_i=(I+J_iH_i)^{-1}J_i,
\end{equation}
then we can write $u_i$ and $d_i$ in the following way: 
\begin{equation}\label{eq:linear-response-definition}
u_i=-c_iQ_ix_i,\qquad
d_i=J_iu_i=-c_iB_ix_i.
\end{equation}
With this choice of $u_i$, we can now return to \eqref{eq:important:upper} and derive upper bounds for the different quantities appearing on its right-hand side. Since all of these quantities are random, it is convenient to first restrict attention to events that we expect to hold with high probability. We construct these regularity events below. Define the held-out score and deleted-data null-score energy by
\begin{equation}\label{eq:score-and-energy}
 S_i=1+|\ell'(y_i;x_i^\top b_i)|,\qquad
 V_i=\frac{1}{n-1}\sum_{j\ne i}\eta_j^2,\qquad \eta_j=\ell'(y_j;0).
\end{equation}
Recall $h_n=\sqrt{\log(en)/n}$. Choose sufficiently large fixed
constants $K_X,K_x,C_G$, followed by $D_0\ge4C_G$,
and put $T_n=(D_0h_n)^{-1}$. Define
\begin{equation}\label{eq:definition-of-a_i}
 \mathcal A_i=\left\{\|X_i\|_{\op}\le K_X,\quad
 \frac{v_-}{2}\le V_i\le2v_+\right\},
\end{equation}
\begin{equation}\label{eq:Gi}
 \mathcal G_i=\left\{\|x_i\|\le K_x,\quad
 \|Q_ix_i\|_{2,\infty}+\|X_iB_ix_i\|_\infty\le C_Gh_n\right\},
\end{equation}
\begin{equation}\label{eq:Hi}
 \mathcal H_i=\mathcal A_i\cap\mathcal G_i\cap\{S_i\le T_n\}.
\end{equation}
Here $\|z\|_{2,\infty}=\max_g\|z_g\|$, and $v_-,v_+$ are
from Assumption~\ref{ass:loss}.

One of the terms in \eqref{eq:important:upper} is $\max_g\|u_g\|^2$. By \eqref{eq:scorecontraction} and the definition of $\mathcal G_i$,
\begin{equation}\label{eq:ui-upperbound}
 \|u_i\|_{2,\infty}\le C_GS_ih_n.
\end{equation}
where $C_G$ is the constant chosen in event $\mathcal G_i$. 

Finding upper bounds for $\mathcal S_i$ and $\|v(b)-v_i^T\|^2$ is less direct and we do them in the following way. A stratum change between $b_i$ and $b$ requires the dual movement from $v_{i,g}^0$ to $v_g(b)$ to reach the activation boundary $\partial K_g$. Thus, blocks whose deleted dual inputs are close to this boundary may change strata under a small perturbation. To quantify this, for $v\in\R^{d_g}$ and a nonempty set $K\subseteq\R^{d_g}$, define $\operatorname{dist}(v,K)=\inf_{u\in K}\|v-u\|$. We then introduce the deleted boundary count at width $r>0$:
\begin{equation}\label{eq:def-boundarycount}
N_i(r)=\bigl|\{g:\operatorname{dist}(v_{i,g}^0,\partial K_g)\le r\}\bigr|.
\end{equation}
Thus, $N_i(r)$ counts blocks whose strata may change under small dual movements at scale $r$. For the fixed constant $D_0\ge4C_G$, select a width
\begin{equation}\label{eq:linearcontrol}
 t_i=D_0S_ih_n,\qquad \text{then} \quad
 \|u_i\|_{2,\infty}\le t_i/4\quad\text{on }\mathcal G_i.
\end{equation}
Define
\begin{equation}\label{eq:define-setB-ki}
 \mathcal B_i=\{g:\operatorname{dist}(v_{i,g}^0,\partial K_g)\le2t_i\},
 \qquad k_i=N_i(2t_i)+S_i^2.
\end{equation}
The set $\mathcal B_i$ contains the blocks close to their activation boundaries and $|\mathcal B_i| = N_i(2t_i)$. 

Assume that $\mathcal A_i\cap\mathcal G_i$ holds and $t_i\le1$. The following lemma gives deterministic bounds on the trial approximation error and controls the stratum changes between $b_i$ and $b$.

\begin{subequations}
\begin{lemma}
\label{lem:tilt}
The trial dual vector $v^T_i=v_i^0-c_iQ_ix_i$ and trial point $b_{i}^T=\Phi(v_i^T)$ satisfies
\begin{equation}\label{eq:tilterror}
\|b-b_i^T\|+\|v(b)-v_i^T\| \le Ct_i\sqrt{k_i}.
\end{equation}
For the actual stratum-change set
\[
\mathcal S_i
=
\{g:M_g(b_g)\ne M_g(b_{i,g})\},
\]
one has
\begin{equation}\label{eq:swap}
|\mathcal S_i|\le Ck_i,
\qquad
\sum_{g\in\mathcal S_i}
\|v_g(b)-v_{i,g}^0\|^2
\le Ct_i^2k_i.
\end{equation}
\end{lemma}
\end{subequations}

The second bound in \eqref{eq:swap} controls the resulting increase in the deleted objective. Combining this lemma with the Lemma~\ref{lem:strong_convexity} gives the following result.

\begin{lemma}\label{lem:bridge}
For either restricted estimator $m_i\in\{m_i^A,m_i^C\}$, the following bound holds on $\mathcal H_i$:
\begin{equation}\label{eq:vectorlocal}
 \|b_i-m_i\|+\|v_i(b_i)-v_i(m_i)\|\le Ct_i\sqrt{k_i}.
\end{equation}
\end{lemma}

\paragraph{Prediction, loss, and projector comparisons.}
The estimate in~\eqref{eq:vectorlocal} controls the Euclidean
coefficient displacement.
The next lemma uses the deleted-data linear responses
$Q_i$ and $B_i$ and the bounds in $\mathcal G_i$ to obtain
the sharper scalar prediction estimate.
Set 
\begin{equation}\label{eq:losslocal-def}
 \Delta\ell_i=\ell(y_i;x_i^\top b_i)-\ell(y_i;x_i^\top m_i).
\end{equation}
The following lemma combines the bounds in \eqref{eq:vectorlocal} with the bounds in event $\mathcal G_i$ to control both the held-out prediction difference $x_i^\top\Delta_i$ and the loss difference $\Delta\ell_i$.

\begin{subequations}
\begin{lemma}\label{lem:prediction}
For either restricted estimator $m_i\in\{m_i^A,m_i^C\}$, define $\Delta_i=b_i-m_i$. Then the following bounds hold on $\mathcal H_i$:
\begin{align}
 |x_i^\top\Delta_i|&\le Ch_nt_i k_i,\label{eq:predictionlocal}\\
 |\Delta\ell_i|+\|\Delta_i\|^2&\le Ct_i^2k_i,\label{eq:losslocal}\\
 \|P-P_i\|_F^2&\le Ck_i.\label{eq:projectorlocal}
\end{align}
\end{lemma}
\end{subequations}

Substituting $t_i=D_0S_ih_n$ from \eqref{eq:linearcontrol} and $k_i=N_i(2t_i)+S_i^2$ from \eqref{eq:define-setB-ki} makes the dependence on the held-out score and the boundary count explicit:
\begin{align}
 |\Delta\ell_i|+\|\Delta_i\|^2
 &\le Ch_n^2\{S_i^2N_i(2D_0S_ih_n)+S_i^4\},\label{eq:localsummaryloss}\\
 |x_i^\top\Delta_i|
 &\le Ch_n^2\{S_iN_i(2D_0S_ih_n)+S_i^3\},\label{eq:localsummarypred}\\
 p^{-1}\|P-P_i\|_F^2
 &\le Cp^{-1}\{N_i(2D_0S_ih_n)+S_i^2\}.\label{eq:localsummaryproj}
\end{align}
These bounds hold on $\mathcal H_i$ and involve powers of $S_i$ together with the boundary count at the random width $2D_0S_ih_n$.

So far, our arguments have been deterministic under certain events. 
The goal of the remaining sections is to remove the conditioning arguments and obtain bounds on the expectations uniformly over the deletion index \(i\), thereby establishing \eqref{eq:mainloss}, \eqref{eq:mainprediction}, and \eqref{eq:mainprojector}.
The LO comparison in \eqref{eq:mainrisk} then follows by averaging the loss differences. The rest of the proof proceeds in two parts. \\


\noindent\textbf{Section~\ref{sec:integration}: Expected errors on $\mathcal H_i$.}
To take expectations of the deterministic comparisons, we control the random boundary count $N_i(2D_0S_ih_n)$ and the accompanying powers of $S_i$.
At a fixed width $r$, the boundary count $N_i(r)$ depends only on the remaining observations.
The random width $2D_0S_ih_n$ introduces additional dependence on the omitted observation through $S_i$.
We combine fixed-width boundary-count estimates with conditional bounds for $S_i$ to obtain expectation bounds for the four comparison quantities on $\mathcal H_i$.\\

\noindent\textbf{Section 3.4: Exceptional events and completion of the proof.}
We control the expectations of the four quantities on $\mathcal H_i^c$ and combine them with the bounds on $\mathcal H_i$ to establish the uniform expectation bounds in the main theorem.
Finally, averaging the held-out loss differences gives the LO comparison, and Markov's inequality yields the stated stochastic order.

\subsection{Expected error bounds on $\mathcal H_i$}\label{sec:integration}

Section~\ref{sec:local} establishes deterministic bounds on the coefficient, prediction, and loss differences on $\mathcal H_i$. These bounds involve the held-out score $S_i$ and the boundary count $N_i(2t_i)$, both of which depend on the sampled data and are therefore random.
To take expectations of these error bounds over the data, we need to control these random quantities on \(H_i\), as well as the contribution from \(H_i^c\). This subsection addresses the first task

\paragraph{Boundary counts on $\mathcal H_i$.}

We first consider the randomness in the boundary count $N_i(2t_i)$.
At a fixed width $r>0$, the boundary count $N_i(r)$ defined in \eqref{eq:def-boundarycount} depends only on the remaining observations. The random width $2t_i$ introduces additional dependence on the omitted observation through $S_i$. We therefore first establish an expectation bound for $N_i(r)$ at each fixed width, and later combine this bound with the conditional bounds for $S_i$ to handle $N_i(2t_i)$.
The following lemma gives the required fixed-width expectation bound.

\begin{lemma}[Boundary counts at fixed widths]
\label{lem:antic}
With $\mathcal A_i$ as defined in \eqref{eq:definition-of-a_i}, the following bound holds for every deterministic $r>0$, uniformly in $i$:
\begin{equation}\label{eq:fixedwidth}
\E[N_i(r)\ind_{\mathcal A_i}]\le Cpr.
\end{equation}
For small $r$, each block satisfies
$\Prob\{\dist(v_{i,g}^0,\partial K_g)\le r,
\mathcal A_i\}\le Cr$.
\end{lemma}

We next establish conditional bounds for $S_i$. Recall that
\[
\mathcal D_{-i} = \sigma((y_j,x_j):j\ne i)
\]
is the $\sigma$-field generated by all observations except observation $i$.
Since $b_i$ is fitted using only these observations, it is $\mathcal D_{-i}$-measurable.
The omitted pair $(y_i,x_i)$ is independent of $\mathcal D_{-i}$.
Thus, conditional on $\mathcal D_{-i}$, $b_i$ is fixed, and the randomness in $S_i$ comes from the omitted pair.


\begin{subequations}
\begin{lemma}[Conditional bounds for $S_i$]\label{lem:tail}
There are deterministic functions $q_n:[1,\infty)\to[0,\infty)$, independent of $i$, such that, for every $i\in\{1,\ldots,n\}$ and $u\ge1$, the following bounds hold almost surely on $\mathcal A_i$:
\begin{equation}\label{eq:conditionaltail}
 \Prob(S_i>u\mid\mathcal D_{-i})\le q_n(u),\qquad
 \E[S_i^3\mid\mathcal D_{-i}]\le C,
\end{equation}
and
\begin{equation}\label{eq:tailsum}
 \sup_n\sum_{k=0}^\infty2^{3k}q_n(2^k)<\infty.
\end{equation}
For sufficiently large fixed $u_0$, one can use $q_n(u)=1$ on $[1,u_0]$ and
\begin{equation}\label{eq:expanded-display-013}
q_n(u)=\Prob(|\eta_1|>u/4)+2e^{-c u^2},\qquad u>u_0,
\end{equation}
after changing fixed constants if needed.
\end{lemma}
\end{subequations}

Combining the conditional bounds for $S_i$ on $\mathcal A_i$ in Lemma~\ref{lem:tail} with the fixed-width expectation bound in Lemma~\ref{lem:antic}, we obtain the following expectation bounds.

\begin{lemma}[Weighted boundary counts at random widths]
\label{lem:dyadic}
For every fixed $a_0>0$ and each $q\in\{0,1,2\}$, the following bound holds uniformly in $i$:
\begin{equation}\label{eq:weightedcount}
\E\!\left[ S_i^q N_i(a_0S_ih_n)\ind_{\mathcal A_i} \right]
\le Cph_n,
\end{equation}
where $C$ may depend on $a_0$, but not on $n$ or $i$.
\end{lemma}

Taking $a_0=2D_0$, where $D_0$ is the fixed constant chosen in Section~\ref{sec:local}, and recalling $t_i=D_0S_ih_n$, we obtain
\[
\E\!\left[
S_i^q N_i(2t_i)\ind_{\mathcal A_i}
\right]
\le Cph_n,
\qquad q\in\{0,1,2\}.
\]
Since $\mathcal H_i\subseteq\mathcal A_i$, this provides the expectation bounds for the weighted boundary counts appearing in the deterministic comparisons.

\paragraph{Error bounds on $\mathcal H_i$.}
We now use these bounds to control the loss, coefficient, and prediction errors on $\mathcal H_i$.
For either restricted estimator $m_i\in\{m_i^A,m_i^C\}$, define
\begin{equation}\label{eq:expanded-display-014}
W_i = |\Delta\ell_i|+\|\Delta_i\|^2+|x_i^\top\Delta_i|.
\end{equation}
In Lemma~\ref{lem:prediction}, we have established the deterministic bound for the components of $W_i$. Combining \eqref{eq:localsummaryloss} and \eqref{eq:localsummarypred}, and using $S_i\ge1$, gives
\begin{align}
\label{eq:main-local-loss-expectation}
\E[(|\Delta\ell_i|+\|\Delta_i\|^2)\ind_{\mathcal H_i}]
&\le 
Ch_n^2\E[S_i^2N_i(2D_0S_ih_n)\ind_{\mathcal H_i}]
+Ch_n^2\E[S_i^4\ind_{\mathcal H_i}].
\end{align}
For the first term, we use $\mathcal H_i\subseteq\mathcal A_i$ and apply Lemma~\ref{lem:dyadic} with $q=2$ and $a_0=2D_0$ to obtain
\begin{equation}\label{eq:main-weighted-boundary-term}
 h_n^2\E[S_i^2N_i(2D_0S_ih_n)\ind_{\mathcal H_i}]
 \le h_n^2\E[S_i^2N_i(2D_0S_ih_n)\ind_{\mathcal A_i}]
 \le Cph_n^3.
\end{equation}
For the second term, we use $S_i\le T_n$ on $\mathcal H_i$.
Since $\mathcal A_i$, defined in \eqref{eq:definition-of-a_i}, is $\mathcal D_{-i}$-measurable, the conditional third-moment bound in \eqref{eq:conditionaltail} yields
\begin{equation}\label{eq:main-truncated-fourth-power}
 h_n^2\E[S_i^4\ind_{\mathcal H_i}]
 \le h_n^2T_n\E[S_i^3\ind_{\mathcal A_i}]=\frac{h_n}{D_0}\E[\ind_{\mathcal A_i}\E(S_i^3\mid\mathcal D_{-i})]
 \le Ch_n.
\end{equation}
A similar argument for the prediction error, using \eqref{eq:localsummarypred}, Lemma~\ref{lem:dyadic} with $q=1$ and $a_0=2D_0$, and \eqref{eq:conditionaltail}, gives
\begin{align}\label{eq:main-local-prediction-final}
\E[|x_i^\top\Delta_i|\ind_{\mathcal H_i}]
\le Ch_n^2\E[S_iN_i(2D_0S_ih_n)\ind_{\mathcal A_i}]
+Ch_n^2\E[S_i^3\ind_{\mathcal A_i}]
\le Cph_n^3+Ch_n^2.
\end{align}
Combining these results, we obtain, uniformly in $i$,
\begin{equation}\label{eq:good-event-error}
\E[W_i\ind_{\mathcal H_i}]
\le C(ph_n^3+h_n^2+h_n).
\end{equation}

We also bound the projector discrepancy $p^{-1}\|P-P_i\|_F^2$ on $\mathcal H_i$.
We combine the deterministic bound in \eqref{eq:localsummaryproj} from Lemma~\ref{lem:prediction}, the expectation bound in \eqref{eq:weightedcount} from Lemma~\ref{lem:dyadic}, applied with $q=0$ and $a_0=2D_0$, and the conditional third-moment bound in \eqref{eq:conditionaltail} from Lemma~\ref{lem:tail}.
Using $\mathcal H_i\subseteq\mathcal A_i$ and $S_i^2\le S_i^3$, we obtain
\begin{equation}\label{eq:main-local-projector}
\frac1p\E[\|P-P_i\|_F^2\ind_{\mathcal H_i}]
\le
\frac Cp\E[N_i(2D_0S_ih_n)\ind_{\mathcal A_i}]
+
\frac Cp\E[S_i^2\ind_{\mathcal A_i}]
\le C(h_n+p^{-1}).
\end{equation}

\subsection{Exceptional events and completion of the proof}\label{sec:integration2}

The previous subsection establishes an expectation bound for $W_i$ on $\mathcal H_i$. We now consider the exceptional event $\mathcal H_i^c$ and control the expectation $\E[W_i\ind_{\mathcal H_i^c}]$. Then we combine the two bounds to obtain the loss, coefficient, and prediction bounds in the main theorem.

\paragraph{The remaining exceptional events}
To control the expected errors outside $\mathcal H_i$, we use the disjoint decomposition
\[
\mathcal H_i^c
=
\mathcal A_i^c
\,\dot\cup\,
(\mathcal A_i\cap\mathcal G_i^c)
\,\dot\cup\,
(\mathcal A_i\cap\mathcal G_i\cap\{S_i>T_n\}).
\]
The first event corresponds to failure of the deleted-data conditions defining $\mathcal A_i$.
The second corresponds to failure of the Gaussian bounds defining $\mathcal G_i$ while $\mathcal A_i$ holds.
The third corresponds to $S_i$ exceeding the cutoff $T_n$ while both $\mathcal A_i$ and $\mathcal G_i$ hold. The global comparison in Lemma~\ref{lem:global}, together with Lemmas~\ref{lem:baddeleted}, \ref{lem:gaussian}, and~\ref{lem:tail}, gives the following estimates.


\begin{subequations}
\begin{lemma}[Bounds on the remaining exceptional events]\label{lem:complements}
For either restricted estimator $m_i\in\{m_i^A,m_i^C\}$, let $W_i$ be as defined in \eqref{eq:expanded-display-014}. With $T_n=(D_0h_n)^{-1}$, the following bounds hold uniformly in $i$:
\begin{align}
\E[W_i\ind_{\mathcal A_i^c}] &\le Cn^{-1/2}, \label{eq:compA}\\
\E[W_i\ind_{\mathcal A_i\cap\mathcal G_i^c}] &\le Cn^{-4}+Ce^{-cn}, \label{eq:compG}\\
\E[W_i\ind_{\mathcal A_i\cap\mathcal G_i \cap\{S_i>T_n\}}] &\le Ch_n.
\label{eq:compS}
\end{align}
\end{lemma}
\end{subequations}

Now we are ready to prove the main results.
\paragraph{Proof of \eqref{eq:mainloss} and \eqref{eq:mainprediction}.}
Since the three events in Lemma~\ref{lem:complements} partition $\mathcal H_i^c$, combining these estimates with \eqref{eq:good-event-error} gives
\begin{equation}\label{eq:main-unconditional-sum}
\begin{aligned}
\E W_i
=
\E[W_i\ind_{\mathcal H_i}]
+
\E[W_i\ind_{\mathcal H_i^c}]
\le
C(ph_n^3+h_n^2+h_n+n^{-1/2}+n^{-4}+e^{-cn}).
\end{aligned}
\end{equation}
We now simplify the bound on the right hand side.
For all sufficiently large $n$, since $h_n=\sqrt{\frac{\log(en)}{n}}$, we have $h_n^2\le h_n$ and $n^{-1/2}+n^{-4}+e^{-cn}\le Ch_n$.
Consequently,
\begin{equation}\label{eq:main-rate-substitution}
\begin{aligned}
\E W_i
&\le C(ph_n^3+h_n)
=
\frac{C}{\sqrt n}
\left\{\frac{p}{n}\log^{3/2}(en)+\log^{1/2}(en)\right\}
\le C\frac{\log^{3/2}(en)}{\sqrt n}
= C\mathfrak r_n,
\end{aligned}
\end{equation}
where the last inequality uses the uniform upper bound on $p/n$ and $\log(en)\ge1$.
The bound holds uniformly in $i$ and for both restricted estimators.
Since $W_i$ bounds each of the three error terms, taking suprema over $i$ proves \eqref{eq:mainloss} and \eqref{eq:mainprediction}, with the expectations remaining inside the suprema.

\paragraph{Proof of \eqref{eq:mainprojector}.}
The expectation on $\mathcal H_i$ is bounded in \eqref{eq:main-local-projector}.
We now estimate the normalized projector discrepancy on the three complementary events. The block expression in \eqref{eq:prediction-projector-detail} gives $0\le\|P-P_i\|_F^2\le p$.
Thus, Lemmas~\ref{lem:baddeleted} and~\ref{lem:gaussian} yield
\begin{align}
\frac1p\E[\|P-P_i\|_F^2\ind_{\mathcal A_i^c}]
&\le \Prob(\mathcal A_i^c)
\le Cn^{-1/2},
\label{eq:main-projector-complements_1}\\
\frac1p\E[\|P-P_i\|_F^2
             \ind_{\mathcal A_i\cap\mathcal G_i^c}]
&\le \Prob(\mathcal A_i\cap\mathcal G_i^c)
\le Cn^{-20}+Ce^{-cn}.
\label{eq:main-projector-complements_2}
\end{align}
For the event $\mathcal A_i\cap\mathcal G_i\cap\{S_i>T_n\}$, we use the $\mathcal D_{-i}$-measurability of $\mathcal A_i$, Markov's inequality, and the conditional third-moment bound in \eqref{eq:conditionaltail}:
\begin{equation}\label{eq:main-projector-complements_3}
\begin{aligned}
\frac1p\E[\|P-P_i\|_F^2
\ind_{\mathcal A_i\cap\mathcal G_i\cap\{S_i>T_n\}}]
&\le \Prob(\mathcal A_i\cap\{S_i>T_n\})\\
&=
\E[\ind_{\mathcal A_i}
\Prob(S_i>T_n\mid\mathcal D_{-i})]\\
&\le
T_n^{-3}\E\!\left[
\ind_{\mathcal A_i}
\E(S_i^3\mid\mathcal D_{-i})
\right]\\
&\le CT_n^{-3}
=CD_0^3h_n^3
\le Ch_n^3.
\end{aligned}
\end{equation}
Combining \eqref{eq:main-local-projector} with \eqref{eq:main-projector-complements_1}--\eqref{eq:main-projector-complements_3} and summing over the four events gives
\[
\frac1p\E\|P-P_i\|_F^2
\le
C(h_n+p^{-1}+n^{-1/2}+n^{-20}+e^{-cn}+h_n^3)
\le Ch_n,
\]
where the last inequality uses $p^{-1}=O(n^{-1})$, $n^{-20}+e^{-cn}\le Cn^{-1/2}\le Ch_n$, and $h_n^3\le h_n$ for all sufficiently large $n$.
These estimates are uniform in $i$, so taking the supremum proves \eqref{eq:mainprojector}. For the complementary normal projectors,
\begin{equation}\label{eq:main-normal-projector-identity}
(I-P)-(I-P_i)=-(P-P_i),
\end{equation}
so the same Frobenius bound holds.

\paragraph{Proof of \eqref{eq:mainrisk}.}
We obtain the LO comparison by averaging the held-out loss differences $\Delta\ell_i$ over the deletion indices. For each $D\in\{A,C\}$ and every dataset, the triangle inequality gives
\begin{equation}\label{eq:main-risk-triangle-detail}
|\LO-\widetilde\LO^D|
=
\left|\frac1n\sum_{i=1}^n\Delta\ell_i\right|
\le
\frac1n\sum_{i=1}^n|\Delta\ell_i|.
\end{equation}
By linearity of expectation and the uniform loss bound in \eqref{eq:mainloss}, we obtain
\begin{equation}\label{eq:main-risk-expectation-detail}
\E|\LO-\widetilde\LO^D|
\le
\frac1n\sum_{i=1}^n\E|\Delta\ell_i|
\le C\mathfrak r_n.
\end{equation}
For any $M>0$, Markov's inequality further yields
\begin{equation}\label{eq:main-risk-stochastic-order}
\Prob\left\{ \frac1n\sum_{i=1}^n|\Delta\ell_i| >M\mathfrak r_n \right\}
\le\frac{C}{M}.
\end{equation}
Thus $\frac1n\sum_{i=1}^n|\Delta\ell_i|=O_p(\mathfrak r_n)$.

\section{Proof of Theorem~\ref{thm:uniform} }\label{sec:proof:thm2.8}

We prove Theorem~\ref{thm:uniform}, with $M_n$ and $t_n=D_0M_nh_n$ as defined in Section~\ref{sec:main results}. The deterministic estimates from Section~\ref{sec:local} remain unchanged. The additional work is to construct a common event, dominate all deleted boundary counts by one full-data count, and control the maximum error on the complement.

\begin{proof}
The additional tail assumption is used only in this proof and in the corresponding maximum assertions for affine continuation. It supplies score moments of every fixed order and a common high-probability score bound.

\noindent\textbf{Step 1: construct simultaneous spectral and score-energy events.}
First, integration of~\eqref{eq:weibull} gives, for every fixed $r>0$,
\begin{equation}\label{eq:uniform-score-moments}
 \E|\eta_j|^r
 =r\int_0^\infty u^{r-1}\Prob(|\eta_j|>u)\,du
 \le2r\int_0^\infty u^{r-1}e^{-(u/K)^\alpha}\,du
 =2K^r\Gamma(1+r/\alpha)<\infty.
\end{equation}
In particular, the fourth and eighth moments used in this section are available from the additional assumption; they are not being added to the mean theorem.

Let $Y_j=\eta_j^2$, $v=\E Y_j\in[v_-,v_+]$, and
\begin{equation}\label{eq:uniform-energy-truncation}
 R_n=C_R\log^{2/\alpha}(en),\qquad
 U_j=Y_j\ind_{\{Y_j\le R_n\}},\qquad \xi_j=U_j-\E U_j.
\end{equation}
The variables $U_j$ are independent, lie in $[0,R_n]$, and have variance at most $\E\eta_j^4\le C_4$. The tail of $Y_j$ obeys $\Prob(Y_j>s)\le2e^{-(s/K^2)^{\alpha/2}}$. The layer-cake identity gives
\begin{equation}\label{eq:uniform-truncation-tail-integral}
 \E[Y_j\ind_{\{Y_j>R\}}]
 =R\Prob(Y_j>R)+\int_R^\infty\Prob(Y_j>s)\,ds
 \le C e^{-cR^{\alpha/2}}
\end{equation}
for large $R$. For the integral, split its exponential into two equal factors, bound one by its value at $R$, and integrate the other over $[0,\infty)$; the factor $R$ in the first term is absorbed by decreasing $c$. Choose $C_R$ large enough that
\begin{equation}\label{eq:uniform-truncation-probability}
 \Prob(\max_jY_j>R_n)\le Cn^{-22},\qquad
 0\le v-\E U_j=\E[Y_j\ind_{\{Y_j>R_n\}}]\le Cn^{-22}.
\end{equation}

For clarity, the concentration estimate for the truncated average can be derived directly. Since $|\xi_j|\le R_n$ and $\E\xi_j=0$, the inequalities $k!\ge2\,3^{k-2}$ for $k\ge2$ and $\E|\xi_j|^k\le R_n^{k-2}\E\xi_j^2$ imply, for $|s|R_n<3$,
\begin{equation}\label{eq:uniform-bernstein-mgf}
 \begin{aligned}
 \E e^{s\xi_j}
 &\le1+\sum_{k\ge2}\frac{|s|^kR_n^{k-2}\E\xi_j^2}{k!}\\
 &\le1+\frac{s^2C_4}{2(1-|s|R_n/3)}
 \le\exp\left(\frac{s^2C_4}{2(1-|s|R_n/3)}\right).
 \end{aligned}
\end{equation}
Independence, exponential Markov inequality, and the choice $s=\varepsilon/(C_4+R_n\varepsilon/3)$ give the two-sided Bernstein bound
\begin{equation}\label{eq:uniform-energy-bernstein}
 \Prob\left(\left|\frac1n\sum_j\xi_j\right|>\varepsilon\right)
 \le2\exp\left(-\frac{n\varepsilon^2}{2(C_4+R_n\varepsilon/3)}\right)
 \le2e^{-c_\varepsilon n/R_n}
\end{equation}
for a fixed $\varepsilon>0$. Since $R_n$ is a fixed power of $\log(en)$, the final probability is at most $Cn^{-22}$ for sufficiently large $n$.

Take $\varepsilon=v_-/16$. On the event of no truncation and no deviation in~\eqref{eq:uniform-energy-bernstein}, equation~\eqref{eq:uniform-truncation-probability} gives, for large $n$,
\begin{equation}\label{eq:uniform-full-energy-window}
 V_{\mathrm{full}}:=\frac1n\sum_j\eta_j^2,
 \qquad |V_{\mathrm{full}}-v|\le v_-/8.
\end{equation}
Moreover, deleting any one observation changes this mean by
\begin{equation}\label{eq:uniform-deletion-energy-change}
 |V_i-V_{\mathrm{full}}|
 =\frac{|V_{\mathrm{full}}-Y_i|}{n-1}
 \le\frac{V_{\mathrm{full}}+R_n}{n-1}\le v_-/8
\end{equation}
for all sufficiently large $n$, simultaneously over $i$. Hence every deleted energy lies in $[v_-/2,2v_+]$, and the full energy lies in the same window. Let $\mathcal E_n^{\mathrm{eng}}$ denote the event just constructed; its failure probability is at most $Cn^{-22}$.

The Gaussian net calculation in~\eqref{eq:bad-net-gaussian-tail}--\eqref{eq:bad-design-spectral-tail}, with $n$ rows, gives
\begin{equation}\label{eq:uniform-full-spectral-event}
 \Prob(\|X\|_{\op}>K_X)\le Ce^{-cn}.
\end{equation}
On $\{\|X\|_{\op}\le K_X\}$ every deleted design also has norm at most $K_X$. Together with $\mathcal E_n^{\mathrm{eng}}$, this ensures all events $\mathcal A_i$ and the full-data spectral and energy bounds.

\noindent\textbf{Step 2: simultaneous Gaussian responses and held-out scores.}
Recall that $\mathcal D_{-i}$ is the $\sigma$-field generated by all observations except observation $i$. The event $\mathcal A_i$ and the matrices $X_i,B_i,Q_i$ are $\mathcal D_{-i}$-measurable, whereas the held-out pair $(y_i,x_i)$ is independent of $\mathcal D_{-i}$. In particular, $x_i\mid\mathcal D_{-i}\sim N(0,\Sigma)$. For the events $B_i^{\mathrm{lin}}$ and $B_i^{\mathrm{row}}$ defined in~\eqref{eq:complements-gaussian-failure-events}, conditional Gaussian control and a union bound give
\begin{equation}\label{eq:uniform-gaussian-union}
 \begin{aligned}
 \Prob\left(\bigcup_{i=1}^n(\mathcal A_i\cap B_i^{\mathrm{lin}})\right)
 &\le\sum_i\E[\ind_{\mathcal A_i}\Prob(B_i^{\mathrm{lin}}\mid\mathcal D_{-i})]
 \le Cn^{-19},\\
 \Prob\left(\bigcup_{i=1}^nB_i^{\mathrm{row}}\right)
 &\le Cne^{-cn}\le Ce^{-c'n}.
 \end{aligned}
\end{equation}
On $\mathcal A_i$, the score decomposition and its conditional Gaussian predictor give, for sufficiently large $u$,
\begin{equation}\label{eq:uniform-score-tail}
 \Prob(S_i>u\mid\mathcal D_{-i})
 \le2e^{-c_1u^\alpha}+2e^{-c_2u^2}.
\end{equation}
For $M_n=C_M\log^a(en)$, $a=\max\{1/2,1/\alpha\}$, both $a\alpha\ge1$ and $2a\ge1$. Choose $C_M$ large enough that each term in~\eqref{eq:uniform-score-tail}, evaluated at $u=M_n$, is at most $Cn^{-22}$. It follows that
\begin{equation}\label{eq:uniform-score-union}
 \Prob\left(\bigcup_i\{\mathcal A_i,S_i>M_n\}\right)
 \le\sum_i\E[\ind_{\mathcal A_i}\Prob(S_i>M_n\mid\mathcal D_{-i})]
 \le Cn^{-21}.
\end{equation}
Let $\mathcal E_n$ be the intersection of the full spectral event, $\mathcal E_n^{\mathrm{eng}}$, all Gaussian row and linear-response events, and $\{\max_iS_i\le M_n\}$. On the first two events all $\mathcal A_i$ hold, so~\eqref{eq:uniform-gaussian-union} and~\eqref{eq:uniform-score-union} apply to the remaining failures. Therefore
\begin{equation}\label{eq:uniform-common-event-probability}
 \Prob(\mathcal E_n^c)
 \le Cn^{-22}+Ce^{-cn}+Cn^{-19}+Cn^{-21}
 \le Cn^{-12}.
\end{equation}
This construction uses union bounds, not independence of events corresponding to different deletions.

If the loss derivative is globally bounded by a fixed constant, then $S_i$ is deterministically bounded by one plus that constant, and the null scores are also bounded. In that case take $M_n$ constant, omit the score-tail union, and apply~\eqref{eq:uniform-energy-bernstein} directly to the bounded squared scores. The same common-event conclusion holds.

\noindent\textbf{Step 3: compare all deleted boundary counts with the full-data count.}
Within this proof abbreviate
\begin{equation}\label{eq:uniform-common-counts}
 N_i=N_i(2t_n),\qquad k_i=N_i+M_n^2,\qquad
 N_{\mathrm{full}}(r)=\#\{g:\dist(v_g(b),\partial K_g)\le r\},
 \qquad t_n=D_0M_nh_n.
\end{equation}
For large $n$, $t_n\le1$. On $\mathcal E_n$, the original row-specific width and count satisfy $D_0S_ih_n\le t_n$ and $N_i(2D_0S_ih_n)+S_i^2\le k_i$. Thus the deterministic lemmas, by monotonicity of their upper bounds, imply for their original trial $u_i=-c_iQ_ix_i$
\begin{equation}\label{eq:fulltrial}
 \max_g\|u_{i,g}\|\le t_n/4,
 \qquad
 \|v(b)-(v_i^0+u_i)\|^2\le Ct_n^2(N_i+M_n^2).
\end{equation}
Fix a constant $R>9/4$. A block counted by $N_i$ but not by $N_{\mathrm{full}}(Rt_n)$ satisfies
\begin{equation}\label{eq:uniform-count-separation}
 \begin{aligned}
 \|v_g(b)-v_{i,g}^0\|
 &\ge\dist(v_g(b),\partial K_g)-\dist(v_{i,g}^0,\partial K_g)\\
 &>(R-2)t_n,\\
 \|v_g(b)-(v_{i,g}^0+u_{i,g})\|
 &>(R-2)t_n-t_n/4=(R-9/4)t_n.
 \end{aligned}
\end{equation}
Let $J_i^{\mathrm{out}}$ be the set of these blocks. Summing their squared contributions in~\eqref{eq:fulltrial} gives
\begin{equation}\label{eq:uniform-count-pre-absorption}
 \begin{aligned}
 |J_i^{\mathrm{out}}|(R-9/4)^2t_n^2
 &\le Ct_n^2(N_i+M_n^2),\\
 N_i&\le N_{\mathrm{full}}(Rt_n)
       +\frac{C}{(R-9/4)^2}(N_i+M_n^2).
 \end{aligned}
\end{equation}
Choose the fixed $R$ so that $C/(R-9/4)^2\le1/2$. Moving that multiple of $N_i$ to the left yields, simultaneously over all deletions,
\begin{equation}\label{eq:fullabsorb}
 \max_iN_i\le2N_{\mathrm{full}}(Rt_n)+CM_n^2
 \quad\text{on }\mathcal E_n.
\end{equation}

\noindent\textbf{Step 4: integrate the single full-data count.}
Let $\mathcal A_{\mathrm{full}}$ be the full-data spectral and energy event, which contains $\mathcal E_n$. The full-data version of Lemma~\ref{lem:antic} follows from the same clamped objectives with all $n$ rows retained. Explicitly, clamp block $g$ in $F$ and define its scores $a_j^{\mathrm{full}}$ using the clamped fit. Conditional on $\sigma(y_j,x_{j,-g}:1\le j\le n)$, these scores are fixed, and on the corresponding measurable enlargement of $\mathcal A_{\mathrm{full}}$ the score sandwich gives
\begin{equation}\label{eq:uniform-full-clamped-energy}
 cn\le\sum_{j=1}^n(a_j^{\mathrm{full}})^2\le Cn.
\end{equation}
The conditional Gaussian-likelihood curvature bound still has size $Cp$, and the convolution weights give curvature parameter $Cp/\sum_j(a_j^{\mathrm{full}})^2\le C$. The clamping residual and tube transfer use only the retained design norm, so they hold with $X$ in place of $X_i$. The conditional-density and tube-volume bounds therefore give, for each block,
\begin{equation}\label{eq:uniform-full-block-probability}
 \Prob\{\dist(v_g(b),\partial K_g)\le r,\mathcal A_{\mathrm{full}}\}
 \le Cr,
\end{equation}
with the same large-width extension as before. Summing over $G\le p$ and using the event inclusion gives
\begin{equation}\label{eq:uniform-full-count-expectation}
 \E[N_{\mathrm{full}}(Rt_n)\ind_{\mathcal E_n}]
 \le\E[N_{\mathrm{full}}(Rt_n)\ind_{\mathcal A_{\mathrm{full}}}]
 \le Cpt_n.
\end{equation}
Equation~\eqref{eq:fullabsorb} now yields
\begin{equation}\label{eq:maxcountmean}
 \E[\max_i k_i\,\ind_{\mathcal E_n}]
 \le C\E[N_{\mathrm{full}}(Rt_n)\ind_{\mathcal E_n}]+CM_n^2
 \le C(pt_n+M_n^2).
\end{equation}
This is an expected bound for a single full-data count, not a union bound over all deleted counts.

\noindent\textbf{Step 5: prediction, loss, displacement, and projectors on the common event.}
The deterministic inequalities, with widths and counts bounded as above, imply
\begin{equation}\label{eq:uniform-local-maxima}
 \begin{aligned}
 \max_i|x_i^\top\Delta_i|&\le Ch_nt_n\max_i k_i,\\
 \max_i|\Delta\ell_i|+\max_i\|\Delta_i\|^2
 &\le Ct_n^2\max_i k_i,\\
 \max_i p^{-1}\|P-P_i\|_F^2&\le Cp^{-1}\max_i k_i
 \qquad\text{on }\mathcal E_n.
 \end{aligned}
\end{equation}
Multiplying by $\ind_{\mathcal E_n}$ and taking expectations in~\eqref{eq:uniform-local-maxima}, then using~\eqref{eq:maxcountmean}, proves the two maximum expectation bounds on this event. It also gives the moment bounds needed for the two stochastic-order statements.

\noindent\textbf{Step 6: an integrable envelope on the complementary event.}
For each literal restricted fit, zero is feasible and the original penalty is nonnegative. The tangent inequality at zero gives
\begin{equation}\label{eq:uniform-restricted-zero-comparison}
 \begin{aligned}
 F_i(m_i)&\le F_i(0),\\
 F_i(m_i)&\ge F_i(0)+\eta_{-i}^\top X_im_i+\frac\mu2\|m_i\|^2,\\
 \|m_i\|&\le\frac2\mu\|X_i\|_{\op}\|\eta_{-i}\|
 \le\frac2\mu\|X\|_{\op}\|\eta\|.
 \end{aligned}
\end{equation}
The norm inequality is immediate when $m_i=0$, and otherwise follows by division by $\|m_i\|$. Together with the ambient bound, it yields, for every dataset,
\begin{equation}\label{eq:uniform-all-fit-norms}
 \max_i(\|b_i\|+\|m_i\|)\le C\|X\|_{\op}\|\eta\|.
\end{equation}
Put $T=\|X\|_{\op}$ and $U=\|\eta\|$ temporarily. Since $\max_i\|x_i\|\le T$ and $\max_i|\eta_i|\le U$,
\begin{equation}\label{eq:uniform-prediction-score-envelope}
 \begin{aligned}
 \max_i|x_i^\top(b_i-m_i)|&\le CT^2U,\\
 \max_i|\ell'(y_i;x_i^\top b_i)|&\le U+L_2T\max_i\|b_i\|
 \le C(1+T^2)U.
 \end{aligned}
\end{equation}
Taylor's inequality then gives the common envelope
\begin{equation}\label{eq:uniform-polynomial-envelope}
 \max_i|\Delta\ell_i|\le C(T^2+T^4)U^2\le Z_n,
 \qquad
 \max_i|x_i^\top\Delta_i|\le Z_n,
 \qquad
 Z_n=C(1+T^4)(1+U^2),
\end{equation}
after increasing $C$.

All fixed Gaussian spectral moments are uniformly bounded. Jensen for the nonnegative sum of squared scores and~\eqref{eq:uniform-score-moments} give
\begin{equation}\label{eq:uniform-score-vector-eighth}
 \E\|\eta\|^8
 =\E\left(\sum_{j=1}^n\eta_j^2\right)^4
 \le n^3\sum_{j=1}^n\E|\eta_j|^8\le Cn^4.
\end{equation}
Cauchy--Schwarz, which does not require independence of $X$ and $\eta$, therefore yields
\begin{equation}\label{eq:uniform-envelope-second-moment}
 \begin{aligned}
 \E Z_n^2
 &\le C\left(\E(1+T^4)^4\right)^{1/2}
         \left(\E(1+U^2)^4\right)^{1/2}
 \le Cn^2,\\
 \E[Z_n\ind_{\mathcal E_n^c}]
 &\le(\E Z_n^2)^{1/2}\Prob(\mathcal E_n^c)^{1/2}
 \le Cn\cdot n^{-6}=Cn^{-5}\le Cn^{-4}.
 \end{aligned}
\end{equation}
Adding this term to the common-event expectation estimates proves~\eqref{eq:maxpred}--\eqref{eq:maxloss} without indicators.

\noindent\textbf{Step 7: simplify the rates and justify the stochastic orders.}
For $M_n$ a fixed power of the logarithm, or a constant,
\begin{equation}\label{eq:uniform-count-scale-ratio}
 \frac{M_n^2}{pt_n}=\frac{M_n}{D_0ph_n}\longrightarrow0,
\end{equation}
because $ph_n\asymp\sqrt{n\log(en)}$. Substitute $t_n=D_0M_nh_n$:
\begin{equation}\label{eq:uniform-rate-expansions}
 \begin{aligned}
 h_nt_n(pt_n+M_n^2)
 &=D_0^2pM_n^2h_n^3+D_0M_n^3h_n^2
 =O\left(M_n^2\frac{\log^{3/2}(en)}{\sqrt n}\right),\\
 t_n^2(pt_n+M_n^2)
 &=D_0^3pM_n^3h_n^3+D_0^2M_n^4h_n^2
 =O\left(M_n^3\frac{\log^{3/2}(en)}{\sqrt n}\right),\\
 p^{-1}(pt_n+M_n^2)&=O(M_nh_n).
 \end{aligned}
\end{equation}
Let $a_n=M_n^{3/2}\log^{3/4}(en)n^{-1/4}$. Equations~\eqref{eq:uniform-local-maxima}, \eqref{eq:maxcountmean}, and~\eqref{eq:uniform-rate-expansions} give, for any $K_0>0$,
\begin{equation}\label{eq:uniform-displacement-markov}
 \begin{aligned}
 \Prob(\max_i\|\Delta_i\|>K_0a_n)
 &\le\Prob(\mathcal E_n^c)
    +\frac{\E[\max_i\|\Delta_i\|^2\ind_{\mathcal E_n}]}{K_0^2a_n^2}
 \le Cn^{-12}+\frac{C}{K_0^2},\\
 \Prob\left(\max_i p^{-1}\|P-P_i\|_F^2>K_0M_nh_n\right)
 &\le\Prob(\mathcal E_n^c)
    +\frac{\E[\max_i p^{-1}\|P-P_i\|_F^2\ind_{\mathcal E_n}]}{K_0M_nh_n}
 \le Cn^{-12}+\frac{C}{K_0}.
 \end{aligned}
\end{equation}
These are precisely~\eqref{eq:maxvector}--\eqref{eq:maxproj}. The common-event and moment arguments above also establish all stated bounded-score variants with constant $M_n$.
\end{proof}

\section{Discussion}\label{sec:discussion}

\subsection{Assumptions that remain essential}\label{sec:discussion-assumptions}

Several assumptions remain essential to the arguments developed in this paper. Extending the results beyond these assumptions, or alternatively constructing counterexamples showing that some of the conclusions fail when the assumptions are relaxed, remains an interesting direction for future work.

\paragraph{Deterministic regularization.}
All penalty functions and tuning parameters are deterministic in the stated results. The same penalty is used in the full and deleted objectives. Full-sample adaptive weights or a tuning parameter chosen from the observed data can destroy the deleted-data measurability used for Gaussian conditioning and block clamping. Such procedures need a separate analysis rather than a samplewise appeal to the fixed-penalty assumptions.

\paragraph{Block structure.}
The proof uses separability to clamp a single block without retaining dependence on its removed columns. It also uses activation at the block origin and a dimension-uniform tube-volume estimate. General fused penalties, nuclear-norm penalties, and blocks with unbounded dimensions are not covered as stated. Smoothness of the proximal map, rather than boundedness of the primal group-penalty Hessian, is what allows active blocks to approach zero.

\paragraph{Ridge and design.}
A ridge parameter bounded away from zero supplies the global strong-convexity and resolvent bounds, including on exceptional datasets. The constants in these arguments need not remain bounded as the ridge vanishes. Gaussianity is used both for held-out linear responses and for the conditional block densities after multiplying by the response likelihood. The result should not be read as a distribution-free statement for the design.

\paragraph{Moments and risk targets.}
The mean proof uses third moments through a conditional tail envelope and weighted exceptional-event bounds. Its apparent fourth-power term is truncated before expectation. This does not show that a third moment is necessary, nor that it is sufficient for every population-risk or tuning claim. A response-dependent additive term in the loss can cancel from LO differences even when it is not separately integrable. Population-risk consistency and uniform control over tuning parameters remain distinct from the fixed-parameter approximation theorem.

\subsection{Computational implications and further questions}\label{sec:discussion-computation}
Solving the restricted optimization problems in \eqref{eq:restricted}, while generally easier than solving the original leave-one-out problem in \eqref{eq:LOO_prob}, can still be computationally expensive, particularly because approximating LO requires solving a closely related problem for each of the \(n\) deleted observations. As discussed earlier, following the approach of \cite{rad2020scalable}, one can instead approximate the solution of each restricted problem by taking a single Newton step initialized at the full-data estimator.
For a nonquadratic loss or a curved active penalty, replacing the exact restricted minimizer by a one-step Newton approximation introduces an additional source of error. A complete theoretical guarantee for the fully computational ALO approximation would therefore need to control this Newton approximation error, as well as any optimization error in the original full-data fit. The deterministic comparison developed in this paper can serve as one component of such an analysis, but we do not establish a general bound for the Newton-step error here.
The analysis of this additional term is, however, closely related to the arguments developed in \cite{rad2020scalable}. Nevertheless, a complete analysis of the full ALO formula proposed in \cite{wang2018approximate} would require combining our restricted-fit comparison with a bound on this Newton approximation error, and such a result is beyond the scope of the present paper.

\section*{AI Disclosure}

The proof was initially obtained by Stellar Colosseum \cite{colosseum}, a fully automated agentic
system built on the Gemini~4 base model. The raw output of Stellar Colosseum is available at
\url{https://drive.google.com/file/d/1eD2DxOj4TtG_fh-NF6TnVel_O98lI6SQ/view?usp=sharing}, and a Lean~4 verification of the main theorem at \url{https://drive.google.com/file/d/1Ai19UYIojXzbCnWztgM9l-c8iUz50OK1/view?usp=sharing}. 
The authors verified the proof and edited it for
clarity of presentation; other models were used later to edit the text.


\bibliographystyle{alpha}
\bibliography{references}

\appendix
\section{Preliminary}\label{Proofs of preliminary results}

This section includes some useful results that we used in the proof.
\paragraph{Why the restricted domains are closed.}
The following proposition shows that passing to the closure preserves the infimum, while minimization over the open stratum need not attain it.

\begin{proposition}[Open strata and attainment]\label{prop:attainment}
For every fixed dataset,
\begin{equation}\label{eq:expanded-display-007}
\inf_{\mathcal M^\circ(b)}F_i=\min_{C(b)}F_i.
\end{equation}
The open-stratum problem attains its minimum if and only if $m_i^C\in\mathcal M^\circ(b)$, in which case its minimizer is $m_i^C$. The assumptions above do not guarantee attainment.
\end{proposition}

\begin{remark}
This proposition is an isolated observation motivating
the use of closed feasible domains. It shows that
minimization over the open stratum may fail to attain
its infimum, whereas passing to its closure preserves
the optimal value and ensures attainment. The active
subspace $A(b)$ is another closed alternative, generally
larger than $C(b)$. Neither this proposition nor its
counterexample is used in the subsequent proofs.
\end{remark}

\paragraph{A comparison valid on every dataset.}
The following lemma record comparison bounds that hold for every dataset and will be used to control the contribution of exceptional events.

\begin{lemma}[A global comparison for exceptional events]\label{lem:global}
Set
\begin{equation}\label{eq:localscore}
 S_i=1+|\ell'(y_i;x_i^\top b_i)|.
\end{equation}
For either restricted domain $D\in\{A,C\}$ defined in \eqref{eq:domains}, without any probabilistic event,
\begin{align}
 \|\Delta_i^D\|&\le\mu^{-1}S_i\|x_i\|,\label{eq:globaldisplacement}\\
 |x_i^\top\Delta_i^D|&\le\mu^{-1}S_i\|x_i\|^2,\label{eq:globalprediction}\\
 |\ell(y_i;x_i^\top b_i)-\ell(y_i;x_i^\top m_i^D)|
 &\le C S_i^2(\|x_i\|^2+\|x_i\|^4).\label{eq:globalloss}
\end{align}
\end{lemma}


\paragraph{Block proximal geometry.} Following lemma records the derivative and remainder bounds for $\Phi_g$, together with a Bregman identity for comparing two proximal points.

\begin{lemma}[Block proximal identities and examples]\label{lem:prox}
The map $\Phi_g$ is globally $\mu^{-1}$-Lipschitz, vanishes exactly on $K_g$, and on its active exterior
\begin{equation}\label{eq:expanded-display-008}
D\Phi_g(v)=\bigl(\mu I+\lambda_1\nabla^2r_g(\Phi_g(v))\bigr)^{-1},
 \qquad 0\prec D\Phi_g(v)\preceq\mu^{-1}I.
\end{equation}
Define $J_g(v)=D\Phi_g(v)$ off $K_g$ and $J_g(v)=0$ on $K_g$, including its topological boundary $\partial K_g$ in $\R^{d_g}$. If the segment $[v,v+u]$ avoids $\partial K_g$, then
\begin{equation}\label{eq:proxtaylor}
 \|\Phi_g(v+u)-\Phi_g(v)-J_g(v)u\|
 \le\frac{C_\Phi}{2}\|u\|^2.
\end{equation}
For an arbitrary segment the left side is at most $2\mu^{-1}\|u\|$. For $a=\Phi_g(v)$, $a'=\Phi_g(v')$, and $z=(v-\mu a)/\lambda_1$, $z'=(v'-\mu a')/\lambda_1$,
\begin{equation}\label{eq:proxenergy}
 \mu\|a-a'\|^2+\lambda_1D_{r_g}^{z'}(a,a')+
 \lambda_1D_{r_g}^{z}(a',a)=\langle v-v',a-a'\rangle,
\end{equation}
For a convex function $f$ and $\zeta\in\partial f(u_1)$, the Bregman divergence based at $u_1$ is $D_f^\zeta(u_0,u_1)=f(u_0)-f(u_1)-\langle\zeta,u_0-u_1\rangle$.
\end{lemma}

\paragraph{Resolvent identities without commutativity.} The loss Hessian and the proximal derivative need not commute. The following lemma provides the resolvent identities and bounds needed to construct and control the deleted-data linear responses.

\begin{lemma}[Noncommuting resolvents]\label{lem:resolvent}
If $H,J$ are symmetric positive semidefinite and $J\preceq\mu^{-1}I$, then
\begin{align}
 B&=(I+JH)^{-1}J=J(I+HJ)^{-1}
  =J^{1/2}(I+J^{1/2}HJ^{1/2})^{-1}J^{1/2},\label{eq:B}\\
 Q&=(I+HJ)^{-1}=I-HB,\qquad Q^\top=I-BH.\label{eq:Q}
\end{align}
In particular $0\preceq B\preceq J$ and $\|Q\|_{\op}\le1+\|H\|_{\op}/\mu$.
\end{lemma}

\begin{lemma}[Conditional Gaussian control]\label{lem:gaussian}
Recall that $\mathcal D_{-i}$ is the $\sigma$-field generated by all observations except observation $i$. On $\mathcal A_i$, $\|b_i\|\le C\sqrt n$, and the conditional probability given $\mathcal D_{-i}$ that the linear-response bound in~\eqref{eq:Gi} fails is at most $Cn^{-20}$. The row-norm bound fails with probability at most $Ce^{-cn}$. Choose $C_G$ to obtain the stated exponent, then choose $D_0$ below. Several later estimates use only the weaker bound $Cn^{-12}$.
\end{lemma}

\paragraph{Density bound from one-sided log-curvature.}
To establish anti-concentration near the activation
boundaries, we first need a uniform density bound
for the conditional clamped dual scores.
The following lemma provides such a bound for
convolutions of independent random vectors under
a one-sided log-curvature condition, without
requiring log-concavity.

\begin{lemma}[Density bound from one-sided log-curvature]
\label{lem:density}
Let $Z_1,\ldots,Z_m\in\R^d$ be independent random
vectors with positive $C^2$ densities $f_j$ satisfying
$-D^2\log f_j\preceq\gamma_j I$, where $\gamma_j>0$.
For real $\theta_j$ with $\sum_{j=1}^m\theta_j=1$,
set $\Gamma=\sum_{j=1}^m\gamma_j\theta_j^2$.
The convolution density $f$ of $\sum_{j=1}^m Z_j$
satisfies
\begin{equation}\label{eq:expanded-display-011}
\|f\|_\infty\le
\left(\frac{\Gamma}{2\pi}\right)^{d/2},
\end{equation}
where the norm is the essential supremum.
\end{lemma}

\paragraph{Volume near convex activation boundaries.}
The preceding density bound is combined with a
geometric estimate for neighborhoods of convex
boundaries. The following lemma shows that the
volume of such a neighborhood is proportional
to its width when the dimension and the
inner and outer radii are uniformly bounded.
Together, these two estimates are used to prove
Lemma~\ref{lem:antic}.

\begin{lemma}[Tube volume]\label{lem:tube}
Let $K\subseteq\R^d$ be closed and convex, with
$B(0,r_0)\subseteq K\subseteq B(0,R_0)$
for $0<r_0\le R_0<\infty$, where
$B(0,r)=\{x\in\R^d:\|x\|\le r\}$.
Then, for $0<\varepsilon<r_0$,
\begin{equation}\label{eq:expanded-display-012}
\{z\in\R^d:\dist(z,\partial K)\le\varepsilon\}
\subseteq
(1+\varepsilon/r_0)K
\setminus
\operatorname{int}((1-\varepsilon/r_0)K),
\end{equation}
where $\partial K$ and $\operatorname{int}K$ denote
the topological boundary and interior of $K$ in $\R^d$,
respectively.
Let $\operatorname{vol}$ denote $d$-dimensional Lebesgue measure.
For bounded $d,r_0^{-1},R_0$, the tube volume satisfies
\begin{equation*}
\operatorname{vol}
\bigl(\{z\in\R^d:\dist(z,\partial K)\le\varepsilon\}\bigr)
\le C\varepsilon,
\end{equation*}
where $C$ depends only on these bounds.
\end{lemma}

\paragraph{Control of irregular deleted datasets.}
Lemma~\ref{lem:gaussian} controls the held-out
Gaussian responses when the deleted data satisfy
$\mathcal A_i$. We also need to control the
contribution from datasets where $\mathcal A_i$
fails.

\begin{lemma}[Weighted control of bad deleted datasets]\label{lem:baddeleted}
Uniformly in $i$,
\begin{align}
 \Prob(\mathcal A_i^c)&\le Cn^{-1/2},\label{eq:badprob}\\
 \E\left[\left(1+\frac{\|b_i\|^2}{p}\right)\ind_{\mathcal A_i^c}\right]
 &\le Cn^{-1/2}.\label{eq:weightedbad}
\end{align}
\end{lemma}

\section{Proofs for Section~\ref{Proofs of preliminary results}}

\subsection{Proof of Lemma~\ref{lem:global}}

\begin{proof}
Fix one of the two domains and abbreviate its optimizer by $m_i$. Since $b$ is feasible, exact minimization gives
\begin{equation}\label{eq:global-feasibility}
 F_i(m_i)\le F_i(b).
\end{equation}
Use $-c_ix_i\in\partial F_i(b)$ in the strong-convexity inequality at $b$:
\begin{equation}\label{eq:global-gap-square}
 \begin{aligned}
 F_i(b)-F_i(b_i)
 &\le c_ix_i^\top(b_i-b)-\frac\mu2\|b_i-b\|^2\\
 &=\frac{c_i^2\|x_i\|^2}{2\mu}
   -\frac\mu2\left\|b_i-b-\frac{c_i}{\mu}x_i\right\|^2
 \le\frac{c_i^2\|x_i\|^2}{2\mu}.
 \end{aligned}
\end{equation}
On the other hand, $0\in\partial F_i(b_i)$, so
\begin{equation}\label{eq:global-gap-lower}
 \frac\mu2\|m_i-b_i\|^2
 \le F_i(m_i)-F_i(b_i)
 \le F_i(b)-F_i(b_i).
\end{equation}
Combining~\eqref{eq:global-gap-square}
and~\eqref{eq:global-gap-lower},
and then using the score-contraction bound
in~\eqref{eq:scorecontraction}, gives
\begin{equation}\label{eq:global-displacement-detail}
 \|\Delta_i\|\le\frac{|c_i|}{\mu}\|x_i\|
 \le\frac{|\ell'(y_i;x_i^\top b_i)|}{\mu}\|x_i\|
 \le\frac{S_i}{\mu}\|x_i\|.
\end{equation}
This proves~\eqref{eq:globaldisplacement}. Pairing with $x_i$ proves~\eqref{eq:globalprediction}.

Put $z_i=x_i^\top b_i$ and $a_i=x_i^\top\Delta_i$, so $x_i^\top m_i=z_i-a_i$. Taylor's integral identity and the curvature upper bound give
\begin{equation}\label{eq:global-loss-taylor}
 \begin{aligned}
 \ell(y_i;z_i-a_i)-\ell(y_i;z_i)
 &=-a_i\ell'(y_i;z_i)
   +a_i^2\int_0^1(1-s)\ell''(y_i;z_i-sa_i)\,ds,\\
 |\ell(y_i;z_i-a_i)-\ell(y_i;z_i)|
 &\le S_i|a_i|+\frac{L_2}{2}|a_i|^2\\
 &\le\frac{S_i^2}{\mu}\|x_i\|^2
       +\frac{L_2S_i^2}{2\mu^2}\|x_i\|^4.
 \end{aligned}
\end{equation}
The lower bound on $\mu$ (Assumption~\ref{ass:hperparam}) makes the last expression at most the right side of~\eqref{eq:globalloss}. No independence, support agreement, or probabilistic event has been used.
\end{proof}

\subsection{Proof of Proposition~\ref{eq:expanded-display-007}}
\begin{proof}
\noindent\textbf{Step 1: equality of the two optimal values.}
Let $m=m_i^C$. Recall that $\text{cl}$ denotes closure. Since $m\in C(b)=\cl\mathcal M^\circ(b)$,
there exists a sequence $(a_k)_{k\ge1}\subseteq\mathcal M^\circ(b)$ such that $a_k\to m$. Continuity of $F_i$ gives $F_i(a_k)\to F_i(m)$. Since $\mathcal M^\circ(b)\subseteq C(b)$,
\begin{equation}\label{eq:attainment-infima-detail}
\min_{\beta\in C(b)}F_i(\beta)
\le
\inf_{\beta\in\mathcal M^\circ(b)}F_i(\beta)
\le
\lim_{k\to\infty}F_i(a_k)
=
F_i(m)
=
\min_{\beta\in C(b)}F_i(\beta).
\end{equation}
Consequently,
\begin{equation*}
\inf_{\beta\in\mathcal M^\circ(b)}F_i(\beta)=F_i(m).
\end{equation*}
If $m\in\mathcal M^\circ(b)$, this infimum is attained at $m$. Conversely, if $u\in\mathcal M^\circ(b)$ satisfies
\begin{equation*}
F_i(u)=\inf_{\beta\in\mathcal M^\circ(b)}F_i(\beta),
\end{equation*}
then $u\in C(b)$ and $F_i(u)=F_i(m)$, so uniqueness of the minimizer on $C(b)$ gives $u=m$. Thus
\begin{equation*}
\left\{
u\in\mathcal M^\circ(b):
F_i(u)=\inf_{\beta\in\mathcal M^\circ(b)}F_i(\beta)
\right\}
=
\begin{cases}
\{m\},&m\in\mathcal M^\circ(b),\\
\varnothing,&m\notin\mathcal M^\circ(b).
\end{cases}
\end{equation*}

\noindent\textbf{Step 2: an admissible example without attainment.}
Take
\begin{equation*}
(x_j,y_j)\overset{\mathrm{i.i.d.}}{\sim}
N(0,I_p/p)\otimes N(0,1),
\qquad j=1,\ldots,n,
\qquad
\beta^*=0,
\end{equation*}
and set
\begin{equation*}
\ell(y;z)=-z,
\qquad
G=p,
\qquad
d_g=1,
\qquad
r_g(t)=|t|
\quad(g=1,\ldots,p),
\qquad
r(\beta)=\sum_{g=1}^p|\beta_g|.
\end{equation*}
Keep $\lambda_1$ and $\mu=2\lambda_2$ in their prescribed deterministic parameter ranges. Write
\begin{equation*}
0<\underline{\lambda}_1\le\lambda_1\le\overline{\lambda}_1<\infty,
\qquad
\delta_-\le\frac np\le\delta_+,
\end{equation*}
where $\underline{\lambda}_1$ and $\overline{\lambda}_1$ are fixed bounds for the permitted values of $\lambda_1$.

The function $\ell(y;\cdot)$ is finite and affine, and
\begin{equation}\label{eq:attainment-example-assumptions}
\ell'(y;z)=-1,
\qquad
\ell''(y;z)=0,
\qquad
\eta_j=-1,
\qquad
\E|\eta_j|^3=\E\eta_j^2=1.
\end{equation}
Define the response density for every $(y,u)\in\mathbb R^2$ by
\begin{equation*}
q(y\mid u)=\frac{1}{\sqrt{2\pi}}\exp\left(-\frac{y^2}{2}\right).
\end{equation*}
Then $q(y\mid u)>0$ on the common support $\mathbb R$, and
\begin{equation*}
\partial_u^2\log q(y\mid u)=0,
\qquad (y,u)\in\mathbb R^2.
\end{equation*}
For each $g\in\{1,\ldots,p\}$,
\begin{equation*}
\partial r_g(0)=[-1,1],
\qquad
K_g=\lambda_1\partial r_g(0)=[-\lambda_1,\lambda_1],
\end{equation*}
and the map $\Phi_g$ defined earlier is
\begin{equation*}
\Phi_g(v)
=
\frac{\operatorname{sign}(v)}{\mu}
(|v|-\lambda_1)_+,
\end{equation*}
where we denote $u_+:=\max\{u,0\}$. For $v\notin K_g$,
\begin{equation*}
\Phi_g'(v)=\mu^{-1},
\qquad
\Phi_g''(v)=0.
\end{equation*}
In particular,
\begin{equation*}
\sup_{1\le g\le p}\sup_{v\notin K_g}|\Phi_g''(v)|=0,
\end{equation*}
which verifies the required proximal smoothness bound.

Fix the deletion index $i=1$ and define
\begin{equation*}
A_k=\sum_{j=2}^n x_{jk},
\qquad
B_k=x_{1k},
\qquad k=1,\ldots,p.
\end{equation*}
The objectives are
\begin{equation*}
\begin{aligned}
F(\beta)
&=
\sum_{k=1}^p
\left\{
\frac{\mu}{2}\beta_k^2
+\lambda_1|\beta_k|
-(A_k+B_k)\beta_k
\right\},\\
F_1(\beta)
&=
\sum_{k=1}^p
\left\{
\frac{\mu}{2}\beta_k^2
+\lambda_1|\beta_k|
-A_k\beta_k
\right\}.
\end{aligned}
\end{equation*}
Minimizing each summand gives
\begin{equation}\label{eq:attainment-soft-threshold}
b_k
=
\frac{\operatorname{sign}(A_k+B_k)}{\mu}
(|A_k+B_k|-\lambda_1)_+,
\qquad
b_{1,k}
=
\frac{\operatorname{sign}(A_k)}{\mu}
(|A_k|-\lambda_1)_+.
\end{equation}

Assume $p>\underline{\lambda}_1^{-2}$, so that $\lambda_1-p^{-1/2}>0$. For each $k$, define
\begin{equation}\label{eq:attainment-coordinate-event}
H_k
=
\{\lambda_1-p^{-1/2}<A_k<\lambda_1,\ B_k>p^{-1/2}\}.
\end{equation}
On $H_k$,
\begin{equation*}
A_k+B_k>\lambda_1,
\qquad
0<A_k<\lambda_1,
\qquad
b_k>0,
\qquad
b_{1,k}=0.
\end{equation*}
In this singleton-block setting,
\begin{equation*}
\mathcal M^\circ(b)
=
\left\{
\beta\in\mathbb R^p:
\begin{array}{ll}
\beta_h=0&\text{if }b_h=0,\\
b_h\beta_h>0&\text{if }b_h\ne0
\end{array}
\text{ for every }h=1,\ldots,p
\right\}.
\end{equation*}
Thus, on $H_k$, every $\beta\in\mathcal M^\circ(b)$ satisfies $\beta_k>0$. For $t>0$, the $k$th summand of $F_1$ is
\begin{equation}\label{eq:attainment-ray-objective}
f_k(t)
=
\frac{\mu}{2}t^2+(\lambda_1-A_k)t,
\qquad
f_k'(t)
=
\mu t+\lambda_1-A_k>0
\quad\text{on }H_k.
\end{equation}
For any $\beta\in\mathcal M^\circ(b)$, define $\beta^{(k)}\in\mathbb R^p$ by
\begin{equation*}
\beta_h^{(k)}
=
\begin{cases}
\beta_k/2,&h=k,\\
\beta_h,&h\ne k,
\end{cases}
\qquad h=1,\ldots,p.
\end{equation*}
On $H_k$, we have
\begin{equation*}
\beta^{(k)}\in\mathcal M^\circ(b),
\qquad
F_1(\beta^{(k)})-F_1(\beta)
=
f_k(\beta_k/2)-f_k(\beta_k)<0.
\end{equation*}
Hence no $\beta\in\mathcal M^\circ(b)$ attains $\inf_{\mathcal M^\circ(b)}F_1$ on $H_k$. By Step 1,
\begin{equation*}
H_k\subseteq\{m_1^C\notin\mathcal M^\circ(b)\}.
\end{equation*}

The Gaussian design gives
\begin{equation*}
(A_k,B_k)\overset{\mathrm{i.i.d.}}{\sim}
N\left(0,\frac{n-1}{p}\right)\otimes
N\left(0,\frac1p\right),
\qquad k=1,\ldots,p.
\end{equation*}
Since $n\ge2$,
\begin{equation*}
\frac{\delta_-}{2}
\le
\frac{n-1}{p}
\le
\delta_+.
\end{equation*}
Define the fixed positive constants
\begin{equation*}
c
=
\frac{1}{\sqrt{2\pi\delta_+}}
\exp\left(-\frac{\overline{\lambda}_1^2}{\delta_-}\right),
\qquad
c'
=
c\int_1^\infty
\frac{1}{\sqrt{2\pi}}\exp\left(-\frac{v^2}{2}\right)\,dv.
\end{equation*}
For every $t\in(\lambda_1-p^{-1/2},\lambda_1)\subseteq(0,\overline{\lambda}_1)$,
\begin{equation*}
\sqrt{\frac{p}{2\pi(n-1)}}
\exp\left(-\frac{pt^2}{2(n-1)}\right)
\ge c.
\end{equation*}
Independence of $A_k$ and $B_k$ therefore gives
\begin{equation}\label{eq:attainment-event-probability}
\begin{aligned}
\Prob(H_k)
&=
\Prob(\lambda_1-p^{-1/2}<A_k<\lambda_1)
\Prob(\sqrt p\,B_k>1)\\
&\ge
cp^{-1/2}
\int_1^\infty
\frac{1}{\sqrt{2\pi}}\exp\left(-\frac{v^2}{2}\right)\,dv\\
&=
c'p^{-1/2}.
\end{aligned}
\end{equation}
Since the pairs $(A_k,B_k)$ are independent across $k$, the events $H_1,\ldots,H_p$ are independent. Consequently,
\begin{equation}\label{eq:attainment-failure-probability}
\begin{aligned}
\Prob\bigl(m_1^C\notin\mathcal M^\circ(b)\bigr)
&\ge
\Prob\left(\bigcup_{k=1}^p H_k\right)\\
&=
1-\prod_{k=1}^p(1-\Prob(H_k))\\
&\ge
1-\exp\left(-\sum_{k=1}^p\Prob(H_k)\right)\\
&\ge
1-\exp(-c'\sqrt p).
\end{aligned}
\end{equation}
Thus the stated assumptions admit a model for which
\begin{equation*}
\Prob\bigl(m_1^C\notin\mathcal M^\circ(b)\bigr)\longrightarrow1
\qquad(p\longrightarrow\infty).
\end{equation*}
By Step 1, this is the probability that $\inf_{\beta\in\mathcal M^\circ(b)}F_1(\beta)$ is not attained. The example uses $\ell(y;z)=-z$ and makes no assertion about attainment for $\ell(y;z)=\tfrac12(y-z)^2$.
\end{proof}

\subsection{Proof of Lemma~\ref{lem:prox}}

\begin{proof}
\noindent\textbf{Step 1: proximal optimality and the energy identity.}
For each input $v$, the strictly convex, coercive function
\begin{equation}\label{eq:prox-primal-objective}
 a\longmapsto\frac\mu2\|a\|^2+\lambda_1r_g(a)-\langle v,a\rangle
\end{equation}
has a unique minimizer. Its KKT equation is
\begin{equation}\label{eq:prox-optimality-expanded}
 v=\mu a+\lambda_1z,\qquad z\in\partial r_g(a).
\end{equation}
Thus this minimizer is exactly $\Phi_g(v)$. For two inputs, subtract the equations and pair with $a-a'$:
\begin{equation}\label{eq:prox-energy-expanded}
 \begin{aligned}
 \langle v-v',a-a'\rangle
 &=\mu\|a-a'\|^2+\lambda_1\langle z-z',a-a'\rangle,\\
 D_{r_g}^{z'}(a,a')+D_{r_g}^{z}(a',a)
 &=r_g(a)-r_g(a')-\langle z',a-a'\rangle\\
 &\quad+r_g(a')-r_g(a)-\langle z,a'-a\rangle\\
 &=\langle z-z',a-a'\rangle.
 \end{aligned}
\end{equation}
This proves~\eqref{eq:proxenergy}. Both Bregman divergences are nonnegative by the subgradient inequality. Hence
\begin{equation}\label{eq:prox-lipschitz-expanded}
 \mu\|a-a'\|^2\le\langle v-v',a-a'\rangle
 \le\|v-v'\|\|a-a'\|,
 \qquad
 \|\Phi_g(v)-\Phi_g(v')\|\le\mu^{-1}\|v-v'\|.
\end{equation}
At $a=0$, equation~\eqref{eq:prox-optimality-expanded} is equivalent to $v\in\lambda_1\partial r_g(0)=K_g$. Uniqueness therefore gives
\begin{equation}\label{eq:prox-zero-set-expanded}
 \Phi_g(v)=0\quad\Longleftrightarrow\quad v\in K_g.
\end{equation}

For a singleton block, write $K_g=[\kappa_g^-,\kappa_g^+]$. Monotonicity of $\partial r_g$ between $a$ and zero gives, for any $z_0\in\partial r_g(0)$, $(z-z_0)a\ge0$. Choose the largest such $z_0$ when $a>0$ and the smallest when $a<0$. Equation~\eqref{eq:prox-optimality-expanded} then yields
\begin{equation}\label{eq:prox-scalar-sign-components}
 a>0\ \Longrightarrow\ v\ge\mu a+\kappa_g^+>\kappa_g^+,
 \qquad
 a<0\ \Longrightarrow\ v\le\mu a+\kappa_g^-<\kappa_g^-.
\end{equation}
Together with the zero-set characterization, this also proves the converse sign assignments for the two exterior intervals.

\noindent\textbf{Step 2: the derivative on the active exterior.}
For $a\ne0$, define $T_g(a)=\mu a+\lambda_1\nabla r_g(a)$. Convexity of $r_g$ implies
\begin{equation}\label{eq:prox-forward-derivative}
 DT_g(a)=\mu I+\lambda_1\nabla^2r_g(a)\succeq\mu I\succ0.
\end{equation}
The map is injective: if $T_g(a)=T_g(a')$, strong monotonicity in~\eqref{eq:prox-lipschitz-expanded} forces $a=a'$. Every $v\notin K_g$ has a nonzero proximal solution, and that solution satisfies $T_g(a)=v$. Thus $T_g$ maps the nonzero coefficient region bijectively onto the active exterior. Since $DT_g(a)$ is invertible for every $a\neq0$, applying the inverse function theorem to $T_g$ at each nonzero $a$ gives its inverse $\Phi_g$ is $C^2$ on $\mathbb R^{d_g}\setminus K_g$, with
\begin{equation}\label{eq:prox-nverse-derivative}
 D\Phi_g(v)=\bigl[DT_g(\Phi_g(v))\bigr]^{-1}
 =\bigl(\mu I+\lambda_1\nabla^2r_g(\Phi_g(v))\bigr)^{-1},
 \qquad 0\prec D\Phi_g(v)\preceq\mu^{-1}I.
\end{equation}
The inverse function is $C^2$ because $r_g$ is $C^3$ off zero.

\noindent\textbf{Step 3: Taylor remainders, with the boundary cases separated.}
If the segment $\{v+su:0\le s\le1\}$ is contained in the active exterior, the integral Taylor formula gives
\begin{equation}\label{eq:prox-taylor-integral}
 \Phi_g(v+u)-\Phi_g(v)-D\Phi_g(v)u
 =\int_0^1(1-s)D^2\Phi_g(v+su)[u,u]\,ds.
\end{equation}
By the bounds on $\|D^2\Phi_g\|$ given by \eqref{eq:bound_on_d2_phi_g}, its norm is at most $C_\Phi\|u\|^2/2$. If the segment lies in the interior of $K_g$, both proximal values and $J_g(v)$ vanish. A segment avoiding $\partial K_g$ must be in one of these two regions: otherwise continuity of the segment would force a boundary intersection. This proves~\eqref{eq:proxtaylor}.

For a segment that may meet the boundary, no differentiability along the segment is asserted. Instead, the global Lipschitz bound and the definition of $J_g$ imply
\begin{equation}\label{eq:prox-arbitrary-segment}
 \begin{aligned}
 \|\Phi_g(v+u)-\Phi_g(v)-J_g(v)u\|
 &\le\|\Phi_g(v+u)-\Phi_g(v)\|+\|J_g(v)u\|\\
 &\le2\mu^{-1}\|u\|.
 \end{aligned}
\end{equation}
This includes inputs exactly on the activation boundary, where $J_g$ was defined to be zero.

\noindent\textbf{Step 4: norm-plus-quadratic blocks.}
For $r_g(a)=w_g\|a\|+\gamma_g\|a\|^2/2$, put $a_g=\mu+\lambda_1\gamma_g$ and $k_g=\lambda_1w_g$. At a nonzero solution, optimality reads
\begin{equation}\label{eq:prox-norm-optimality}
 v=a_g a+k_g\frac{a}{\|a\|}
   =\left(a_g+\frac{k_g}{\|a\|}\right)a.
\end{equation}
Thus $v$ and $a$ have the same direction and $\|v\|=a_g\|a\|+k_g$. At zero the optimality condition is $\|v\|\le k_g$. Solving these two cases yields
\begin{equation}\label{eq:prox-norm-formula}
 \Phi_g(v)=
 \begin{cases}
 a_g^{-1}(1-k_g/\|v\|)v,&\|v\|>k_g,\\
 0,&\|v\|\le k_g.
 \end{cases}
\end{equation}
For $r=\|v\|>k_g$ and $e=v/r$, differentiation gives
\begin{equation}\label{eq:prox-norm-jacobian}
 D\Phi_g(v)=a_g^{-1}\bigl[(1-k_g/r)I+(k_g/r)ee^\top\bigr].
\end{equation}
For a direction $h$, set $h_\parallel=e^\top h$ and $h_\perp=h-h_\parallel e$. Then
\begin{equation}\label{eq:prox-radius-direction-derivatives}
 Dr[h]=h_\parallel,\qquad De[h]=r^{-1}h_\perp,
\end{equation}
and differentiating~\eqref{eq:prox-norm-jacobian} in that direction gives the explicit matrix
\begin{equation}\label{eq:prox-norm-second-derivative}
 D(D\Phi_g)(v)[h]
 =\frac{k_g}{a_g r^2}
   \left[h_\parallel(I-ee^\top)+h_\perp e^\top+e h_\perp^\top\right].
\end{equation}
Since $\|e\|=1$, $|h_\parallel|\le\|h\|$, and $\|h_\perp\|\le\|h\|$, it follows that
\begin{equation}\label{eq:prox-norm-uniform-second}
 \|D(D\Phi_g)(v)[h]\|_{\op}
 \le\frac{3k_g}{a_g r^2}\|h\|
 \le\frac{3}{a_gk_g}\|h\|.
\end{equation}
By the definition of the bilinear norm and~\eqref{eq:prox-norm-uniform-second},
\begin{equation*}
\|D^2\Phi_g(v)\|=\sup_{\substack{h\in\R^{d_g}\\ \|h\|=1}}\|D(D\Phi_g)(v)[h]\|_{\op}\le \frac{3}{a_gk_g}.
\end{equation*}
The uniform positive lower bounds on $a_g$ and $k_g$ therefore yield \eqref{eq:bound_on_d2_phi_g} with constant $C_\Phi$ independent of $g,n,p$. This bound remains uniform as $r\downarrow k_g$ and $\|\Phi_g(v)\|\downarrow0$ from the active exterior.
In dimension one, $I-ee^\top=0$ and $h_\perp=0$, so the second derivative is identically zero on each active branch.
\end{proof}

\subsection{Proof of Lemma~\ref{lem:resolvent}}

\begin{proof}
The argument first constructs the candidate operators and verifies inverses by multiplication; it does not assume that $J$ is invertible. Set
\begin{equation}\label{eq:resolvent-candidate}
 T=J^{1/2}HJ^{1/2},\qquad
 B_0=J^{1/2}(I+T)^{-1}J^{1/2}.
\end{equation}
The matrix $T$ is symmetric positive semidefinite, so $I+T\succ0$. Associativity, without exchanging the order of $H$ and $J$, gives
\begin{equation}\label{eq:resolvent-intertwining}
 (I+JH)J^{1/2}=J^{1/2}(I+T),\qquad
 J^{1/2}(I+HJ)=(I+T)J^{1/2}.
\end{equation}
Consequently
\begin{equation}\label{eq:resolvent-candidate-products}
 (I+JH)B_0=J,\qquad B_0(I+HJ)=J.
\end{equation}
In particular, $B_0+JHB_0=J$. Use this identity to check two right inverses explicitly:
\begin{equation}\label{eq:resolvent-inverse-products}
 \begin{aligned}
 (I+JH)(I-B_0H)
 &=I+JH-B_0H-JHB_0H\\
 &=I+(J-B_0-JHB_0)H=I,\\
 (I+HJ)(I-HB_0)
 &=I+HJ-HB_0-HJHB_0\\
 &=I+H(J-B_0-JHB_0)=I.
 \end{aligned}
\end{equation}
A square matrix with a right inverse is invertible. Thus
\begin{equation}\label{eq:resolvent-inverses-expanded}
 (I+JH)^{-1}=I-B_0H,\qquad
 (I+HJ)^{-1}=I-HB_0.
\end{equation}
Multiplying the two identities in~\eqref{eq:resolvent-candidate-products} by their corresponding inverses yields
\begin{equation}\label{eq:resolvent-identification}
 (I+JH)^{-1}J=B_0=J(I+HJ)^{-1}.
\end{equation}
This proves~\eqref{eq:B}, with $B=B_0$. Since $B_0$ is symmetric, transposing $I-HB_0$ gives $I-B_0H$ and proves both identities in~\eqref{eq:Q}.

Finally, the eigenvalues of $(I+T)^{-1}$ lie in $(0,1]$, so congruence by $J^{1/2}$ gives
\begin{equation}\label{eq:resolvent-order-and-norm}
 0\preceq B\preceq J\preceq\mu^{-1}I,
 \qquad
 \|Q\|_{\op}=\|I-HB\|_{\op}
 \le1+\|H\|_{\op}\|B\|_{\op}
 \le1+\frac{\|H\|_{\op}}\mu.
\end{equation}
All products remain valid on inactive blocks, where the corresponding block of $J$ is zero.
\end{proof}

\subsection{Proof of Lemma~\ref{lem:gaussian}}

\begin{proof}
\noindent\textbf{Step 1: deleted-data measurability and deterministic operator bounds.}
The fit $b_i$, its dual score, and its loss Hessian are functions of $\mathcal D_{-i}=\sigma((y_j,x_j):j\neq i)$. The map $J_g$ is continuous on the active exterior and is defined to be zero on the closed set $K_g$, so it is Borel measurable. Matrix inversion is continuous on the set of invertible matrices. Lemma~\ref{lem:resolvent} therefore makes $J_i,B_i,Q_i$ measurable with respect to $\mathcal D_{-i}$.

On $\mathcal A_i$, Lemma~\ref{lem:stability} and the norm and energy bounds give
\begin{equation}\label{eq:gaussian-fit-size}
 \|b_i\|\le\mu^{-1}\|X_i\|_{\op}\|\eta_{-i}\|
 \le\mu^{-1}K_X\sqrt{(n-1)2v_+}\le C\sqrt n.
\end{equation}
Writing the Hessian as a weighted Gram matrix,
\begin{equation}\label{eq:gaussian-hessian-bounds}
 H_i=X_i^\top\diag\{\ell''(y_j;x_j^\top b_i):j\ne i\}X_i,
 \qquad 0\preceq H_i\preceq L_2X_i^\top X_i,
\end{equation}
yields
\begin{equation}\label{eq:gaussian-operator-bounds}
 \|H_i\|_{\op}\le L_2K_X^2,\quad
 \|B_i\|_{\op}\le\mu^{-1},\quad
 \|Q_i\|_{\op}\le1+L_2K_X^2/\mu,\quad
 \|X_iB_i\|_{\op}\le K_X/\mu.
\end{equation}
These constants are uniform over all deleted datasets in $\mathcal A_i$.

\noindent\textbf{Step 2: conditional Gaussian variances and the coordinate union bound.}
Independence of the observation rows implies $x_i\mid\mathcal D_{-i}\sim N(0,\Sigma)$. For a deterministic coordinate vector $e_k$ in $\R^p$, and a retained-row coordinate vector $e_j$ in $\R^{n-1}$, conditional variances satisfy
\begin{equation}\label{eq:gaussian-coordinate-variances}
 \begin{aligned}
 \operatorname{Var}(e_k^\top Q_ix_i\mid\mathcal D_{-i})
 &=e_k^\top Q_i\Sigma Q_i^\top e_k
 \le\frac{C_\Sigma}{p}\|Q_i\|_{\op}^2\le\frac Cn,\\
 \operatorname{Var}(e_j^\top X_iB_ix_i\mid\mathcal D_{-i})
 &=e_j^\top X_iB_i\Sigma B_iX_i^\top e_j
 \le\frac{C_\Sigma}{p}\|X_iB_i\|_{\op}^2\le\frac Cn.
 \end{aligned}
\end{equation}
The final inequalities use $p\asymp n$ and~\eqref{eq:gaussian-operator-bounds}. A centered Gaussian variable with variance at most $C/n$ has tail at most $2\exp(-ns^2/(2C))$ at threshold $s$. There are at most $p+n-1\le C'n$ scalar coordinates in these two vectors. Thus, for a fixed constant $A>0$,
\begin{equation}\label{eq:gaussian-coordinate-union}
 \begin{aligned}
 &\Prob\left(\max\{\|Q_ix_i\|_\infty,\|X_iB_ix_i\|_\infty\}>Ah_n
                      \,\middle|\,\mathcal D_{-i}\right)\\
 &\hspace{2em}\le 2(p+n-1)\exp\left(-\frac{A^2}{2C}\log(en)\right)
 \le C'n(en)^{-A^2/(2C)}
 \quad\text{on }\mathcal A_i.
 \end{aligned}
\end{equation}
Every block contains at most $d_0$ scalar coordinates, so
\begin{equation}\label{eq:gaussian-block-conversion}
 \|Q_ix_i\|_{2,\infty}+\|X_iB_ix_i\|_\infty
 \le(\sqrt{d_0}+1)
      \max\{\|Q_ix_i\|_\infty,\|X_iB_ix_i\|_\infty\}.
\end{equation}
Choose $A$ with $A^2/(2C)\ge22$ and then $C_G\ge(\sqrt{d_0}+1)A$. Equations~\eqref{eq:gaussian-coordinate-union}--\eqref{eq:gaussian-block-conversion} give conditional failure probability at most $Cn^{-20}$. No conditioning on the response $y_i$, the held-out score, or a selected boundary set has occurred.

\noindent\textbf{Step 3: the held-out row norm.}
Write $x_i=\Sigma^{1/2}z$, where $z\sim N(0,I_p)$. Then
\begin{equation}\label{eq:gaussian-row-domination}
 \|x_i\|^2\le\frac{C_\Sigma}{p}\|z\|^2,
 \qquad \E\exp(\|z\|^2/4)=2^{p/2}.
\end{equation}
Exponential Markov inequality gives, for any $r>0$,
\begin{equation}\label{eq:gaussian-row-tail}
 \Prob(\|x_i\|>r)
 \le\exp\left[-\frac{pr^2}{4C_\Sigma}+\frac p2\log2\right].
\end{equation}
Choose $K_x$ sufficiently large so that the exponent at $r=K_x$ is at most $-cp$, which is at most $-c'n$. This proves the asserted row-norm bound. We will also use the following consequences: for every fixed $m>0$,
\begin{equation}\label{eq:gaussian-row-moments}
 \E\|x_i\|^m\le C_m,
 \qquad
 \E[\|x_i\|^m\ind_{\{\|x_i\|>K_x\}}]\le C_me^{-c_m n}.
\end{equation}
For the second bound, integrate~\eqref{eq:gaussian-row-tail} using
\begin{equation}\label{eq:gaussian-restricted-tail-integral}
 \E[R^m\ind_{\{R>K_x\}}]
 =K_x^m\Prob(R>K_x)+\int_{K_x}^\infty mr^{m-1}\Prob(R>r)\,dr,
 \qquad R=\|x_i\|.
\end{equation}
For $r\ge K_x$, the exponent is bounded by $-cpr^2$ after increasing $K_x$, so the integral is exponentially small. Integrating from a fixed threshold, and bounding the part below that threshold deterministically, proves the first bound in~\eqref{eq:gaussian-row-moments}.
\end{proof}

\subsection{Proof of Lemma~\ref{lem:density}}

\begin{proof}
Let $m\ge1$ be the number of summands, and recall that
\begin{equation*}
\theta_j\in\mathbb R,\qquad \sum_{j=1}^m\theta_j=1,\qquad
\Gamma=\sum_{j=1}^m\gamma_j\theta_j^2.
\end{equation*}
All integrals below are Lebesgue integrals. Their nonnegative integrands allow the use of Tonelli's theorem before finiteness has been established.

\noindent\textbf{Step 1: apply the one-sided curvature inequality before convolution.}
The hypothesis is equivalent to $D^2\log f_j\succeq-\gamma_j I$. For any $z_j,h\in\R^d$, the integral Taylor formula gives
\begin{equation}\label{eq:density-log-taylor}
 \begin{aligned}
 \log f_j(z_j\pm\theta_jh)
 &=\log f_j(z_j)\pm\theta_j\nabla\log f_j(z_j)^\top h\\
 &\quad+\theta_j^2\int_0^1(1-s)
       h^\top D^2\log f_j(z_j\pm s\theta_jh)h\,ds\\
 &\ge\log f_j(z_j)\pm\theta_j\nabla\log f_j(z_j)^\top h
            -\frac{\gamma_j\theta_j^2}{2}\|h\|^2.
 \end{aligned}
\end{equation}
For $z=(z_1,\ldots,z_m)$, define
\begin{equation*}
\mathfrak p(z)=\prod_{j=1}^m f_j(z_j),\qquad
\mathfrak a(z,h)=\sum_{j=1}^m\theta_j\nabla\log f_j(z_j)^\top h.
\end{equation*}
Summing~\eqref{eq:density-log-taylor}, exponentiating, and averaging over the two signs gives
\begin{equation}\label{eq:density-product-cosh}
 \begin{aligned}
 \frac{\prod_jf_j(z_j+\theta_jh)+\prod_jf_j(z_j-\theta_jh)}2
 &\ge \mathfrak p(z)e^{-\Gamma\|h\|^2/2}\frac{e^{\mathfrak a(z,h)}+e^{-\mathfrak a(z,h)}}2\\
 &=\mathfrak p(z)e^{-\Gamma\|h\|^2/2}\cosh(\mathfrak a(z,h))\\
 &\ge \mathfrak p(z)e^{-\Gamma\|h\|^2/2}.
 \end{aligned}
\end{equation}
There is no upper-curvature condition in this calculation.

\noindent\textbf{Step 2: write the convolution fiber in coordinates.}
For $m\ge2$, define a version of the convolution density by
\begin{equation}\label{eq:density-convolution-coordinates}
 f(t)=\int_{\R^{d(m-1)}}
     \left(\prod_{j=1}^{m-1}f_j(z_j)\right)
     f_m\left(t-\sum_{j=1}^{m-1}z_j\right)
     \,dz_1\cdots dz_{m-1}.
\end{equation}
The linear change of variables from $(z_1,\ldots,z_m)$ to $(z_1,\ldots,z_{m-1},t=\sum_jz_j)$ has determinant one. Independence and Tonelli therefore give
\begin{equation}\label{eq:density-convolution-normalization}
 \int_{\R^d}f(t)\,dt
 =\prod_{j=1}^m\int_{\R^d}f_j(z_j)\,dz_j=1.
\end{equation}
For $m=1$, use $f=f_1$, with no fiber integral.

In~\eqref{eq:density-product-cosh}, put $z_m=t-\sum_{j<m}z_j$ and integrate over the first $m-1$ variables. For the positive translation, set $z'_j=z_j+\theta_jh$ for $j<m$. This translation has Jacobian one, and
\begin{equation}\label{eq:density-fiber-translation}
 z_m+\theta_mh
 =t-\sum_{j<m}z_j+\theta_mh
 =t+h-\sum_{j<m}z'_j,
\end{equation}
because $\sum_j\theta_j=1$. Thus that translated product integrates to $f(t+h)$. The negative translation integrates to $f(t-h)$. Consequently, for almost every $t,h$,
\begin{equation}\label{eq:density-convolution-midpoint}
 \frac{f(t+h)+f(t-h)}2\ge f(t)e^{-\Gamma\|h\|^2/2}.
\end{equation}
This argument justifies the fiber measure factor explicitly; it does not differentiate the convolution or assume log-concavity of its factors.

\noindent\textbf{Step 3: integrate the displacement variable.}
Since all $\gamma_j>0$ and at least one $\theta_j\ne0$, $\Gamma>0$. For almost every $t$, Tonelli and~\eqref{eq:density-convolution-normalization} give
\begin{equation}\label{eq:density-final-integration}
 \begin{aligned}
 1&=\frac12\int_{\R^d}f(t+h)\,dh
        +\frac12\int_{\R^d}f(t-h)\,dh\\
 &\ge f(t)\int_{\R^d}e^{-\Gamma\|h\|^2/2}\,dh
 =f(t)\left(\frac{2\pi}{\Gamma}\right)^{d/2}.
 \end{aligned}
\end{equation}
The finite left side also rules out an infinite fiber value for almost every $t$. Rearranging yields
\begin{equation*}
\|f\|_\infty
=\operatorname*{ess\,sup}_{t\in\mathbb R^d}f(t)
\le\left(\frac{\Gamma}{2\pi}\right)^{d/2},
\end{equation*}
as claimed.
\end{proof}

\subsection{Proof of Lemma~\ref{lem:tube}}

\begin{proof}
Define
\begin{equation*}
a=\varepsilon/r_0\in(0,1),\qquad
\mathcal T_\varepsilon(K)=\{z\in\mathbb R^d:\dist(z,\partial K)\le\varepsilon\}.
\end{equation*}
Here $\dist(z,\partial K)$ denotes the Euclidean distance from $z$ to $\partial K$ in $\mathbb R^d$. Since $B(0,r_0)\subseteq K$, scaling gives $B(0,\varepsilon)\subseteq aK$. Convexity also gives the Minkowski-sum identities
\begin{equation}\label{eq:tube-minkowski-dilations}
 K+aK=(1+a)K,
 \qquad (1-a)K+aK=K.
\end{equation}
For example, $x+ay=(1+a)[x/(1+a)+ay/(1+a)]\in(1+a)K$ for $x,y\in K$, and the reverse inclusion follows by setting $x=y$.

\noindent\textbf{Outer inclusion.}
Fix $z\in\mathcal T_\varepsilon(K)$. Compactness of $\partial K$ allows us to choose
\begin{equation*}
u_z\in\operatorname*{arg\,min}_{u\in\partial K}\|z-u\|,
\qquad \|z-u_z\|=\dist(z,\partial K)\le\varepsilon.
\end{equation*}
Thus $u_z\in K$ and $z-u_z\in B(0,\varepsilon)$, which imply
\begin{equation}\label{eq:tube-outer-inclusion}
 z\in K+B(0,\varepsilon)\subseteq K+aK=(1+a)K.
\end{equation}

\noindent\textbf{Exclusion of the interior of the inner dilation.}
Let $z\in\operatorname{int}((1-a)K)$. There exists $\delta>0$ such that $z+B(0,\delta)\subseteq(1-a)K$. Therefore
\begin{equation}\label{eq:tube-inner-buffer}
 \begin{aligned}
 z+B(0,\varepsilon+\delta)
 &=z+B(0,\delta)+B(0,\varepsilon)\\
 &\subseteq(1-a)K+aK=K.
 \end{aligned}
\end{equation}
For any $w$ with $\|w-z\|<\varepsilon+\delta$, set
\begin{equation*}
\rho_w=\frac{\varepsilon+\delta-\|w-z\|}{2}>0.
\end{equation*}
The triangle inequality and~\eqref{eq:tube-inner-buffer} give
\begin{equation*}
w+B(0,\rho_w)\subseteq z+B(0,\varepsilon+\delta)\subseteq K.
\end{equation*}
Hence $w\in\operatorname{int}K$, so $\dist(z,\partial K)\ge\varepsilon+\delta>\varepsilon$. Combining this exclusion with~\eqref{eq:tube-outer-inclusion} yields
\begin{equation*}
\mathcal T_\varepsilon(K)
\subseteq(1+a)K\setminus\operatorname{int}((1-a)K).
\end{equation*}

\noindent\textbf{Volume comparison.}
We first justify that $K$ and its interior have the same volume. For any $\tau\in(0,1)$,
\begin{equation}\label{eq:tube-boundary-null}
 (1-\tau)K+\tau B(0,r_0)\subseteq K,
 \qquad (1-\tau)K\subseteq\operatorname{int}K.
\end{equation}
The second inclusion follows because the first gives a positive-radius ball around each point of $(1-\tau)K$. Hence
\begin{equation}\label{eq:tube-interior-volume}
 (1-\tau)^d\vol(K)\le\vol(\operatorname{int}K)\le\vol(K),
\end{equation}
where $\operatorname{vol}(K)$ denotes the $d$-dimensional Lebesgue measure. Letting $\tau\downarrow0$ proves equality. The same fact holds for the positive dilation $(1-a)K$.

The inner dilation is contained in the outer dilation. By scaling of Lebesgue measure and the set inclusion already proved,
\begin{equation}\label{eq:tube-volume-detail}
 \begin{aligned}
 \vol\mathcal T_\varepsilon(K)
 &\le\bigl[(1+a)^d-(1-a)^d\bigr]\vol(K)\\
 &=d\left(\int_{1-a}^{1+a}s^{d-1}\,ds\right)\vol(K)\\
 &\le 2ad\,2^{d-1}\vol(B(0,R_0))\\
 &=d2^d\frac{\varepsilon}{r_0}\vol(B(0,R_0)).
 \end{aligned}
\end{equation}
For bounded $d,r_0^{-1},R_0$, the last coefficient is uniformly bounded, as required.
\end{proof}

\subsection{Proof of Lemma~\ref{lem:baddeleted}}

\begin{proof}
Let $v=\E\eta_1^2\in[v_-,v_+]$. We first control the empirical score energy, then the Gaussian design norm, and finally combine them without assuming independence of scores and design.

\noindent\textbf{Step 1: a centered-sum inequality for moments between one and two.}
Let $Z_1,\ldots,Z_{n-1}$ be independent real variables satisfying
\begin{equation*}
\E Z_j=0,\qquad \E|Z_j|^q<\infty,\qquad 1\le q\le2.
\end{equation*}
Let $(Z'_1,\ldots,Z'_{n-1})$ be an independent copy of $(Z_1,\ldots,Z_{n-1})$. Conditional Jensen gives
\begin{equation}\label{eq:bad-symmetrization}
 \E\left|\sum_{j=1}^{n-1}Z_j\right|^q
 =\E\left|\E\left[\sum_{j=1}^{n-1}(Z_j-Z'_j)\,\middle|\,Z_1,\ldots,Z_{n-1}\right]\right|^q
 \le\E\left|\sum_{j=1}^{n-1}(Z_j-Z'_j)\right|^q.
\end{equation}
The differences are independent and symmetric. Introduce independent random signs, independent of both copies, with
\begin{equation*}
\Prob(\epsilon_j=1)=\Prob(\epsilon_j=-1)=\tfrac12,\qquad j=1,\ldots,n-1.
\end{equation*}
Conditional on the differences, the concavity of $s\mapsto s^{q/2}$ and cancellation of cross terms yield
\begin{equation}\label{eq:bad-rademacher-moment}
 \E_\epsilon\left|\sum_{j=1}^{n-1}\epsilon_j(Z_j-Z'_j)\right|^q
 \le\left(\E_\epsilon\left|\sum_{j=1}^{n-1}\epsilon_j(Z_j-Z'_j)\right|^2\right)^{q/2}=\left(\sum_{j=1}^{n-1}|Z_j-Z'_j|^2\right)^{q/2}
 \le\sum_{j=1}^{n-1}|Z_j-Z'_j|^q.
\end{equation}
The last inequality uses $q/2\le1$. Symmetry means that multiplying by the independent signs leaves the joint law unchanged. Taking expectations and using $|a-b|^q\le2^{q-1}(|a|^q+|b|^q)$ proves
\begin{equation}\label{eq:bad-centered-sum-inequality}
 \E\left|\sum_{j=1}^{n-1}Z_j\right|^q\le2^q\sum_{j=1}^{n-1}\E|Z_j|^q.
\end{equation}

\noindent\textbf{Step 2: apply the inequality to squared scores.}
Take $q=3/2$ and $Z_j=\eta_j^2-v$ over $j\ne i$. Their $3/2$ moments are bounded because
\begin{equation}\label{eq:bad-centered-energy-moment}
 \E|\eta_j^2-v|^{3/2}
 \le2^{1/2}\{\E|\eta_j|^3+v^{3/2}\}\le C.
\end{equation}
After dividing~\eqref{eq:bad-centered-sum-inequality} by $(n-1)^{3/2}$, and using Jensen for the nonnegative average $V_i$, we obtain
\begin{equation}\label{eq:energymoment}
 \begin{aligned}
 \E|V_i-v|^{3/2}
 &\le (n-1)^{-3/2}C\sum_{j\ne i}\E|\eta_j^2-v|^{3/2}
 \le C(n-1)^{-1/2}\le Cn^{-1/2},\\
 \E V_i^{3/2}
 &=\E\left(\frac1{n-1}\sum_{j\ne i}\eta_j^2\right)^{3/2}
 \le\frac1{n-1}\sum_{j\ne i}\E|\eta_j|^3\le K_3.
 \end{aligned}
\end{equation}
Set $B_i^{\mathrm{eng}}=\{V_i\notin[v_-/2,2v_+]\}$. On the lower part $V_i<v_-/2$, we have $|V_i-v|\ge v_-/2$ and $1+V_i\le1+v_-/2$. Thus
\begin{equation}\label{eq:bad-lower-energy}
 (1+V_i)\ind_{\{V_i<v_-/2\}}\le C|V_i-v|^{3/2}.
\end{equation}
On the upper part $V_i>2v_+$, put $z=V_i-v$. Then $z\ge v_+$ and $V_i=v+z\le v_++z$. Since $z\ge v_+>0$, both $1$ and $z$ are bounded by fixed multiples of $z^{3/2}$. Therefore
\begin{equation}\label{eq:bad-upper-energy}
 (1+V_i)\ind_{\{V_i>2v_+\}}
 \le C|V_i-v|^{3/2}.
\end{equation}
Combining the two tails with~\eqref{eq:energymoment} gives the weighted estimate
\begin{equation}\label{eq:energybad}
 \E[(1+V_i)\ind_{B_i^{\mathrm{eng}}}]\le Cn^{-1/2}.
\end{equation}
The unweighted probability bound follows because $1+V_i\ge1$.

\noindent\textbf{Step 3: a Gaussian spectral tail and its restricted moments.}
Recall that $X_i$ is the design matrix with observation $i$ removed, and therefore it is an $(n-1)\times p$ matrix. Write $X_i=\mathsf{Z}\Sigma^{1/2}$, where $\mathsf{Z}$ is an $(n-1)\times p$ matrix of independent standard normal variables. Choose $1/4$-nets $\mathcal N_{n-1}$ and $\mathcal N_p$ of the two Euclidean unit spheres with cardinalities at most $9^{n-1}$ and $9^p$. To see the cardinality bound, take a maximal $1/4$-separated set on the unit sphere in $\R^d$. Maximality makes it a $1/4$-net. The balls of radius $1/8$ around its points have disjoint interiors and lie in $B(0,9/8)$, so volume comparison gives
\begin{equation}\label{eq:bad-net-packing}
 |\mathcal N_d|\,(1/8)^d\vol(B(0,1))
 \le(9/8)^d\vol(B(0,1)),\qquad |\mathcal N_d|\le9^d.
\end{equation}
Approximating both unit vectors in the variational formula for the operator norm gives
\begin{equation}\label{eq:bad-net-operator-bound}
 \|\mathsf{Z}\|_{\op}\le2\max_{u\in\mathcal N_{n-1},\,w\in\mathcal N_p}|u^\top \mathsf{Z}w|.
\end{equation}
Indeed, for unit vectors $u,w$, choose $u_0\in\mathcal N_{n-1}$ and $w_0\in\mathcal N_p$ with $\|u-u_0\|\le1/4$ and $\|w-w_0\|\le1/4$. Then
\begin{equation*}
|u^\top \mathsf{Z}w-u_0^\top \mathsf{Z}w_0|\le |(u-u_0)^\top \mathsf{Z}w|+|u_0^\top \mathsf{Z}(w-w_0)|\le\tfrac12\|\mathsf{Z}\|_{\op}.
\end{equation*}
Taking the supremum over $u,w$ and moving the error term to the left proves~\eqref{eq:bad-net-operator-bound}. Each $u_0^\top \mathsf{Z}w_0$ is $N(0,1)$, so a union bound implies
\begin{equation}\label{eq:bad-net-gaussian-tail}
 \Prob(\|\mathsf{Z}\|_{\op}>s)\le2\exp\{(n-1+p)\log9-s^2/8\}.
\end{equation}
Since $\|X_i\|_{\op}\le\sqrt{C_\Sigma/p}\|\mathsf{Z}\|_{\op}$,
\begin{equation}\label{eq:bad-design-spectral-tail}
 \Prob(\|X_i\|_{\op}>r)
 \le2\exp\left\{(n-1+p)\log9-\frac{pr^2}{8C_\Sigma}\right\}
 \le2e^{-cpr^2},\qquad r\ge K_X,
\end{equation}
after choosing $K_X$ large enough, using the bounded ratio $(n-1)/p$. For $T=\|X_i\|_{\op}$,
\begin{equation}\label{eq:bad-spectral-moment-integral}
 \E[T^6\ind_{\{T>K_X\}}]
 =K_X^6\Prob(T>K_X)+\int_{K_X}^\infty6r^5\Prob(T>r)\,dr
 \le Ce^{-cn}.
\end{equation}
This proves
\begin{equation}\label{eq:spectralbad}
 \Prob(\|X_i\|_{\op}>K_X)
 +\E[\|X_i\|_{\op}^6\ind_{\{\|X_i\|_{\op}>K_X\}}]
 \le Ce^{-cn}.
\end{equation}
The same tail integration also gives uniformly bounded spectral moments of every fixed order, a fact used in the light-tail refinement.

\noindent\textbf{Step 4: combine the energy and spectral failures.}
Let $B_i^{\mathrm{sp}}=\{\|X_i\|_{\op}>K_X\}$. Then $\mathcal A_i^c=B_i^{\mathrm{eng}}\cup B_i^{\mathrm{sp}}$. Equations~\eqref{eq:energybad} and~\eqref{eq:spectralbad} immediately imply~\eqref{eq:badprob}.

For the weighted claim, Lemma~\ref{lem:stability} and the elementary norm bound give
\begin{equation}\label{eq:bad-fit-energy-bound}
 \frac{\|b_i\|^2}{p}
 \le\frac{\|X_i\|_{\op}^2}{\mu^2p}\sum_{j\ne i}\eta_j^2
 =\frac{n-1}{\mu^2p}\|X_i\|_{\op}^2V_i
 \le C\|X_i\|_{\op}^2V_i.
\end{equation}
On the good spectral event, use~\eqref{eq:energybad}:
\begin{equation}\label{eq:bad-weighted-energy-part}
 \E\left[\left(1+\frac{\|b_i\|^2}{p}\right)
            \ind_{B_i^{\mathrm{eng}}\cap(B_i^{\mathrm{sp}})^c}\right]
 \le C\E[(1+V_i)\ind_{B_i^{\mathrm{eng}}}]
 \le Cn^{-1/2}.
\end{equation}
On the spectral failure, use H\"older with exponents $3/2$ and $3$:
\begin{equation}\label{eq:bad-holder-spectral-part}
 \E[\|X_i\|_{\op}^2V_i\ind_{B_i^{\mathrm{sp}}}]
 \le(\E V_i^{3/2})^{2/3}
       \left(\E[\|X_i\|_{\op}^6\ind_{B_i^{\mathrm{sp}}}]\right)^{1/3}\le Ce^{-cn}.
\end{equation}
Using the disjoint decomposition
\begin{equation*}
\mathcal A_i^c
=\bigl(B_i^{\mathrm{eng}}\cap(B_i^{\mathrm{sp}})^c\bigr)\mathbin{\dot\cup}B_i^{\mathrm{sp}},
\end{equation*}
we combine~\eqref{eq:bad-weighted-energy-part}, \eqref{eq:bad-holder-spectral-part}, and~\eqref{eq:spectralbad} to obtain
\begin{equation*}
\E\left[\left(1+\frac{\|b_i\|^2}{p}\right)\ind_{\mathcal A_i^c}\right]
\le Cn^{-1/2}+Ce^{-cn}\le Cn^{-1/2}.
\end{equation*}
This proves~\eqref{eq:weightedbad}. H\"older's inequality does not require independence between the score energy and the design norm.
\end{proof}
\section{Proofs for Section~\ref{sec:proof-of-thm2.5}}

\subsection{Proof of Lemma~\ref{lem:stability}}

\begin{proof}
\noindent\textbf{Step 1: coercivity and uniqueness on the stated domains.}
Fix deterministic sets $T\subseteq\{1,\ldots,n\}$ and $\mathcal I\subseteq\{1,\ldots,G\}$. The rows of $Z$ are ordered by increasing $j\in T$, and its column blocks by increasing $g\in\mathcal I$. 

For $u\in\mathbb R^{d_{\mathcal I}}$, define $\beta\in\mathbb R^p$ by
\begin{equation*}
\beta_g=
\begin{cases}
u_g,&g\in\mathcal I,\\
0,&g\notin\mathcal I.
\end{cases}
\end{equation*}
Since $r_g(0)=0$ for every $g$, we have
\begin{equation*}
x_j^\top\beta=z_j^\top u\quad(j\in T),
\qquad
r(\beta)=r_{\mathcal I}(u),
\qquad
\|\beta\|=\|u\|.
\end{equation*}
Write the objective in these free blocks as
\begin{equation*}
F_{T,\mathcal I}(u)
=
\sum_{j\in T}\ell(y_j;z_j^\top u)
+\lambda_1r_{\mathcal I}(u)
+\frac{\mu}{2}\|u\|^2,
\qquad
r_{\mathcal I}(u)=\sum_{g\in\mathcal I}r_g(u_g).
\end{equation*}
Convexity of $\ell(y_j;\cdot)$ gives, for every $j\in T$,
\begin{equation}\label{eq:stable-tangent-zero}
\ell(y_j;z_j^\top u)
\ge
\ell(y_j;0)+\eta_jz_j^\top u.
\end{equation}
Since $r_{\mathcal I}(u)\ge r_{\mathcal I}(0)=0$, it follows that
\begin{equation}\label{eq:stable-coercive}
\begin{aligned}
F_{T,\mathcal I}(u)
&\ge
F_{T,\mathcal I}(0)
+(Z^\top\eta)^\top u
+\frac{\mu}{2}\|u\|^2\\
&\ge
F_{T,\mathcal I}(0)
-\|Z\|_{\op}\|\eta\|\|u\|
+\frac{\mu}{2}\|u\|^2
\longrightarrow+\infty
\quad(\|u\|\longrightarrow\infty).
\end{aligned}
\end{equation}

For each fixed dataset, let $\mathcal D\subseteq\mathbb R^{d_{\mathcal I}}$ be nonempty, closed, and convex. The function $F_{T,\mathcal I}$ is finite and continuous, and~\eqref{eq:stable-coercive} implies that $\inf_{u\in\mathcal D}F_{T,\mathcal I}(u)>-\infty$. Choose a sequence $(u_m)_{m\ge1}\subseteq\mathcal D$ such that
\begin{equation*}
F_{T,\mathcal I}(u_m)
\longrightarrow
\inf_{u\in\mathcal D}F_{T,\mathcal I}(u).
\end{equation*}
By~\eqref{eq:stable-coercive}, $(u_m)_{m\ge1}$ is bounded. Hence some subsequence satisfies $u_{m_\nu}\to a_{\mathcal D}\in\mathcal D$, and continuity gives
\begin{equation*}
F_{T,\mathcal I}(a_{\mathcal D})
=
\inf_{u\in\mathcal D}F_{T,\mathcal I}(u).
\end{equation*}

The objective $F_{T,\mathcal I}$ is $\mu$-strongly convex. For $u,v\in\mathcal D$ and $0<t<1$,
\begin{equation}\label{eq:stable-strict-convexity}
F_{T,\mathcal I}((1-t)u+tv)
\le
(1-t)F_{T,\mathcal I}(u)
+tF_{T,\mathcal I}(v)
-\frac{\mu}{2}t(1-t)\|u-v\|^2.
\end{equation}
Since $(1-t)u+tv\in\mathcal D$, two distinct minimizers would contradict~\eqref{eq:stable-strict-convexity}. Thus $a_{\mathcal D}$ is unique. For every $u\in\mathcal D$ and $0<t<1$,
\begin{equation*}
F_{T,\mathcal I}(a_{\mathcal D})
\le
F_{T,\mathcal I}((1-t)a_{\mathcal D}+tu).
\end{equation*}
Combining this inequality with~\eqref{eq:stable-strict-convexity}, dividing by $t$, and letting $t\downarrow0$ gives
\begin{equation}\label{eq:stable-constrained-growth}
F_{T,\mathcal I}(u)-F_{T,\mathcal I}(a_{\mathcal D})
\ge
\frac{\mu}{2}\|u-a_{\mathcal D}\|^2.
\end{equation}

For $v\in\mathbb R^p$, the domains defined earlier have the descriptions
\begin{equation*}
\begin{aligned}
A(v)
&=
\{\beta\in\mathbb R^p:
\beta_g=0\text{ for every }g\notin S(v)\},\\
C(v)
&=
\{\beta\in A(v):
v_g\beta_g\ge0
\text{ for every }g\in S(v)\text{ with }d_g=1\}.
\end{aligned}
\end{equation*}
Both sets are nonempty, closed, and convex. Taking $\mathcal I=\{1,\ldots,G\}$, the preceding argument therefore applies to
\begin{equation*}
\mathcal D=\mathbb R^p,
\qquad
\mathcal D=\{\beta\in\mathbb R^p:\beta_g=0\},
\qquad
\mathcal D=A(b),
\qquad
\mathcal D=C(b).
\end{equation*}
Moreover,
\begin{equation*}
F_{\{1,\ldots,n\},\{1,\ldots,G\}}=F,
\qquad
F_{\{1,\ldots,n\}\setminus\{i\},\{1,\ldots,G\}}=F_i.
\end{equation*}
This proves existence and uniqueness for $b$, $b_i$, $m_i^A$, $m_i^C$, and the minimizers of $F$ and $F_i$ over $\{\beta\in\mathbb R^p:\beta_g=0\}$. Taking $\mathcal D=\mathbb R^{d_{\mathcal I}}$ also proves existence and uniqueness of the estimator $a$ in the statement.

\noindent\textbf{Step 2: measurability, including selected domains.}
Fix deterministic $T$ and $\mathcal I$ as above, and let $\mathcal D\subseteq\mathbb R^{d_{\mathcal I}}$ be a deterministic nonempty closed convex set. Define
\begin{equation*}
\mathcal F_{T,\mathcal I}
=
\sigma((y_j,x_{j,\mathcal I}):j\in T).
\end{equation*}
Choose a deterministic countable dense subset $\{u_r:r\ge1\}$ of $\mathcal D$. For every fixed $u$, the random variable $F_{T,\mathcal I}(u)$ is $\mathcal F_{T,\mathcal I}$-measurable. Continuity in $u$ therefore implies that
\begin{equation}\label{eq:stable-measurable-infimum}
I_{\mathcal D}
=
\inf_{r\ge1}F_{T,\mathcal I}(u_r)
=
\inf_{u\in\mathcal D}F_{T,\mathcal I}(u)
\end{equation}
is a finite $\mathcal F_{T,\mathcal I}$-measurable random variable. For each integer $m\ge1$, define
\begin{equation}\label{eq:stable-measurable-selection}
k_m
=
\min\{r\ge1:
F_{T,\mathcal I}(u_r)\le I_{\mathcal D}+m^{-1}\}.
\end{equation}
The set in~\eqref{eq:stable-measurable-selection} is nonempty by the definition of $I_{\mathcal D}$. Moreover,
\begin{equation}\label{eq:stable-selection-event}
\{k_m=r\}
=
\{F_{T,\mathcal I}(u_r)\le I_{\mathcal D}+m^{-1}\}
\cap
\bigcap_{1\le s<r}
\{F_{T,\mathcal I}(u_s)>I_{\mathcal D}+m^{-1}\}.
\end{equation}
Thus $u_{k_m}$ is $\mathcal F_{T,\mathcal I}$-measurable. For
\begin{equation*}
a_{\mathcal D}
=
\arg\min_{u\in\mathcal D}F_{T,\mathcal I}(u),
\end{equation*}
equation~\eqref{eq:stable-constrained-growth} gives
\begin{equation}\label{eq:stable-selection-convergence}
\frac{\mu}{2}\|u_{k_m}-a_{\mathcal D}\|^2
\le
F_{T,\mathcal I}(u_{k_m})-I_{\mathcal D}
\le
m^{-1}.
\end{equation}
Hence $u_{k_m}\to a_{\mathcal D}$ pointwise, so $a_{\mathcal D}$ is $\mathcal F_{T,\mathcal I}$-measurable. In particular, taking $\mathcal D=\mathbb R^{d_{\mathcal I}}$ proves the stated measurability of $a$.

Taking $T=\{1,\ldots,n\}$, $\mathcal I=\{1,\ldots,G\}$, and $\mathcal D=\mathbb R^p$ proves that $b$ is measurable with respect to $\sigma((y_j,x_j):1\le j\le n)$. Define the deterministic finite family
\begin{equation*}
\mathfrak D
=
\{A(v):v\in\mathbb R^p\}
\cup
\{C(v):v\in\mathbb R^p\},
\qquad
|\mathfrak D|\le4^G.
\end{equation*}
Indeed, each block factor is one of $\{0\}$ and $\mathbb R^{d_g}$, with the additional possibilities $[0,\infty)$ and $(-\infty,0]$ when $d_g=1$. For every $\mathcal D\in\mathfrak D$ and $D\in\{A,C\}$, the event $\{D(b)=\mathcal D\}$ is a finite combination of events of the form
\begin{equation*}
\{b_g=0\},
\qquad
\{b_g\ne0\},
\qquad
\{b_g>0\},
\qquad
\{b_g<0\},
\end{equation*}
where the last two forms are used only when $d_g=1$. These events are measurable. For each $i$ and $\mathcal D\in\mathfrak D$, let
\begin{equation*}
a_{i,\mathcal D}
=
\arg\min_{\beta\in\mathcal D}F_i(\beta).
\end{equation*}
The preceding construction shows that $a_{i,\mathcal D}$ is measurable with respect to $\sigma((y_j,x_j):j\ne i)$. Consequently, for $D\in\{A,C\}$,
\begin{equation*}
m_i^D
=
\sum_{\mathcal D\in\mathfrak D}
\mathbf 1_{\{D(b)=\mathcal D\}}a_{i,\mathcal D}
\end{equation*}
is measurable with respect to $\sigma((y_j,x_j):1\le j\le n)$.

For a deterministic block $g$, take $\mathcal I=\{1,\ldots,G\}\setminus\{g\}$. Then
\begin{equation*}
\mathcal F_{T,\mathcal I}
=
\sigma((y_j,x_{j,-g}):j\in T).
\end{equation*}
Extending the estimator $a\in\mathbb R^{d_{\mathcal I}}$ by setting its $g$th block to zero gives the unique minimizer
\begin{equation*}
\arg\min_{\beta\in\mathbb R^p:\,\beta_g=0}
F_{T,\{1,\ldots,G\}}(\beta).
\end{equation*}
This extension is deterministic and linear, so the minimizer is measurable with respect to $\sigma((y_j,x_{j,-g}):j\in T)$. In particular,
\begin{equation*}
\breve b_i^{(-g)}
=
\arg\min_{\beta\in\mathbb R^p:\,\beta_g=0}F_i(\beta)
\end{equation*}
is measurable with respect to
\begin{equation*}
\mathcal F_{i,-g}
=
\sigma((y_j,x_{j,-g}):j\ne i).
\end{equation*}
This conclusion uses the deterministic constraint $\beta_g=0$; the measurability established above for $m_i^A$ and $m_i^C$ is with respect to $\sigma((y_j,x_j):1\le j\le n)$.

\noindent\textbf{Step 3: the coefficient and score bounds.}
For deterministic $T$ and $\mathcal I$, let $a$ be the estimator in the statement, so that
\begin{equation*}
a
=
\arg\min_{u\in\mathbb R^{d_{\mathcal I}}}
F_{T,\mathcal I}(u).
\end{equation*}
Since $r_{\mathcal I}$ is minimized at zero,
\begin{equation*}
0\in\partial r_{\mathcal I}(0),
\qquad
Z^\top\eta\in\partial F_{T,\mathcal I}(0),
\qquad
0\in\partial F_{T,\mathcal I}(a).
\end{equation*}
Strong monotonicity of $\partial F_{T,\mathcal I}$ gives
\begin{equation}\label{eq:stable-origin-monotonicity}
\mu\|a\|^2
\le
\langle-Z^\top\eta,a\rangle
\le
\|Z\|_{\op}\|\eta\|\|a\|.
\end{equation}
If $a=0$, the norm bound is immediate. Otherwise, dividing by $\|a\|$ gives~\eqref{eq:normbound}.

Recall that
\begin{equation*}
\dot\ell(a)
=
(\ell'(y_j;z_j^\top a))_{j\in T}
\in\mathbb R^{|T|}.
\end{equation*}
The condition $0\in\partial F_{T,\mathcal I}(a)$ supplies $\zeta\in\partial r_{\mathcal I}(a)$ such that
\begin{equation}\label{eq:stable-fit-kkt}
Z^\top\dot\ell(a)+\lambda_1\zeta+\mu a=0,
\qquad
\langle\zeta,a\rangle
\ge
r_{\mathcal I}(a)-r_{\mathcal I}(0)
\ge0.
\end{equation}
Taking the inner product of~\eqref{eq:stable-fit-kkt}
with $a$ and using $\langle\zeta,a\rangle\ge0$ gives
\begin{equation}
\langle Z^\top\dot\ell(a),a\rangle+\mu\|a\|^2
\le
\langle Z^\top\dot\ell(a),a\rangle
+\lambda_1\langle\zeta,a\rangle+\mu\|a\|^2
=0.
\end{equation}

Moreover, the Lipschitz property of the score gives
\begin{equation*}
|\ell'(y_j;z_j^\top a)-\ell'(y_j;0)|
\le L_2|z_j^\top a|,
\qquad j\in T.
\end{equation*}

Combining these estimates yields
\begin{equation}\label{eq:stable-score-two-inequalities}
\mu\|a\|
\le
\|Z\|_{\op}\|\dot\ell(a)\|,
\qquad
\|\dot\ell(a)-\eta\|
\le
L_2\|Za\|
\le
L_2\|Z\|_{\op}\|a\|.
\end{equation}
The first inequality is immediate when $a=0$ and otherwise follows by division by $\|a\|$. Combining the two inequalities yields
\begin{equation}\label{eq:stable-score-lower}
\|\eta\|
\le
\|\dot\ell(a)\|+L_2\|Z\|_{\op}\|a\|
\le
\left(1+\frac{L_2\|Z\|_{\op}^2}{\mu}\right)
\|\dot\ell(a)\|.
\end{equation}
Using~\eqref{eq:normbound} instead gives
\begin{equation}\label{eq:stable-score-upper}
\|\dot\ell(a)\|
\le
\|\eta\|+L_2\|Z\|_{\op}\|a\|
\le
\left(1+\frac{L_2\|Z\|_{\op}^2}{\mu}\right)
\|\eta\|.
\end{equation}
Rearranging~\eqref{eq:stable-score-lower} and combining it with~\eqref{eq:stable-score-upper} proves~\eqref{eq:scoresandwich}.

\noindent\textbf{Step 4: deletion stability and score contraction.}
Fix $i\in\{1,\ldots,n\}$ and recall that $c_i=\ell'(y_i;x_i^\top b)$. Since
\begin{equation*}
0\in\partial F(b),
\qquad
\partial F_i(b)=\partial F(b)-c_ix_i,
\qquad
0\in\partial F_i(b_i),
\end{equation*}
we have $-c_ix_i\in\partial F_i(b)$. Strong monotonicity of $\partial F_i$ therefore gives
\begin{equation}\label{eq:stable-deletion-monotonicity}
-c_ix_i^\top(b-b_i)
\ge
\mu\|b-b_i\|^2.
\end{equation}
If $b=b_i$, the displacement bound is immediate. Otherwise, Cauchy--Schwarz and division by $\|b-b_i\|$ give
\begin{equation*}
\|b-b_i\|
\le
\mu^{-1}|c_i|\|x_i\|,
\end{equation*}
which is the first assertion in~\eqref{eq:scorecontraction}.

Define
\begin{equation*}
c_i^-=\ell'(y_i;x_i^\top b_i).
\end{equation*}
Equation~\eqref{eq:stable-deletion-monotonicity} implies $-c_ix_i^\top(b-b_i)\ge0$. Together with monotonicity of $\ell'(y_i;\cdot)$, this yields
\begin{equation*}
\begin{aligned}
x_i^\top b>x_i^\top b_i
&\quad\Longrightarrow\quad
c_i^-\le c_i\le0,\\
x_i^\top b<x_i^\top b_i
&\quad\Longrightarrow\quad
0\le c_i\le c_i^-,\\
x_i^\top b=x_i^\top b_i
&\quad\Longrightarrow\quad
c_i=c_i^-.
\end{aligned}
\end{equation*}
Thus in all cases
\begin{equation}\label{eq:stable-score-order}
|c_i|\le|c_i^-|,
\end{equation}
which completes~\eqref{eq:scorecontraction}.
\end{proof}
\subsection{Proof of Lemma \ref{lem:strong_convexity}}\label{proof:lemstrong:sconvex}
\begin{proof}
Since $F_i$ is $\mu$-strongly convex and $b_i$ is its unrestricted minimizer,
\[
F_i(\beta)-F_i(b_i)
\ge
\frac{\mu}{2}\|\beta-b_i\|^2
\]
for every $\beta\in\mathbb R^p$. Since $a_i\in D(b)$ and $m_i^D$ minimizes $F_i$ over $D(b)$,
\[
F_i(m_i^D)\le F_i(a_i).
\]
Therefore,
\[
\frac{\mu}{2}\|m_i^D-b_i\|^2
\le
F_i(m_i^D)-F_i(b_i)
\le
F_i(a_i)-F_i(b_i).
\]

It remains to prove the upper bound. By the optimality conditions for $b_i$ and $b$, there exist
\[
z_i\in\partial r(b_i),
\qquad
z\in\partial r(b),
\]
such that
\[
v_i^0=\mu b_i+\lambda_1 z_i,
\qquad
v(b)=\mu b+\lambda_1 z.
\]
Hence, for every block $g$,
\[
v_g(b)-v_{i,g}^0
=
\mu(b_g-b_{i,g})
+
\lambda_1(z_g-z_{i,g}).
\]
Since $r_g$ is convex, the monotonicity of $\partial r_g$ gives
\[
\langle z_g-z_{i,g},b_g-b_{i,g}\rangle\ge 0.
\]
Therefore,
\[
\mu\|b_g-b_{i,g}\|^2
\le
\langle v_g(b)-v_{i,g}^0,b_g-b_{i,g}\rangle
\le
\|v_g(b)-v_{i,g}^0\|\,\|b_g-b_{i,g}\|,
\]
and thus
\[
\|b_g-b_{i,g}\|
\le
\frac{1}{\mu}\|v_g(b)-v_{i,g}^0\|.
\]
Since $a_i-b_i$ is supported on $\mathcal S_i$,
\[
\|a_i-b_i\|^2
=
\sum_{g\in\mathcal S_i}\|b_g-b_{i,g}\|^2
\le
\frac{1}{\mu^2}
\sum_{g\in\mathcal S_i}
\|v_g(b)-v_{i,g}^0\|^2.
\]

Let $\delta_i=a_i-b_i$. Using
\[
\nabla L_i(b_i)+\mu b_i+\lambda_1 z_i=0,
\]
we can write
\[
F_i(a_i)-F_i(b_i)
=
D_{L_i}(a_i,b_i)
+
\lambda_1 D_r^{z_i}(a_i,b_i)
+
\frac{\mu}{2}\|\delta_i\|^2,
\]
where
\[
D_{L_i}(a_i,b_i)
=
L_i(a_i)-L_i(b_i)
-\langle \nabla L_i(b_i),\delta_i\rangle,
\]
and
\[
D_r^{z_i}(a_i,b_i)
=
r(a_i)-r(b_i)-\langle z_i,\delta_i\rangle.
\]

By Assumption~2.4,
\[
\nabla^2 L_i(\beta)
\preceq
L_2 X_i^\top X_i.
\]
Hence,
\[
D_{L_i}(a_i,b_i)
\le
\frac{L_2}{2}\|X_i\delta_i\|^2
\le
\frac{L_2K_X^2}{2}\|\delta_i\|^2,
\]
and therefore
\[
D_{L_i}(a_i,b_i)
\le
\frac{L_2K_X^2}{2\mu^2}
\sum_{g\in\mathcal S_i}
\|v_g(b)-v_{i,g}^0\|^2.
\]

For the penalty term, since $\delta_{i,g}=0$ for $g\notin\mathcal S_i$,
\[
D_r^{z_i}(a_i,b_i)
=
\sum_{g\in\mathcal S_i}
\Bigl[
r_g(b_g)-r_g(b_{i,g})
-
\langle z_{i,g},b_g-b_{i,g}\rangle
\Bigr].
\]
By convexity of $r_g$,
\[
r_g(b_{i,g})
\ge
r_g(b_g)+
\langle z_g,b_{i,g}-b_g\rangle,
\]
and hence
\[
r_g(b_g)-r_g(b_{i,g})
-
\langle z_{i,g},b_g-b_{i,g}\rangle
\le
\langle z_g-z_{i,g},b_g-b_{i,g}\rangle.
\]
Using the relation between the primal and dual differences,
\[
\lambda_1
\langle z_g-z_{i,g},b_g-b_{i,g}\rangle
=
\langle v_g(b)-v_{i,g}^0,b_g-b_{i,g}\rangle
-
\mu\|b_g-b_{i,g}\|^2.
\]
Therefore,
\[
\lambda_1 D_r^{z_i}(a_i,b_i)
+
\frac{\mu}{2}\|\delta_i\|^2
\le
\sum_{g\in\mathcal S_i}
\langle v_g(b)-v_{i,g}^0,b_g-b_{i,g}\rangle.
\]
Applying Cauchy--Schwarz and the preceding blockwise bound yields
\[
\lambda_1 D_r^{z_i}(a_i,b_i)
+
\frac{\mu}{2}\|\delta_i\|^2
\le
\frac{1}{\mu}
\sum_{g\in\mathcal S_i}
\|v_g(b)-v_{i,g}^0\|^2.
\]

Combining the last two estimates, we obtain
\[
F_i(a_i)-F_i(b_i)
\le
\left(
\frac{1}{\mu}
+
\frac{L_2K_X^2}{2\mu^2}
\right)
\sum_{g\in\mathcal S_i}
\|v_g(b)-v_{i,g}^0\|^2.
\]
Thus the result holds with
\[
C=
\frac{1}{\mu}
+
\frac{L_2K_X^2}{2\mu^2}.
\]
\end{proof}

\subsection{Proof of Lemma~\ref{lem:tilt}}

\begin{proof}
Fix $i\in\{1,\ldots,n\}$ and a realization in $\mathcal A_i\cap\mathcal G_i$ with $0<t_i\le1$. Recall that $t_i=D_0S_ih_n$ and $k_i=N_i(2t_i)+S_i^2$. All quantities in the proof are evaluated at this realization.

\noindent\textbf{Step 1: the tilted objective and its linear response.}
The full-data first-order condition supplies a vector $z\in\partial r(b)$ such that
\begin{equation}\label{eq:tilt-full-kkt}
 \nabla L_i(b)+c_ix_i+\lambda_1z+\mu b=0.
\end{equation}
Define
\begin{equation*}
\widetilde F_i(\beta)=F_i(\beta)+c_ix_i^\top\beta.
\end{equation*}
Here $c_i=\ell'(y_i;x_i^\top b)$ is fixed as $\beta$ varies. Equation~\eqref{eq:tilt-full-kkt} gives $0\in\partial\widetilde F_i(b)$, and strong convexity gives
\begin{equation*}
\widetilde F_i(u)-\widetilde F_i(b)\ge\frac{\mu}{2}\|u-b\|^2,
\qquad u\in\mathbb R^p.
\end{equation*}
Thus $b$ is the unique minimizer of $\widetilde F_i$.

Define
\begin{equation*}
u_i=-c_iQ_ix_i,\qquad d_i=J_iu_i.
\end{equation*}
The identities in Lemma~\ref{lem:resolvent} imply
\begin{equation}\label{eq:tilt-linear-identities}
 \begin{aligned}
 d_i&=-c_iJ_i(I+H_iJ_i)^{-1}x_i=-c_iB_ix_i,\\
 (I+H_iJ_i)u_i&=-c_ix_i,\\
 u_i+c_ix_i&=-H_iJ_iu_i=-H_id_i.
 \end{aligned}
\end{equation}
Since $|c_i|\le S_i$, the event $\mathcal G_i$, the operator bounds in~\eqref{eq:gaussian-operator-bounds}, and $D_0\ge4C_G$ give
\begin{equation}\label{eq:tilt-linear-bounds-detail}
 \begin{aligned}
 \|u_i\|_{2,\infty}&\le C_GS_ih_n\le t_i/4,\\
 \|u_i\|+\|d_i\|+\|X_id_i\|&\le CS_i\|x_i\|\le CS_i,\\
 \|X_id_i\|_\infty&\le C_GS_ih_n.
 \end{aligned}
\end{equation}
These are the bounds in~\eqref{eq:linearcontrol}, derived before any Taylor expansion.

\noindent\textbf{Step 2: a blockwise proximal expansion of the trial point.}
Set $v_i^T=v_i^0+u_i$ and $b_i^T=\Phi(v_i^T)$. Deleted optimality says $\Phi(v_i^0)=b_i$. Hence
\begin{equation}\label{eq:tilt-trial-decomposition}
 b_i^T-b_i=d_i+r_i,\qquad
 r_{i,g}=\Phi_g(v_{i,g}^0+u_{i,g})-\Phi_g(v_{i,g}^0)
                       -J_g(v_{i,g}^0)u_{i,g}.
\end{equation}
For each block, let
\begin{equation*}
\mathcal L_{ig}^{T}=\{v_{i,g}^0+su_{i,g}:0\le s\le1\}.
\end{equation*}
If $g\notin\mathcal B_i$, then
\begin{equation*}
\begin{aligned}
\dist(v_{i,g}^0+su_{i,g},\partial K_g)
&\ge\dist(v_{i,g}^0,\partial K_g)-s\|u_{i,g}\|\\
&>2t_i-t_i/4>0,\qquad 0\le s\le1.
\end{aligned}
\end{equation*}
Thus $\mathcal L_{ig}^{T}\cap\partial K_g=\varnothing$. Apply the Taylor bound in Lemma~\ref{lem:prox} on these blocks and its global Lipschitz bound on $g\in\mathcal B_i$ to obtain
\begin{equation}\label{eq:tilt-block-remainders}
 \|r_{i,g}\|\le
 \begin{cases}
 (C_\Phi/2)\|u_{i,g}\|^2,&g\notin\mathcal B_i,\\
 2\mu^{-1}\|u_{i,g}\|,&g\in\mathcal B_i.
 \end{cases}
\end{equation}
For the smooth blocks,
\begin{equation}\label{eq:tilt-smooth-square-sum}
 \left(\sum_{g\notin\mathcal B_i}\|r_{i,g}\|^2\right)^{1/2}
 \le C\left(\sum_g\|u_{i,g}\|^4\right)^{1/2}
 \le C\|u_i\|_{2,\infty}\|u_i\|.
\end{equation}
For the boundary blocks, $|\mathcal B_i|=N_i(2t_i)$ and
\begin{equation}\label{eq:tilt-boundary-square-sum}
 \left(\sum_{g\in\mathcal B_i}\|r_{i,g}\|^2\right)^{1/2}
 \le C\|u_i\|_{2,\infty}\sqrt{N_i(2t_i)}
 \le Ct_i\sqrt{N_i(2t_i)}.
\end{equation}
Combining these estimates, and using $t_i=D_0S_ih_n$ and $k_i=N_i(2t_i)+S_i^2$, gives
\begin{equation}\label{eq:proxremainder}
 \begin{aligned}
 \|r_i\|
 &\le C\|u_i\|_{2,\infty}\|u_i\|
             +Ct_i\sqrt{N_i(2t_i)}\\
 &\le C\{S_i^2h_n+t_i\sqrt{N_i(2t_i)}\}\\
 &\le Ct_i\{S_i+\sqrt{N_i(2t_i)}\}
 \le Ct_i\sqrt{k_i}.
 \end{aligned}
\end{equation}
The last step uses $(a+b)^2\le2(a^2+b^2)$. In particular, the smooth remainder is accounted for by $S_i^2$ inside $k_i$ even when the boundary count vanishes.

\noindent\textbf{Step 3: the residual of the tilted objective.}
Proximal optimality at $b_i^T$ gives $v_i^T\in\mu b_i^T+\lambda_1\partial r(b_i^T)$. It follows that
\begin{equation}\label{eq:tilt-trial-subgradient}
 g_i^T=\nabla L_i(b_i^T)+c_ix_i+v_i^T
       \in\partial\widetilde F_i(b_i^T).
\end{equation}
Use $v_i^0=-\nabla L_i(b_i)$, $b_i^T=b_i+d_i+r_i$, and~\eqref{eq:tilt-linear-identities} to write the residual as
\begin{equation}\label{eq:tilt-residual-decomposition}
 \begin{aligned}
 g_i^T
 &=\nabla L_i(b_i^T)-\nabla L_i(b_i)+c_ix_i+u_i\\
 &=\nabla L_i(b_i+d_i)-\nabla L_i(b_i)-H_id_i
 +\nabla L_i(b_i+d_i+r_i)-\nabla L_i(b_i+d_i)\\
 &=R_i^{\text{loss}}+R_i^{\text{prox}},
 \end{aligned}
\end{equation}
where $R_i^{\text{loss}}=\nabla L_i(b_i+d_i)-\nabla L_i(b_i)-H_id_i$, and $R_i^{\text{prox}}=\nabla L_i(b_i+d_i+r_i)-\nabla L_i(b_i+d_i)$.
To bound $R_i^{\text{loss}}$, let $\xi_{ij}=x_j^\top b_i$ and $\delta_{ij}^{\mathrm{lin}}=x_j^\top d_i$ for $j\ne i$. The fundamental theorem of calculus gives
\begin{equation}\label{eq:tilt-loss-score-remainder}
 \begin{aligned}
 \rho_j
 &=\ell'(y_j;\xi_{ij}+\delta_{ij}^{\mathrm{lin}})-\ell'(y_j;\xi_{ij})-\ell''(y_j;\xi_{ij})\delta_{ij}^{\mathrm{lin}}\\
 &=\delta_{ij}^{\mathrm{lin}}\int_0^1[\ell''(y_j;\xi_{ij}+s\delta_{ij}^{\mathrm{lin}})-\ell''(y_j;\xi_{ij})]\,ds,\\
 |\rho_j|&\le\frac{L_3}{2}|\delta_{ij}^{\mathrm{lin}}|^2.
 \end{aligned}
\end{equation}
Set $\rho=(\rho_j)_{j\ne i}$. The componentwise identities give $R_i^{\text{loss}}=X_i^\top\rho$. The estimate uses only the Lipschitz continuity of $\ell''$. By~\eqref{eq:tilt-linear-bounds-detail},
\begin{equation}\label{eq:tilt-loss-vector-remainder}
 \begin{aligned}
 R_i^{\text{loss}}=\|X_i^\top\rho\|
 &\le\frac{L_3}{2}\|X_i\|_{\op}
       \left(\sum_{j\ne i}|\delta_{ij}^{\mathrm{lin}}|^4\right)^{1/2}\\
 &\le\frac{L_3}{2}\|X_i\|_{\op}
                  \|X_id_i\|_\infty\|X_id_i\|\\
 &\le CS_i^2h_n\le Ct_i\sqrt{k_i}.
 \end{aligned}
\end{equation}
The remaining term $R_i^{\text{prox}}$ satisfies
\begin{equation}\label{eq:tilt-loss-lipschitz-remainder}
 \|R_i^{\text{prox}}\|
 \le L_2\|X_i\|_{\op}^2\|r_i\|
 \le Ct_i\sqrt{k_i}.
\end{equation}
Hence $\|g_i^T\|\le Ct_i\sqrt{k_i}$. Strong monotonicity between $g_i^T\in\partial\widetilde F_i(b_i^T)$ and $0\in\partial\widetilde F_i(b)$ gives
\begin{equation}\label{eq:tilt-primal-localization}
 \mu\|b_i^T-b\|^2
 \le\langle g_i^T,b_i^T-b\rangle
 \le\|g_i^T\|\|b_i^T-b\|,
 \qquad \|b_i^T-b\|\le Ct_i\sqrt{k_i}.
\end{equation}
For the dual difference, use $v(b)=-\nabla L_i(b)-c_ix_i$ and~\eqref{eq:tilt-trial-subgradient}:
\begin{equation}\label{eq:tilt-dual-localization}
 \begin{aligned}
 v(b)-v_i^T
 &=-[\nabla L_i(b)-\nabla L_i(b_i^T)]-g_i^T,\\
 \|v(b)-v_i^T\|
 &\le L_2K_X^2\|b-b_i^T\|+\|g_i^T\|
 \le Ct_i\sqrt{k_i}.
 \end{aligned}
\end{equation}
Together~\eqref{eq:tilt-primal-localization} and~\eqref{eq:tilt-dual-localization} prove~\eqref{eq:tilterror}.

\noindent\textbf{Step 4: localization and counting of actual stratum changes.}
Define
\begin{equation}\label{eq:tilt-exceptional-set}
 E_i=\{g:\|v_g(b)-v_{i,g}^T\|>t_i\}.
\end{equation}
Since $\|v_g(b)-v_{i,g}^T\|>t_i$ for $g\in E_i$,
\begin{equation*}
|E_i|t_i^2
\le\sum_{g\in E_i}\|v_g(b)-v_{i,g}^T\|^2
\le\|v(b)-v_i^T\|^2.
\end{equation*}
Using~\eqref{eq:tilt-dual-localization}, we obtain
\begin{equation}\label{eq:tilt-exceptional-count}
 |E_i|t_i^2\le\|v(b)-v_i^T\|^2\le Ct_i^2k_i,
 \qquad |E_i|\le Ck_i.
\end{equation}
For $g\notin\mathcal B_i\cup E_i$,
\begin{equation}\label{eq:tilt-no-crossing-distance}
 \|v_g(b)-v_{i,g}^0\|
 \le\|v_g(b)-v_{i,g}^T\|+\|u_{i,g}\|
 \le\frac54t_i<2t_i
 <\dist(v_{i,g}^0,\partial K_g).
\end{equation}
For $g\notin\mathcal B_i\cup E_i$, define
\begin{equation*}
\mathcal L_{ig}=\{(1-s)v_{i,g}^0+sv_g(b):0\le s\le1\}.
\end{equation*}
Equation~\eqref{eq:tilt-no-crossing-distance} implies $\mathcal L_{ig}\cap\partial K_g=\varnothing$. If $\mathcal L_{ig}\subseteq\operatorname{int}K_g$, then $b_g=b_{i,g}=0$. For $d_g=1$, write $K_g=[\kappa_g^-,\kappa_g^+]$. The scalar sign characterization in Lemma~\ref{lem:prox} gives
\begin{equation*}
\begin{aligned}
\mathcal L_{ig}\subseteq(\kappa_g^+,\infty)
&\ \Longrightarrow\ b_g>0,\ b_{i,g}>0,\\
\mathcal L_{ig}\subseteq(-\infty,\kappa_g^-)
&\ \Longrightarrow\ b_g<0,\ b_{i,g}<0.
\end{aligned}
\end{equation*}
For $d_g\ge2$, $\mathcal L_{ig}\subseteq\mathbb R^{d_g}\setminus K_g$ implies $b_g\ne0$ and $b_{i,g}\ne0$. In each case $M_g(b_g)=M_g(b_{i,g})$. Therefore,
\begin{equation}\label{eq:tilt-change-count-detail}
 \mathcal S_i\subseteq\mathcal B_i\cup E_i,
 \qquad |\mathcal S_i|\le N_i(2t_i)+Ck_i\le Ck_i.
\end{equation}
Finally, $v_g(b)-v_{i,g}^0=(v_g(b)-v_{i,g}^T)+u_{i,g}$, so
\begin{equation}\label{eq:tilt-change-energy-detail}
 \begin{aligned}
 \sum_{g\in\mathcal S_i}\|v_g(b)-v_{i,g}^0\|^2
 &\le2\sum_{g\in\mathcal S_i}\|v_g(b)-v_{i,g}^T\|^2
       +2\sum_{g\in\mathcal S_i}\|u_{i,g}\|^2\\
 &\le2\|v(b)-v_i^T\|^2
       +2|\mathcal S_i|\|u_i\|_{2,\infty}^2
 \le Ct_i^2k_i.
 \end{aligned}
\end{equation}
Equations~\eqref{eq:tilt-change-count-detail}--\eqref{eq:tilt-change-energy-detail} prove~\eqref{eq:swap}.
\end{proof}

\subsection{Proof of Lemma~\ref{lem:bridge}}

\begin{proof}
Fix $D\in\{A,C\}$ and write $m_i=m_i^D$. Throughout the proof, work on $\mathcal A_i\cap\mathcal G_i$ with $0<t_i\le1$.

\noindent\textbf{Step 1: construction and feasibility of the hybrid.}
Define the block vector $a_i$ by
\begin{equation}\label{eq:bridge-hybrid-definition}
 a_{i,g}=\begin{cases}
 b_{i,g},&g\notin\mathcal S_i,\\
 b_g,&g\in\mathcal S_i.
 \end{cases}
\end{equation}
By the definition of $\mathcal S_i$,
\begin{equation*}
\begin{aligned}
g\notin\mathcal S_i
&\ \Longrightarrow\ a_{i,g}=b_{i,g}\in M_g(b_{i,g})=M_g(b_g),\\
g\in\mathcal S_i
&\ \Longrightarrow\ a_{i,g}=b_g\in M_g(b_g).
\end{aligned}
\end{equation*}
Taking the product over $g$ gives
\begin{equation}\label{eq:bridge-hybrid-feasible}
 a_i\in\mathcal M^\circ(b)\subseteq C(b)\subseteq A(b).
\end{equation}
Since $m_i=\arg\min_{\beta\in D(b)}F_i(\beta)$ and $a_i\in D(b)$,
\begin{equation}\label{eq:bridge-objective-order}
 F_i(m_i)\le F_i(a_i).
\end{equation}
The hybrid is used only as a feasible competitor; it is not asserted to be independent of observation $i$.

\noindent\textbf{Step 2: a Bregman comparison centered at the deleted fit.}
Put $f_i(\beta)=L_i(\beta)+\mu\|\beta\|^2/2$ and
\begin{equation}\label{eq:bridge-deleted-kkt}
 z_i=\frac{v_i^0-\mu b_i}{\lambda_1}\in\partial r(b_i),
 \qquad \nabla f_i(b_i)+\lambda_1z_i=0.
\end{equation}
The Hessian satisfies $\mu I\preceq\nabla^2f_i\preceq L_GI$, where $L_G=\mu+L_2K_X^2$. For any vector $u$, the integral Taylor formula gives
\begin{equation}\label{eq:bridge-smooth-bregman}
 \begin{aligned}
 D_{f_i}(u,b_i)
 &:=f_i(u)-f_i(b_i)-\langle\nabla f_i(b_i),u-b_i\rangle\\
 &=\int_0^1(1-s)(u-b_i)^\top
      \nabla^2f_i(b_i+s(u-b_i))(u-b_i)\,ds,\\
 \frac\mu2\|u-b_i\|^2
 &\le D_{f_i}(u,b_i)\le\frac{L_G}{2}\|u-b_i\|^2.
 \end{aligned}
\end{equation}
The identity $\nabla f_i(b_i)+\lambda_1z_i=0$ cancels the linear terms:
\begin{equation*}
\begin{aligned}
F_i(u)-F_i(b_i)
&=D_{f_i}(u,b_i)+\lambda_1D_r^{z_i}(u,b_i)\\
&\quad+\langle\nabla f_i(b_i)+\lambda_1z_i,u-b_i\rangle.
\end{aligned}
\end{equation*}
Hence
\begin{equation}\label{eq:bridge-objective-bregman}
 F_i(u)-F_i(b_i)=D_{f_i}(u,b_i)+\lambda_1D_r^{z_i}(u,b_i).
\end{equation}
Apply the lower bound from~\eqref{eq:bridge-smooth-bregman} at $u=m_i$, the upper bound at $u=a_i$, and~\eqref{eq:bridge-objective-order}. This yields
\begin{equation}\label{eq:hybridbregman}
 \frac\mu2\|m_i-b_i\|^2+\lambda_1D_r^{z_i}(m_i,b_i)
 \le\lambda_1D_r^{z_i}(a_i,b_i)+\frac{L_G}{2}\|a_i-b_i\|^2.
\end{equation}
Both terms on the left are nonnegative.

\noindent\textbf{Step 3: evaluate every nonzero contribution on the right.}
Separability and~\eqref{eq:bridge-hybrid-definition} imply
\begin{equation}\label{eq:bridge-supported-contributions}
 \begin{aligned}
 D_r^{z_i}(a_i,b_i)
 &=\sum_{g\in\mathcal S_i}D_{r_g}^{z_{i,g}}(b_g,b_{i,g}),\\
 \|a_i-b_i\|^2&=\sum_{g\in\mathcal S_i}\|b_g-b_{i,g}\|^2.
 \end{aligned}
\end{equation}
For $g\in\mathcal S_i$, define
\begin{equation*}
q_g=v_g(b)-v_{i,g}^0,\qquad e_g=b_g-b_{i,g},\qquad
z_g=\frac{v_g(b)-\mu b_g}{\lambda_1}\in\partial r_g(b_g).
\end{equation*}
Since $b_g=\Phi_g(v_g(b))$ and $b_{i,g}=\Phi_g(v_{i,g}^0)$, the energy identity in Lemma~\ref{lem:prox} gives
\begin{equation}\label{eq:bridge-block-energy}
 \mu\|e_g\|^2
 +\lambda_1D_{r_g}^{z_{i,g}}(b_g,b_{i,g})
 +\lambda_1D_{r_g}^{z_g}(b_{i,g},b_g)
 =\langle q_g,e_g\rangle.
\end{equation}
The last Bregman term is nonnegative, and $\|e_g\|\le\mu^{-1}\|q_g\|$ by proximal Lipschitz continuity. Therefore
\begin{equation}\label{eq:bridge-block-upper}
 \begin{aligned}
 &\lambda_1D_{r_g}^{z_{i,g}}(b_g,b_{i,g})
                      +\frac{L_G}{2}\|e_g\|^2\\
 &\quad\le\langle q_g,e_g\rangle+\frac{L_G}{2}\|e_g\|^2
 \le\left(\frac1\mu+\frac{L_G}{2\mu^2}\right)\|q_g\|^2.
 \end{aligned}
\end{equation}
Sum~\eqref{eq:bridge-block-upper}, apply~\eqref{eq:swap}, and then drop the nonnegative penalty term on the left of~\eqref{eq:hybridbregman}:
\begin{equation}\label{eq:bridge-primal-bound-detail}
 \frac\mu2\|m_i-b_i\|^2
 \le C\sum_{g\in\mathcal S_i}\|q_g\|^2
 \le Ct_i^2k_i,
 \qquad \|m_i-b_i\|\le Ct_i\sqrt{k_i}.
\end{equation}
For $\beta\in\mathbb R^p$, define the deleted score vector
\begin{equation*}
\mathbf s_i(\beta)=(\ell'(y_j;x_j^\top\beta))_{j\ne i}.
\end{equation*}
Since $v_i(\beta)=-X_i^\top\mathbf s_i(\beta)$,
\begin{equation}\label{eq:bridge-dual-bound-detail}
 \begin{aligned}
 \|v_i(b_i)-v_i(m_i)\|
 &\le\|X_i\|_{\op}
       \|\mathbf s_i(b_i)-\mathbf s_i(m_i)\|\\
 &\le L_2\|X_i\|_{\op}\|X_i(b_i-m_i)\|
 \le L_2K_X^2\|b_i-m_i\|
 \le Ct_i\sqrt{k_i}.
 \end{aligned}
\end{equation}
This proves~\eqref{eq:vectorlocal} without bounding the primal penalty Hessian near zero.
\end{proof}

\subsection{Proof of Lemma~\ref{lem:prediction}}

\begin{proof}
Fix $D\in\{A,C\}$ and set $m_i=m_i^D$ and $\Delta_i=b_i-m_i$. Work on $\mathcal A_i\cap\mathcal G_i$ with $0<t_i\le1$.

\noindent\textbf{Step 1: isolate all exceptional blocks.}
Write $\delta v_i=v_i(b_i)-v_i(m_i)$ and retain the set $E_i$ in~\eqref{eq:tilt-exceptional-set}. Define
\begin{equation}\label{eq:prediction-exceptional-sets}
 E_i'=\{g:\|\delta v_{i,g}\|>t_i\},\qquad
 D_i=\mathcal B_i\cup E_i\cup E_i'.
\end{equation}
The squared dual-displacement bound in Lemma~\ref{lem:bridge} gives
\begin{equation}\label{eq:prediction-exceptional-cardinality}
 |E_i'|\le t_i^{-2}\|\delta v_i\|^2\le Ck_i,
 \qquad |D_i|\le N_i(2t_i)+|E_i|+|E_i'|\le Ck_i.
\end{equation}
If $g\notin D_i$, it is outside $\mathcal B_i\cup E_i$, so its full and deleted strata agree by Lemma~\ref{lem:tilt}. It also satisfies
\begin{equation}\label{eq:prediction-safe-segment}
 \dist(v_{i,g}^0,\partial K_g)>2t_i,\qquad
 \|\delta v_{i,g}\|\le t_i,
 \qquad
 \dist(v_{i,g}^0-s\delta v_{i,g},\partial K_g)>t_i\quad(0\le s\le1).
\end{equation}
Indeed, for $0\le s\le1$,
\begin{equation*}
\dist(v_{i,g}^0-s\delta v_{i,g},\partial K_g)
\ge\dist(v_{i,g}^0,\partial K_g)-s\|\delta v_{i,g}\|
>2t_i-t_i=t_i.
\end{equation*}
Thus the segment from $v_{i,g}^0$ to $v_{i,g}(m_i)=v_{i,g}^0-\delta v_{i,g}$ does not meet $\partial K_g$.

\noindent\textbf{Step 2: justify the proximal equation on each nonexceptional block.}
We show
\begin{equation}\label{eq:prediction-safe-proximal-equation}
 m_{i,g}=\Phi_g(v_{i,g}(m_i)),\qquad g\notin D_i,
\end{equation}
by treating the constraints explicitly.

If $g\notin S(b)$, stratum agreement and feasibility give $b_g=b_{i,g}=m_{i,g}=0$. Since $\Phi_g(v_{i,g}^0)=0$ and $\dist(v_{i,g}^0,\partial K_g)>0$, one has $v_{i,g}^0\in\operatorname{int}K_g$. By~\eqref{eq:prediction-safe-segment},
\begin{equation*}
\{v_{i,g}^0-s\delta v_{i,g}:0\le s\le1\}\subseteq\operatorname{int}K_g.
\end{equation*}
Consequently, $\Phi_g(v_{i,g}(m_i))=0=m_{i,g}$.

If $g\in S(b)$ and $D=A$, the $g$th factor of $A(b)$ is $\mathbb R^{d_g}$. First-order optimality in this block gives
\begin{equation}\label{eq:prediction-free-block-kkt}
 v_{i,g}(m_i)\in\mu m_{i,g}+\lambda_1\partial r_g(m_{i,g}),
\end{equation}
which is equivalent to~\eqref{eq:prediction-safe-proximal-equation}. The same argument applies when $g\in S(b)$, $D=C$, and $d_g\ge2$, because $\cl M_g(b_g)=\mathbb R^{d_g}$.

It remains to handle an active singleton factor of $C(b)$. For a closed convex set $K$ and $a\in K$, use the normal-cone convention
\begin{equation}\label{eq:prediction-normal-definition}
 N_K(a)=\{w:\langle w,u-a\rangle\le0\text{ for every }u\in K\}.
\end{equation}
Continuity of the finite convex penalty permits the subgradient sum rule with the domain indicator. Thus block optimality is
\begin{equation}\label{eq:prediction-constrained-block-kkt}
 v_{i,g}(m_i)\in\mu m_{i,g}+\lambda_1\partial r_g(m_{i,g})
                           +N_{C_g}(m_{i,g}),\qquad C_g=\cl M_g(b_g).
\end{equation}
Write $K_g=[\kappa_g^-,\kappa_g^+]$ and define
\begin{equation*}
\widehat a_{i,g}=\Phi_g(v_{i,g}(m_i)).
\end{equation*}
If $b_g>0$, stratum agreement and~\eqref{eq:prediction-safe-segment} give
\begin{equation*}
b_{i,g}>0,\qquad v_{i,g}^0>\kappa_g^+,\qquad
v_{i,g}(m_i)>\kappa_g^+,\qquad \widehat a_{i,g}>0.
\end{equation*}
Thus $\widehat a_{i,g}\in\operatorname{int}C_g$, where $C_g=[0,\infty)$, and $N_{C_g}(\widehat a_{i,g})=\{0\}$. In particular, $\widehat a_{i,g}$ solves the inclusion in~\eqref{eq:prediction-constrained-block-kkt} at the fixed input $v_{i,g}(m_i)$. Moreover,
\begin{equation}\label{eq:prediction-positive-normal}
 N_{[0,\infty)}(0)=(-\infty,0],\qquad
 K_g+N_{[0,\infty)}(0)=(-\infty,\kappa_g^+].
\end{equation}
If $m_{i,g}=0$, equations~\eqref{eq:prediction-constrained-block-kkt} and~\eqref{eq:prediction-positive-normal} would imply $v_{i,g}(m_i)\le\kappa_g^+$, contradicting $v_{i,g}(m_i)>\kappa_g^+$. Hence $m_{i,g}>0$.

If $b_g<0$, the corresponding inequalities are
\begin{equation*}
b_{i,g}<0,\qquad v_{i,g}^0<\kappa_g^-,\qquad
v_{i,g}(m_i)<\kappa_g^-,\qquad \widehat a_{i,g}<0.
\end{equation*}
Now $C_g=(-\infty,0]$ and $N_{C_g}(\widehat a_{i,g})=\{0\}$, so $\widehat a_{i,g}$ again solves~\eqref{eq:prediction-constrained-block-kkt}. In this case,
\begin{equation}\label{eq:prediction-negative-normal}
 N_{(-\infty,0]}(0)=[0,\infty),\qquad
 K_g+N_{(-\infty,0]}(0)=[\kappa_g^-,\infty),
\end{equation}
Thus $m_{i,g}=0$ would imply $v_{i,g}(m_i)\ge\kappa_g^-$, again a contradiction. Therefore $m_{i,g}<0$.

For either ray, $\mu I+\lambda_1\partial r_g+N_{C_g}$ is strongly monotone. Comparing its two solutions $\widehat a_{i,g}$ and $m_{i,g}$ at the same input gives
\begin{equation*}
\mu|\widehat a_{i,g}-m_{i,g}|^2
\le\langle v_{i,g}(m_i)-v_{i,g}(m_i),
\widehat a_{i,g}-m_{i,g}\rangle=0.
\end{equation*}
Hence $m_{i,g}=\widehat a_{i,g}=\Phi_g(v_{i,g}(m_i))$, proving~\eqref{eq:prediction-safe-proximal-equation}. This comparison fixes the scalar input $v_{i,g}(m_i)$; it does not require the observation loss to be separable.

\noindent\textbf{Step 3: define the remainders and bound their block sums.}
Define vectors with disjoint supports by
\begin{equation}\label{eq:prediction-remainder-definition}
 \begin{aligned}
 \zeta_{i,g}&=\begin{cases}
 \Delta_{i,g}-J_g(v_{i,g}^0)\delta v_{i,g},&g\notin D_i,\\
 0,&g\in D_i,
 \end{cases}\\
 \nu_{i,g}&=\begin{cases}
 0,&g\notin D_i,\\
 \Delta_{i,g}-J_g(v_{i,g}^0)\delta v_{i,g},&g\in D_i.
 \end{cases}
 \end{aligned}
\end{equation}
Then, by definition,
\begin{equation}\label{eq:prediction-proximal-decomposition}
 \Delta_i=J_i \delta v_i+\zeta_i+\nu_i.
\end{equation}
For $g\notin D_i$, equation~\eqref{eq:prediction-safe-proximal-equation} and Taylor expansion along the safe segment give
\begin{equation}\label{eq:prediction-smooth-remainder-integral}
 \begin{aligned}
 \Phi_g(v_{i,g}^0-\delta v_{i,g})
 &=\Phi_g(v_{i,g}^0)-J_g(v_{i,g}^0)\delta v_{i,g}+R_{i,g},\\
 \|R_{i,g}\|&\le(C_\Phi/2)\|\delta v_{i,g}\|^2,
 \qquad \zeta_{i,g}=-R_{i,g}.
 \end{aligned}
\end{equation}
On an interior inactive segment $R_{i,g}=0$; on an exterior segment it is the integral in~\eqref{eq:prox-taylor-integral} with direction $-\delta v_{i,g}$. By Lemma~\ref{lem:bridge}, summing, rather than taking a Euclidean norm of the blockwise bounds, gives
\begin{equation}\label{eq:zeta}
 \|\zeta_i\|_{2,1}
 \le C\sum_{g\notin D_i}\|\delta v_{i,g}\|^2
 \le C\|\delta v_i\|^2\le Ct_i^2k_i.
\end{equation}
For the exceptional blocks there is no Taylor claim. Lemma~\ref{lem:prox}, Lemma~\ref{lem:bridge} and Cauchy--Schwarz on a set with $|D_i|$ blocks gives
\begin{equation}\label{eq:nu}
 \begin{aligned}
 \|\nu_i\|_{2,1}
 &\le\sum_{g\in D_i}\|\Delta_{i,g}\|
       +\mu^{-1}\sum_{g\in D_i}\|\delta v_{i,g}\|\\
 &\le\sqrt{|D_i|}\{\|\Delta_i\|+\mu^{-1}\|\delta v_i\|\}
 \le Ct_i k_i.
 \end{aligned}
\end{equation}

\noindent\textbf{Step 4: expand the deleted loss gradient.}
For $j\ne i$, set $\xi_{ij}=x_j^\top b_i$ and $\delta_{ij}=x_j^\top\Delta_i$, so $x_j^\top m_i=\xi_{ij}-\delta_{ij}$. Define
\begin{equation}\label{eq:prediction-loss-remainder-definition}
 \begin{aligned}
 w_{ij}
 &=\ell'(y_j;\xi_{ij}-\delta_{ij})-\ell'(y_j;\xi_{ij})+\ell''(y_j;\xi_{ij})\delta_{ij}\\
 &=\delta_{ij}\int_0^1[\ell''(y_j;\xi_{ij})-\ell''(y_j;\xi_{ij}-s\delta_{ij})]\,ds.
 \end{aligned}
\end{equation}
Set $w_i=(w_{ij})_{j\ne i}$. The Lipschitz bound on $\ell''$ gives
\begin{equation}\label{eq:prediction-loss-remainder-bound}
 |w_{ij}|\le\frac{L_3}{2}|\delta_{ij}|^2,
 \qquad
 \|w_i\|_1\le\frac{L_3}{2}\|X_i\Delta_i\|^2
 \le C\|\Delta_i\|^2\le Ct_i^2k_i.
\end{equation}
Because $\delta v_i=\nabla L_i(m_i)-\nabla L_i(b_i)$, the exact vector identity is
\begin{equation}\label{eq:prediction-gradient-decomposition}
 \delta v_i=-H_i\Delta_i+X_i^\top w_i.
\end{equation}

\noindent\textbf{Step 5: invert in the correct order and pair with the held-out row.}
Substitute~\eqref{eq:prediction-gradient-decomposition} into~\eqref{eq:prediction-proximal-decomposition}:
\begin{equation}\label{eq:prediction-resolvent-equation}
 \begin{aligned}
 \Delta_i&=-J_iH_i\Delta_i+J_iX_i^\top w_i+\zeta_i+\nu_i,\\
 (I+J_iH_i)\Delta_i&=J_iX_i^\top w_i+\zeta_i+\nu_i,\\
 \Delta_i&=(I+J_iH_i)^{-1}J_iX_i^\top w_i
           +(I+J_iH_i)^{-1}(\zeta_i+\nu_i)\\
 &=B_iX_i^\top w_i+Q_i^\top(\zeta_i+\nu_i).
 \end{aligned}
\end{equation}
The final equality uses both $B_i=(I+J_iH_i)^{-1}J_i$ and $(I+J_iH_i)^{-1}=Q_i^\top$; it does not commute $H_i$ with $J_i$. Since $B_i=B_i^\top$, pairing gives
\begin{equation}\label{eq:scalarpairing}
 \begin{aligned}
 x_i^\top\Delta_i
 &=x_i^\top B_iX_i^\top w_i+x_i^\top Q_i^\top(\zeta_i+\nu_i)\\
 &=(X_iB_ix_i)^\top w_i+(Q_ix_i)^\top(\zeta_i+\nu_i).
 \end{aligned}
\end{equation}
For block vectors $a,z$, $|a^\top z|\le\sum_g\|a_g\|\|z_g\|\le\|a\|_{2,\infty}\|z\|_{2,1}$. Applying this and scalar $\ell_\infty,\ell_1$ duality to~\eqref{eq:scalarpairing} yields
\begin{equation}\label{eq:prediction-pairing-bound}
 \begin{aligned}
 |x_i^\top\Delta_i|
 &\le\|X_iB_ix_i\|_\infty\|w_i\|_1
   +\|Q_ix_i\|_{2,\infty}
      (\|\zeta_i\|_{2,1}+\|\nu_i\|_{2,1})\\
 &\le Ch_n(t_i^2k_i+t_i^2k_i+t_i k_i)
 \le Ch_nt_i k_i.
 \end{aligned}
\end{equation}
The last inequality uses $t_i\le1$. All remainder pairings are deterministic. In particular, their validity requires no independence between $x_i$ and $D_i,\zeta_i,\nu_i,w_i$.

\noindent\textbf{Step 6: the held-out loss and projector bounds.}
For the quadratic term in the held-out Taylor inequality, use a separate bound rather than squaring~\eqref{eq:prediction-pairing-bound}:
\begin{equation}\label{eq:prediction-square-bound}
 |x_i^\top\Delta_i|^2\le\|x_i\|^2\|\Delta_i\|^2
 \le K_x^2\|\Delta_i\|^2\le Ct_i^2k_i.
\end{equation}
Set $\Delta\ell_i=\ell(y_i;x_i^\top b_i)-\ell(y_i;x_i^\top m_i)$. Taylor's inequality at $x_i^\top b_i$, equation~\eqref{eq:prediction-pairing-bound}, and $S_ih_n=t_i/D_0$ give
\begin{equation}\label{eq:prediction-heldout-loss-detail}
 \begin{aligned}
 |\Delta\ell_i|
 &\le S_i|x_i^\top\Delta_i|+\frac{L_2}{2}|x_i^\top\Delta_i|^2\\
 &\le CS_ih_nt_i k_i+Ct_i^2k_i
 =C D_0^{-1}t_i^2k_i+Ct_i^2k_i
 \le Ct_i^2k_i.
 \end{aligned}
\end{equation}
Combine this with Lemma~\ref{lem:bridge} to prove~\eqref{eq:losslocal}.

The projectors have the block representations
\begin{equation*}
P=\diag\{\ind_{\{g\in S(b)\}}I_{d_g}\}_{g=1}^G,\qquad
P_i=\diag\{\ind_{\{g\in S(b_i)\}}I_{d_g}\}_{g=1}^G.
\end{equation*}
Since $S(b)\triangle S(b_i)\subseteq\mathcal S_i$,
\begin{equation}\label{eq:prediction-projector-detail}
 \|P-P_i\|_F^2
 =\sum_{g\in S(b)\triangle S(b_i)}d_g
 \le d_0|\mathcal S_i|\le Ck_i.
\end{equation}
A singleton sign change can belong to $\mathcal S_i$ without changing its coordinate projector; this is why only an inequality is asserted in the middle step. This proves~\eqref{eq:projectorlocal}.
\end{proof}

\subsection{Proof of Lemma~\ref{lem:antic}}

\begin{proof}
Fix the deletion index $i$ and a deterministic block $g$. The constants below are independent of both indices.

\noindent\textbf{Step 1: clamp the block and identify the conditioning sigma-field.}
Define the clamped estimator by
\begin{equation*}
\breve b_i^{(-g)}
=\operatorname*{arg\,min}_{\beta\in\mathbb R^p:\,\beta_g=0}F_i(\beta),
\end{equation*}
and set
\begin{equation}\label{eq:antic-conditioning-and-coefficients}
 \mathcal F_{i,-g}=\sigma(y_j,x_{j,-g}:j\ne i),\qquad
 a_j=-\ell'(y_j;x_{j,-g}^\top\breve b_{i,-g}^{(-g)}),\qquad
 T_{ig}=\sum_{j\ne i}a_jx_{j,g}.
\end{equation}
On this deterministic clamped domain, the loss uses no column from block $g$. The smaller-sigma-field construction in Lemma~\ref{lem:stability} makes $\breve b_{i,-g}^{(-g)}$ and every coefficient $a_j$ measurable with respect to $\mathcal F_{i,-g}$. The coefficients are therefore fixed when conditioning on this sigma-field.

\noindent\textbf{Step 2: transfer a boundary event to the clamped dual score.}
Fix a deterministic $r>0$. Throughout this step, work on $\{\dist(v_{i,g}^0,\partial K_g)\le r\}$, where $\dist$ denotes Euclidean distance of a point to a set in $\mathbb R^{d_g}$. Since $\partial K_g$ is compact and $\Phi_g=0$ there, proximal Lipschitz continuity gives
\begin{equation}\label{eq:antic-small-primal-block}
 \|b_{i,g}\|=\|\Phi_g(v_{i,g}^0)\|
 \le\mu^{-1}\dist(v_{i,g}^0,\partial K_g)\le r/\mu.
\end{equation}
Define the comparison point and the deleted score vector by
\begin{equation*}
\bar b_h=\begin{cases}
0,&h=g,\\
b_{i,h},&h\ne g,
\end{cases}
\qquad
\mathbf s_i(\beta)=(\ell'(y_j;x_j^\top\beta))_{j\ne i}.
\end{equation*}
Deleted optimality gives
\begin{equation*}
z_i=\frac{v_i^0-\mu b_i}{\lambda_1}\in\partial r(b_i).
\end{equation*}
By separability, $z_{i,h}\in\partial r_h(\bar b_h)$ for every $h\ne g$. Thus the complementary residual of the clamped objective at $b_{i,-g}$ is
\begin{equation}\label{eq:antic-clamped-residual}
 \begin{aligned}
 R_{i,-g}
 &=\nabla_{-g}L_i(\bar b)+\lambda_1z_{i,-g}+\mu b_{i,-g}\\
 &=\nabla_{-g}L_i(\bar b)-\nabla_{-g}L_i(b_i)\\
 &=X_{i,-g}^\top\bigl[\mathbf s_i(\bar b)-\mathbf s_i(b_i)\bigr],\\
 \|R_{i,-g}\|
 &\le L_2\|X_{i,-g}\|_{\op}\|X_{i,g}b_{i,g}\|\\
 &\le L_2\|X_{i,-g}\|_{\op}\|X_{i,g}\|_{\op}\|b_{i,g}\|.
 \end{aligned}
\end{equation}
Strong monotonicity in the free coordinates gives
\begin{equation}\label{eq:antic-clamped-distance}
 \|\breve b_{i,-g}^{(-g)}-b_{i,-g}\|
 \le\mu^{-1}\|R_{i,-g}\|
 \le\frac{L_2\|X_{i,-g}\|_{\op}\|X_{i,g}\|_{\op}}\mu\|b_{i,g}\|.
\end{equation}
On $\mathcal A_i$, both column-restricted design norms are at most $K_X$. Combining~\eqref{eq:antic-small-primal-block} and~\eqref{eq:antic-clamped-distance} yields
\begin{equation}\label{eq:antic-full-clamping-distance}
 \|\breve b_i^{(-g)}-b_i\|
 \le\|b_{i,g}\|+\|\breve b_{i,-g}^{(-g)}-b_{i,-g}\|
 \le Cr.
\end{equation}
Since $T_{ig}=-X_{i,g}^\top\mathbf s_i(\breve b_i^{(-g)})$ and $v_{i,g}^0=-X_{i,g}^\top\mathbf s_i(b_i)$,
\begin{equation}\label{eq:antic-dual-transfer-bound}
 \|T_{ig}-v_{i,g}^0\|
 \le L_2\|X_{i,g}\|_{\op}\|X_i\|_{\op}\|\breve b_i^{(-g)}-b_i\|
 \le Cr \quad\text{on }\mathcal A_i\cap\{\dist(v_{i,g}^0,\partial K_g)\le r\}.
\end{equation}
Since the distance function is 1-Lipschitz, there is a fixed constant $C_{\mathrm{tr}}\ge1$, independent of $i,g,r$, such that
\begin{equation}\label{eq:tubeinclusion}
 \{\dist(v_{i,g}^0,\partial K_g)\le r\}\cap\mathcal A_i
 \subseteq\{\dist(T_{ig},\partial K_g)\le C_{\mathrm{tr}}r\}\cap\mathcal A_i.
\end{equation}

\noindent\textbf{Step 3: use a measurable enlargement and bound coefficient energy.}
Define the enlarged event
\begin{equation}\label{eq:antic-measurable-enlargement}
 \mathcal A_{i,-g}
 =\{\|X_{i,-g}\|_{\op}\le K_X,\ v_-/2\le V_i\le2v_+\}.
\end{equation}
It is $\mathcal F_{i,-g}$-measurable, and $\mathcal A_i\subseteq\mathcal A_{i,-g}$. Apply the score sandwich (Lemma~\ref{lem:stability}) to the clamped fit, with free design $X_{i,-g}$. On $\mathcal A_{i,-g}$,
\begin{equation}\label{eq:antic-coefficient-energy}
 \begin{aligned}
 \sum_{j\ne i}a_j^2
 &\ge\left(\frac\mu{\mu+L_2K_X^2}\right)^2
           \sum_{j\ne i}\eta_j^2
 \ge c(n-1)v_-/2\ge cn,\\
 \sum_{j\ne i}a_j^2
 &\le\left(1+\frac{L_2K_X^2}\mu\right)^2
           \sum_{j\ne i}\eta_j^2
 \le C(n-1)2v_+\le Cn.
 \end{aligned}
\end{equation}
The last lower inequality holds for all sufficiently large $n$. In particular, at least one coefficient is nonzero on this event.

\noindent\textbf{Step 4: construct the conditional row densities almost surely.}
Let $\mathfrak m$ denote the fixed dominating measure for the responses. For nonempty complementary coordinates, define
\begin{equation*}
\mu_{g\mid-g}(z)=\Sigma_{g,-g}\Sigma_{-g,-g}^{-1}z.
\end{equation*}
The conditional law $x_{j,g}\mid x_{j,-g}=z$ is $N(\mu_{g\mid-g}(z),\Omega_g)$; denote its density by $\varphi_g(\,\cdot\mid z)$. The conditional covariance satisfies
\begin{equation}\label{eq:antic-gaussian-conditional-covariance}
 \Omega_g=\Sigma_{gg}-\Sigma_{g,-g}\Sigma_{-g,-g}^{-1}\Sigma_{-g,g},
 \qquad \Omega_g^{-1}=(\Sigma^{-1})_{gg}\preceq(p/c_\Sigma)I.
\end{equation}
If there are no complementary coordinates, use the unconditional Gaussian density and $\Omega_g=\Sigma$.

Define the conditional response density
\begin{equation}\label{eq:antic-response-normalizer}
 h_g(y\mid z)=\int_{\R^{d_g}}
   \varphi_g(u\mid z)q(y\mid z^\top\beta_{-g}^*+u^\top\beta_g^*)\,du.
\end{equation}
Tonelli and normalization of $q$ imply
\begin{equation}\label{eq:antic-normalizer-finiteness}
 \int_{\mathcal Y}h_g(y\mid z)\,\mathfrak m(dy)
 =\int_{\R^{d_g}}\varphi_g(u\mid z)
      \left[\int_{\mathcal Y}q(y\mid z^\top\beta_{-g}^*+u^\top\beta_g^*)\,\mathfrak m(dy)\right]du
 =1.
\end{equation}
Thus $h_g(y\mid z)$ is finite for $\mathfrak m$-almost every $y$. The common-support positivity assumption (Assumption~\ref{ass:likelihood}) makes the integrand strictly positive for almost every response in the specified support, and hence the normalizer is positive there. At the observed pair $(y_j,x_{j,-g})$, it is finite and positive almost surely. Bayes' formula therefore gives the conditional density
\begin{equation}\label{eq:antic-posterior-density}
 f_j(u)=
 \frac{\varphi_g(u\mid x_{j,-g})
       q(y_j\mid x_{j,-g}^\top\beta_{-g}^*+u^\top\beta_g^*)}
      {h_g(y_j\mid x_{j,-g})}.
\end{equation}
The original observation rows are independent, and the conditioning variables $(y_j,x_{j,-g})$ are rowwise. Therefore, for Borel sets $E_j\subseteq\mathbb R^{d_g}$,
\begin{equation*}
\Prob\left(\bigcap_{j\ne i}\{x_{j,g}\in E_j\}\,\middle|\,\mathcal F_{i,-g}\right)
=\prod_{j\ne i}\int_{E_j}f_j(u)\,du
\quad\text{almost surely}.
\end{equation*}
Thus the remaining blocks are conditionally independent with the densities in~\eqref{eq:antic-posterior-density}.

Taking the Hessian in $u$ removes the normalizing factor. The Gaussian part contributes $-\Omega_g^{-1}$, while the likelihood part contributes a scalar curvature times $\beta_g^*(\beta_g^*)^\top$. Consequently
\begin{equation}\label{eq:antic-posterior-curvature}
 \begin{aligned}
 -D^2\log f_j(u)
 &=\Omega_g^{-1}
  -\left[\partial_s^2\log q(y_j\mid s)\right]_{s=x_{j,-g}^\top\beta_{-g}^*+u^\top\beta_g^*}
                       \beta_g^*(\beta_g^*)^\top\\
 &\preceq\Omega_g^{-1}+C_q\beta_g^*(\beta_g^*)^\top\\
 &\preceq\left(p/c_\Sigma+C_q\|\beta_g^*\|^2\right)I
 \preceq CpI.
 \end{aligned}
\end{equation}
The final step uses $\|\beta_g^*\|^2\le\|\beta^*\|^2\le B^2p$ by Assumption~\ref{ass:design}. No upper bound on the likelihood log-curvature and no log-concavity assumption have been used.

\noindent\textbf{Step 5: bound the conditional density of the clamped dual sum.}
Condition on $\mathcal F_{i,-g}$ and define
\begin{equation*}
\mathcal J_{ig}=\{j\ne i:a_j\ne0\},\qquad
A_{ig}^{(2)}=\sum_{j\ne i}a_j^2.
\end{equation*}
On $\mathcal A_{i,-g}$, equation~\eqref{eq:antic-coefficient-energy} gives $\mathcal J_{ig}\ne\varnothing$ and $A_{ig}^{(2)}>0$. For $j\in\mathcal J_{ig}$, the conditional density of $a_jx_{j,g}$ is
\begin{equation*}
\widetilde f_j(u)=|a_j|^{-d_g}f_j(u/a_j),\qquad
-D^2\log\widetilde f_j(u)
=-a_j^{-2}D^2\log f_j(u/a_j)
\preceq\frac{Cp}{a_j^2}I.
\end{equation*}
The formula also applies when $a_j<0$. Choose
\begin{equation}\label{eq:antic-convolution-weights}
\begin{gathered}
\gamma_j=Cp/a_j^2,\qquad
\theta_j=a_j^2/A_{ig}^{(2)},\qquad j\in\mathcal J_{ig},\\
\sum_{j\in\mathcal J_{ig}}\theta_j=1,\qquad
\sum_{j\in\mathcal J_{ig}}\gamma_j\theta_j^2=\frac{Cp}{A_{ig}^{(2)}}.
\end{gathered}
\end{equation}
Lemma~\ref{lem:density} and~\eqref{eq:antic-coefficient-energy} now give, on $\mathcal A_{i,-g}$,
\begin{equation}\label{eq:antic-clamped-density-bound}
 \|f_{T_{ig}\mid\mathcal F_{i,-g}}\|_\infty
 \le\left(\frac{Cp}{2\pi A_{ig}^{(2)}}\right)^{d_g/2}
 \le C.
\end{equation}
The last constant is uniform because $p/n$ and $d_g\le d_0$ are uniformly bounded.

\noindent\textbf{Step 6: apply the tube-volume bound before averaging over blocks.}
Since $K_g=\lambda_1\partial r_g(0)$, the uniform ball bounds on $\partial r_g(0)$, and the positive lower and upper bounds on $\lambda_1$, provide constants $0<r_0\le R_0<\infty$, independent of $g,n,p$, such that
\begin{equation*}
B(0,r_0)\subseteq K_g\subseteq B(0,R_0).
\end{equation*}
Here the balls lie in $\R^{d_g}$, while their radii are uniform across blocks. Set
\begin{equation*}
r_*=\frac{r_0}{2C_{\mathrm{tr}}}>0.
\end{equation*}
For $0<r\le r_*$, we have $C_{\mathrm{tr}}r\le r_0/2<r_0$, so Lemma~\ref{lem:tube} applies with $\varepsilon=C_{\mathrm{tr}}r$. Hence Lemma~\ref{lem:tube} and~\eqref{eq:antic-clamped-density-bound} give
\begin{equation}\label{eq:antic-conditional-tube-probability}
 \Prob\{\dist(T_{ig},\partial K_g)\le C_{\mathrm{tr}}r\mid\mathcal F_{i,-g}\}
 \le\|f_{T_{ig}\mid\mathcal F_{i,-g}}\|_\infty
       \vol\{u:\dist(u,\partial K_g)\le C_{\mathrm{tr}}r\}
 \le Cr
\end{equation}
on $\mathcal A_{i,-g}$. Here $\operatorname{vol}$ denotes the Lebesgue measure in $\R^{d_g}$. Since $\mathcal A_i$ need not be $\mathcal F_{i,-g}$-measurable, we first use~\eqref{eq:tubeinclusion} and the enlargement $\mathcal A_i\subseteq\mathcal A_{i,-g}$, and then condition:
\begin{equation}\label{eq:antic-unconditional-block-bound}
 \begin{aligned}
 &\Prob\{\dist(v_{i,g}^0,\partial K_g)\le r,\mathcal A_i\}\\
 &\quad\le\E\left[\ind_{\mathcal A_{i,-g}}
                   \ind_{\{\dist(T_{ig},\partial K_g)\le C_{\mathrm{tr}}r\}}\right]\\
 &\quad=\E\left[\ind_{\mathcal A_{i,-g}}
       \Prob\{\dist(T_{ig},\partial K_g)\le C_{\mathrm{tr}}r\mid\mathcal F_{i,-g}\}\right]
 \le Cr.
 \end{aligned}
\end{equation}
Finally, $G\le p$ because every block has at least one coordinate. Linearity of expectation yields
\begin{equation}\label{eq:antic-fixed-count-sum}
 \E[N_i(r)\ind_{\mathcal A_i}]
 =\sum_{g=1}^G\Prob\{\dist(v_{i,g}^0,\partial K_g)\le r,\mathcal A_i\}
 \le CGr\le Cpr.
\end{equation}
If $r>r_*$, use $N_i(r)\le p\le r_*^{-1}pr$ instead. After increasing the constant,~\eqref{eq:fixedwidth} holds for every deterministic width $r>0$.
\end{proof}

\subsection{Proof of Lemma~\ref{lem:tail}}

\begin{proof}
\noindent\textbf{Step 1: a conditional envelope for the held-out score.}
The scalar score is $L_2$-Lipschitz, so
\begin{equation}\label{eq:tail-score-decomposition}
 S_i=1+|\ell'(y_i;x_i^\top b_i)|
 \le1+|\eta_i|+L_2|x_i^\top b_i|.
\end{equation}
The held-out pair $(y_i,x_i)$ is independent of $\mathcal D_{-i}=\sigma((y_j,x_j):j\neq i)$. Consequently $\eta_i$ has its original marginal law after conditioning on $\mathcal D_{-i}$, and the conditional law of the predictor is
\begin{equation}\label{eq:tail-predictor-variance}
 x_i^\top b_i\mid\mathcal D_{-i}\sim N(0,\sigma_i^2),
 \qquad \sigma_i^2=b_i^\top\Sigma b_i
 \le\frac{C_\Sigma}{p}\|b_i\|^2\le C
 \quad\text{on }\mathcal A_i.
\end{equation}
The last step uses~\eqref{eq:gaussian-fit-size} and $n/p\le\delta_+$. For $u\ge4$, equation~\eqref{eq:tail-score-decomposition} gives
\begin{equation*}
\{|\eta_i|\le u/4\}\cap\{L_2|x_i^\top b_i|\le u/2\}
\subseteq\{S_i\le1+3u/4\}\subseteq\{S_i\le u\}.
\end{equation*}
Taking complements and applying the union bound, we obtain, on $\mathcal A_i$,
\begin{equation}\label{eq:tail-conditional-union}
 \begin{aligned}
 \Prob(S_i>u\mid\mathcal D_{-i})
 &\le\Prob(|\eta_i|>u/4\mid\mathcal D_{-i})
       +\Prob(L_2|x_i^\top b_i|>u/2\mid\mathcal D_{-i})\\
 &\le\Prob(|\eta_1|>u/4)+2e^{-cu^2}.
 \end{aligned}
\end{equation}
When $L_2=0$, the predictor event is empty and the bound remains valid. Choose a fixed $u_0\ge4$ and define
\begin{equation*}
q_n(u)=\begin{cases}
1,&1\le u\le u_0,\\
\Prob(|\eta_1|>u/4)+2e^{-cu^2},&u>u_0.
\end{cases}
\end{equation*}
The subscript permits the marginal score law to vary with $n$ within the uniform assumptions. This proves the first assertion of~\eqref{eq:conditionaltail}. Only a union bound is needed; the score and predictor within the held-out row may be dependent.

\noindent\textbf{Step 2: conditional third-moment control.}
For nonnegative $a,b,c$, convexity of the cube gives $(a+b+c)^3\le9(a^3+b^3+c^3)$. Therefore
\begin{equation}\label{eq:tail-conditional-third-detail}
 \begin{aligned}
 \E[S_i^3\mid\mathcal D_{-i}]
 &\le9\{1+\E|\eta_i|^3+L_2^3\E[|x_i^\top b_i|^3\mid\mathcal D_{-i}]\}\\
 &=9\{1+\E|\eta_i|^3+L_2^3\sigma_i^3\E|Z|^3\}
 \le C\quad\text{on }\mathcal A_i,
 \qquad Z\sim N(0,1).
 \end{aligned}
\end{equation}
The Gaussian moment is a finite numerical constant. No fourth moment of the null score is used.

\noindent\textbf{Step 3: keep the actual score tail in the dyadic summation.}
For a nonnegative finite random variable $T$ and $c_0>0$, the geometric-series identity gives the pointwise bound
\begin{equation}\label{eq:tail-geometric-pointwise}
 \sum_{k\ge0}2^{3k}\ind_{\{T>c_02^k\}}
 =\sum_{\{k:2^k<T/c_0\}}8^k
 \le\frac87\left(\frac{T}{c_0}\right)^3.
\end{equation}
If the index set is empty, the left side is zero and the inequality still holds. Taking expectations and using Tonelli yields
\begin{equation}\label{eq:tail-geometric-expectation}
 \sum_{k\ge0}2^{3k}\Prob(T>c_02^k)
 \le\frac{8}{7c_0^3}\E T^3.
\end{equation}
Apply this to $T=|\eta_1|$ and $c_0=1/4$. The score part of the dyadic sum is at most a fixed multiple of $K_3$. The finitely many terms with $2^k\le u_0$ contribute a fixed constant. For the Gaussian part,
\begin{equation}\label{eq:tail-gaussian-summability}
 \sum_{k\ge0}2^{3k}e^{-c2^{2k}}<\infty,
\end{equation}
because the ratio of consecutive summands is $8e^{-3c4^k}$ and is eventually at most $1/2$. Combining the three parts gives
\begin{equation*}
\begin{aligned}
\sum_{k\ge0}2^{3k}q_n(2^k)
&\le\sum_{\{k:\,2^k\le u_0\}}2^{3k}
+\frac{8\cdot4^3}{7}\E|\eta_1|^3
+2\sum_{k\ge0}2^{3k}e^{-c2^{2k}}\\
&\le C.
\end{aligned}
\end{equation*}
The constant is uniform in $n$, which proves~\eqref{eq:tailsum}.

By contrast, substituting only the pointwise Markov estimate $q_n(2^k)\le C2^{-3k}$ into the critical sum would give
\begin{equation}\label{eq:tail-markov-borderline}
 \sum_{k\ge0}2^{3k}\,C2^{-3k}=\sum_{k\ge0}C=\infty.
\end{equation}
The argument therefore uses the integrated content of the finite third moment, rather than just its pointwise Markov consequence.
\end{proof}

\subsection{Proof of Lemma~\ref{lem:dyadic}}

\begin{proof}
Fix $a_0>0$ and $q\in\{0,1,2\}$. Define the dyadic events and deterministic widths by
\begin{equation*}
\mathcal H_{ik}=\{2^k\le S_i<2^{k+1}\},\qquad
\varepsilon_k=a_02^{k+1}h_n,\qquad k\ge0.
\end{equation*}
Since $S_i\ge1$, exactly one of the events $\mathcal H_{ik}$ holds. On $\mathcal H_{ik}$, monotonicity of $N_i$ gives
\begin{equation*}
S_i^q\le2^{q(k+1)},\qquad
N_i(a_0S_ih_n)\le N_i(\varepsilon_k).
\end{equation*}
Consequently,
\begin{equation}\label{eq:dyadic-pointwise-decomposition}
 S_i^qN_i(a_0S_ih_n)\ind_{\mathcal A_i}
 \le\sum_{k\ge0}2^{q(k+1)}N_i(\varepsilon_k)
                \ind_{\mathcal A_i}\ind_{\mathcal H_{ik}}.
\end{equation}
Taking expectations, using Tonelli, and dropping the upper endpoint restriction yields
\begin{equation}\label{eq:dyadic-expectation-decomposition}
 \E[S_i^qN_i(a_0S_ih_n)\ind_{\mathcal A_i}]
 \le\sum_{k\ge0}2^{q(k+1)}
  \E[N_i(\varepsilon_k)\ind_{\mathcal A_i}\ind_{\{S_i\ge2^k\}}].
\end{equation}
For $k=0$, bound the last indicator by one and apply~\eqref{eq:fixedwidth}:
\begin{equation}\label{eq:dyadic-first-bin}
 2^q\E[N_i(2a_0h_n)\ind_{\mathcal A_i}]
 \le Cph_n.
\end{equation}
For $k\ge1$, the exact inclusion
\begin{equation}\label{eq:dyadic-atoms-inclusion}
 \{S_i\ge2^k\}\subseteq\{S_i>2^{k-1}\}
\end{equation}
avoids any need to assume that $S_i$ has a continuous distribution. In particular, atoms at $2^k$ cause no problem.

For each fixed $k$, the width $\varepsilon_k$ is deterministic, and both its count and $\mathcal A_i$ are $\mathcal D_{-i}$-measurable. The tower property and~\eqref{eq:conditionaltail} give
\begin{equation}\label{eq:dyadic-conditioning-step}
 \begin{aligned}
 &\E[N_i(\varepsilon_k)\ind_{\mathcal A_i}\ind_{\{S_i\ge2^k\}}]\\
 &\quad=\E\left[N_i(\varepsilon_k)\ind_{\mathcal A_i}
                 \Prob(S_i\ge2^k\mid\mathcal D_{-i})\right]\\
 &\quad\le q_n(2^{k-1})
          \E[N_i(\varepsilon_k)\ind_{\mathcal A_i}]\\
 &\quad\le Cpa_02^{k+1}h_n\,q_n(2^{k-1}).
 \end{aligned}
\end{equation}
The middle step uses the conditional tail bound only where $\mathcal A_i$ holds. It does not factor the score event from the count by independence.

Substituting~\eqref{eq:dyadic-first-bin} and~\eqref{eq:dyadic-conditioning-step} into~\eqref{eq:dyadic-expectation-decomposition} gives
\begin{equation}\label{eq:dyadic-final-sum}
 \begin{aligned}
 \E[S_i^qN_i(a_0S_ih_n)\ind_{\mathcal A_i}]
 &\le Cph_n\left[1+\sum_{k\ge1}2^{(q+1)k}q_n(2^{k-1})\right]\\
 &\le Cph_n\left[1+8\sum_{m\ge0}2^{3m}q_n(2^m)\right]
 \le Cph_n.
 \end{aligned}
\end{equation}
The second inequality uses $q+1\le3$ and $m=k-1$; the last uses~\eqref{eq:tailsum}. The fixed-width estimate is valid even for very large dyadic widths, so no truncation of this sum is needed.
\end{proof}

\subsection{Proof of Lemma~\ref{lem:complements}}

\begin{proof}
Fix $D\in\{A,C\}$ and write $m_i=m_i^D$. Recall the quantities
\begin{equation*}
\begin{gathered}
\Delta_i=b_i-m_i,\qquad
\Delta\ell_i=\ell(y_i;x_i^\top b_i)-\ell(y_i;x_i^\top m_i),\\
W_i=|\Delta\ell_i|+\|\Delta_i\|^2+|x_i^\top\Delta_i|,
\qquad T_n=(D_0h_n)^{-1}.
\end{gathered}
\end{equation*}
For this proof, also set $R_i=\|x_i\|$ and $\xi_i=x_i^\top b_i$. These two scalar quantities may be dependent. Lemma~\ref{lem:global}, \eqref{eq:tail-score-decomposition}, and $S_i\ge1$ give the deterministic envelope
\begin{equation}\label{eq:complements-global-envelope}
 W_i\le CS_i^2(R_i^2+R_i^4),\qquad
 S_i^2\le C(1+\eta_i^2+L_2^2\xi_i^2).
\end{equation}
Expanding the product gives
\begin{equation}\label{eq:complete bound for W}
    W_i\le C\sum_{k\in\{2,4\}}(R_i^k+\eta_i^2R_i^k+L_2^2\xi_i^2R_i^k).
\end{equation}
We treat the three exceptional events separately, because different factors must be controlled on each.

\noindent\textbf{Step 1: an integrable conditional envelope for arbitrary deleted data.}
For $k\in\{2,4\}$, the held-out row $x_i^\top$ is independent of $\mathcal D_{-i}=\sigma((y_j,x_j):j\neq i)$. H\"older with exponents $3/2$ and $3$ therefore gives
\begin{equation}\label{eq:complements-null-score-product}
 \begin{aligned}
 \E[\eta_i^2R_i^k\mid\mathcal D_{-i}]
 &=\E[\eta_i^2R_i^k]\\
 &\le(\E|\eta_i|^3)^{2/3}(\E R_i^{3k})^{1/3}\le C.
 \end{aligned}
\end{equation}
The equality uses independence from the deleted data, not independence of $\eta_i$ and $R_i$. The row moments are uniformly bounded by~\eqref{eq:gaussian-row-moments}.

Conditional on $\mathcal D_{-i}$, $\xi_i$ is centered Gaussian with variance $\sigma_i^2=b_i^\top\Sigma b_i\le C_\Sigma\|b_i\|^2/p$. Cauchy--Schwarz gives
\begin{equation}\label{eq:complements-predictor-product}
 \begin{aligned}
 \E[\xi_i^2R_i^k\mid\mathcal D_{-i}]
 &\le\left(\E[\xi_i^4\mid\mathcal D_{-i}]\right)^{1/2}
       (\E R_i^{2k})^{1/2}\\
 &=\sqrt3\,\sigma_i^2(\E R_i^{2k})^{1/2}
 \le C\frac{\|b_i\|^2}{p}.
 \end{aligned}
\end{equation}
Combining~\eqref{eq:complements-global-envelope}--\eqref{eq:complements-predictor-product} and the unweighted row moments \eqref{eq:gaussian-row-moments} proves
\begin{equation}\label{eq:globalconditional}
 \E[W_i\mid\mathcal D_{-i}]\le C\left(1+\frac{\|b_i\|^2}{p}\right).
\end{equation}
Because $\mathcal A_i^c$ is $\mathcal D_{-i}$-measurable, the tower property and~\eqref{eq:weightedbad} yield
\begin{equation}\label{eq:complements-bad-deleted-detail}
 \begin{aligned}
 \E[W_i\ind_{\mathcal A_i^c}]
 &=\E\left[\ind_{\mathcal A_i^c}\E[W_i\mid\mathcal D_{-i}]\right]\\
 &\le C\E\left[\left(1+\frac{\|b_i\|^2}{p}\right)\ind_{\mathcal A_i^c}\right]
 \le Cn^{-1/2}.
 \end{aligned}
\end{equation}
This is~\eqref{eq:compA}.

\noindent\textbf{Step 2: separate linear-response failure from row-norm failure.}
Define
\begin{equation}\label{eq:complements-gaussian-failure-events}
 \begin{aligned}
 B_i^{\mathrm{row}}&=\{R_i>K_x\},\\
 B_i^{\mathrm{lin}}&=
 \{\|Q_ix_i\|_{2,\infty}+\|X_iB_ix_i\|_\infty>C_Gh_n\}.
 \end{aligned}
\end{equation}
Then $\mathcal G_i^c=B_i^{\mathrm{row}}\cup B_i^{\mathrm{lin}}$. On $(B_i^{\mathrm{row}})^c$, equation~\eqref{eq:complements-global-envelope} simplifies to $W_i\le CS_i^2$. On $\mathcal A_i$, conditional H\"older and the conditional bounds for $S_i^3$ and $B_i^{\mathrm{lin}}$ (Lemma~\ref{lem:tail} and Lemma~\ref{lem:gaussian}) give
\begin{equation}\label{eq:complements-linear-failure-detail}
 \begin{aligned}
 &\E[W_i\ind_{(B_i^{\mathrm{row}})^c\cap B_i^{\mathrm{lin}}}\mid\mathcal D_{-i}]\\
 &\quad\le C\E[S_i^2\ind_{B_i^{\mathrm{lin}}}\mid\mathcal D_{-i}]\\
 &\quad\le C\left(\E[S_i^3\mid\mathcal D_{-i}]\right)^{2/3}
             \Prob(B_i^{\mathrm{lin}}\mid\mathcal D_{-i})^{1/3}\\
 &\quad\le C(n^{-20})^{1/3}\le Cn^{-4}.
 \end{aligned}
\end{equation}
The exponents are $3/2$ on $S_i^2$ and $3$ on the indicator. No conditional independence of the score and failure event is required.

On $B_i^{\mathrm{row}}$, use the three terms in~\eqref{eq:complete bound for W} separately. For each $k\in\{2,4\}$, the restricted row-moment bound~\eqref{eq:gaussian-row-moments} gives
\begin{equation}\label{eq:complements-row-constant-term}
 \E[R_i^k\ind_{B_i^{\mathrm{row}}}\mid\mathcal D_{-i}]
 =\E[R_i^k\ind_{B_i^{\mathrm{row}}}]\le Ce^{-cn}.
\end{equation}
For the null-score part, apply H\"older to the score and the \emph{restricted} row moment \eqref{eq:gaussian-row-moments}:
\begin{equation}\label{eq:complements-row-score-term}
 \begin{aligned}
 \E[\eta_i^2R_i^k\ind_{B_i^{\mathrm{row}}}\mid\mathcal D_{-i}]
 &\le(\E|\eta_i|^3)^{2/3}
      \left(\E[R_i^{3k}\ind_{B_i^{\mathrm{row}}}]\right)^{1/3}\\
 &\le Ce^{-cn}.
 \end{aligned}
\end{equation}
For the predictor part, use conditional Cauchy--Schwarz:
\begin{equation}\label{eq:complements-row-predictor-term}
 \begin{aligned}
 \E[\xi_i^2R_i^k\ind_{B_i^{\mathrm{row}}}\mid\mathcal D_{-i}]
 &\le\left(\E[\xi_i^4\mid\mathcal D_{-i}]\right)^{1/2}
      \left(\E[R_i^{2k}\ind_{B_i^{\mathrm{row}}}]\right)^{1/2}\\
 &\le Ce^{-cn}\quad\text{on }\mathcal A_i.
 \end{aligned}
\end{equation}
Here the conditional Gaussian fourth moment is uniformly bounded on $\mathcal A_i$ by~\eqref{eq:tail-predictor-variance}. Combining these three terms for $k=2,4$ shows
\begin{equation}\label{eq:complements-row-failure-detail}
 \E[W_i\ind_{B_i^{\mathrm{row}}}\mid\mathcal D_{-i}]
 \le Ce^{-cn}\quad\text{on }\mathcal A_i.
\end{equation}
Since $\mathcal A_i$ is $\mathcal D_{-i}$-measurable and
\begin{equation*}
\mathcal G_i^c
=B_i^{\mathrm{row}}\mathbin{\dot\cup}
\bigl((B_i^{\mathrm{row}})^c\cap B_i^{\mathrm{lin}}\bigr),
\end{equation*}
the tower property and the two conditional bounds give
\begin{equation*}
\begin{aligned}
\E[W_i\ind_{\mathcal A_i\cap\mathcal G_i^c}]
&=\E\left[\ind_{\mathcal A_i}
\E[W_i\ind_{\mathcal G_i^c}\mid\mathcal D_{-i}]\right]\\
&\le Cn^{-4}+Ce^{-cn}.
\end{aligned}
\end{equation*}
This proves~\eqref{eq:compG} without requiring a moment bound for the product $S_i^3R_i^6$.

\noindent\textbf{Step 3: a score larger than the local cutoff.}
On $\mathcal G_i$, the row norm is bounded, so $W_i\le CS_i^2$. Since $S_i^2\ind_{\{S_i>T_n\}}\le T_n^{-1}S_i^3$, the conditional third moment gives
\begin{equation}\label{eq:complements-large-score-detail}
 \begin{aligned}
 &\E[W_i\ind_{\mathcal A_i\cap\mathcal G_i\cap\{S_i>T_n\}}]\\
 &\quad\le C\E[S_i^2\ind_{\mathcal A_i}\ind_{\{S_i>T_n\}}]\\
 &\quad\le CT_n^{-1}\E[S_i^3\ind_{\mathcal A_i}]\\
 &\quad=CT_n^{-1}\E[\ind_{\mathcal A_i}\E(S_i^3\mid\mathcal D_{-i})]
 \le CT_n^{-1}=CD_0h_n.
 \end{aligned}
\end{equation}
This proves~\eqref{eq:compS}. For the prediction term alone, the global bound on $\mathcal G_i$ is $|x_i^\top\Delta_i|\le CS_i$. Hence its contribution on the same large-score event is at most
\begin{equation}\label{eq:complements-large-score-prediction}
 C\E[S_i\ind_{\mathcal A_i}\ind_{\{S_i>T_n\}}]
 \le CT_n^{-2}\E[S_i^3\ind_{\mathcal A_i}]\le Ch_n^2.
\end{equation}
This sharper prediction-tail estimate is not needed for the stated common rate.
\end{proof}

\end{document}